\documentclass[1 leqno,11pt,psfig]{amsart}
\usepackage{amssymb, amsmath,amsmath,latexsym,amssymb,amsfonts,amsbsy, amsthm,mathtools,graphicx,color}
\usepackage{graphicx,epsfig}
\usepackage{graphicx}
\usepackage{amssymb}
\usepackage{graphicx}
\usepackage{graphicx,epsfig}
\usepackage{amsmath}
\usepackage{mathrsfs}
\usepackage{amssymb}
\usepackage{longtable}

\usepackage{amssymb,mathrsfs,amsfonts,amsmath,amsbsy,tikz}\usepackage{fontenc}\usepackage{textcomp}

\numberwithin{equation}{section}

\allowdisplaybreaks

\newcommand{\beq}{\begin{equation}}
\newcommand{\eeq}{\end{equation}}
\newcommand{\ben}{\begin{eqnarray}}
\newcommand{\een}{\end{eqnarray}}
\newcommand{\beno}{\begin{eqnarray*}}
\newcommand{\eeno}{\end{eqnarray*}}

\newtheorem{theorem}{Theorem}[section]
\newtheorem{definition}{Definition}[section]
\newtheorem{lemma}{Lemma}[section]
\newtheorem{hypothesis}{Hypothesis}[section]
\newtheorem{proposition}{Proposition}[section]
\newtheorem{corollary}{Corollary}[section]
\newtheorem{remark}{Remark}[section]

\begin{document}
\title[Barotropic instability]{Barotropic instability of shear flows}

\author{Zhiwu Lin}
\address{School of Mathematics, Georgia Institute of Technology, 30332, Atlanta, GA, USA}
\email{zlin@math.gatech.edu}

\author{Jincheng Yang}
\address{Department of Mathematics, The University of Texas at Austin, 78712, Austin, TX, USA}
\email{jcyang@math.utexas.edu}

\author{Hao Zhu}
\address{Chern Institute of Mathematics, Nankai University, 300071, Tianjin, P. R. China}
\email{haozhu@nankai.edu.cn}

\date{\today}

\maketitle

\begin{abstract}
We consider barotropic instability of shear flows for incompressible fluids
with Coriolis effects. For a class of shear flows, we develop a new method to
find the sharp stability conditions. We study the flow with Sinus profile in
details and obtain the sharp stability boundary in the whole parameter space,
which corrects previous results in the fluid literature. Our new results are
confirmed by more accurate numerical computation. The addition of the Coriolis
force is found to bring fundamental changes to the stability of shear flows.
Moreover, we study dynamical behaviors near the shear flows, including the
bifurcation of nontrivial traveling wave solutions and the linear inviscid
damping. The first ingredient of our proof is a careful classification of the
neutral modes. The second one is to write the linearized fluid equation in a
Hamiltonian form and then use an instability index theory for general
Hamiltonian PDEs. The last one is to study the singular and non-resonant
neutral modes using Sturm-Liouville theory and hypergeometric functions.
\end{abstract}

{\bf \it Keywords}:\ shear flow; barotropic instability; fluid dynamics;  Hamiltonian structure.

{2010 {\bf \it Mathematics Subject Classification}}: 76E05; 76E09.

\section{Introduction}

When studying the large-scale motion of ocean and atmosphere, the rotation of
the earth may affect the dynamics of the fluids significantly and therefore,
Coriolis effects must be taken into account (\cite{Pedlosky1987}). In this
paper, we study stability and instability of shear flows under Coriolis
forces. We consider the fluids in a strip or channel denoted by
\[
D=\{(x,y)\ |\ y\in\lbrack y_{1},y_{2}]\},
\]
where $x$ is periodic. The fluid motion is modeled by the two-dimensional
inviscid incompressible Euler equation with rotation
\begin{equation}
\partial_{t}\vec{u}+(\vec{u}\cdot\nabla)\vec{u}=-\nabla P-\beta yJ\vec
{u},\;\;t\times(x,y)\in\lbrack0,+\infty)\times D, \label{11}%
\end{equation}
where $\vec{u}=(u_{1},u_{2})$ is the fluid velocity, $P$ is the pressure,
\[
J=%
\begin{pmatrix}
0 & -1\\
1 & 0
\end{pmatrix}
\]
is the rotation matrix, and $\beta\ $is the Rossby number. Here, the term
$-\beta yJ\vec{u}$ denotes the Coriolis force under the beta-plane
approximation. We assume the incompressible condition $\nabla\cdot\vec{u}=0$
and the non-permeable boundary condition
\begin{equation}
u_{2}=0,\ \text{on }\partial D=\left\{  y=y_{1},y_{2}\right\}  .
\label{bc-Euler}%
\end{equation}
The vorticity $\omega$ is defined as $\omega:=\operatorname{curl}\vec
{u}=\partial_{x}u_{2}-\partial_{y}u_{1}$, and the stream function $\psi$ is
introduced such that $\vec{u}=\nabla^{\perp}\psi=(\partial_{y}\psi
,-\partial_{x}\psi)$. The vorticity form of (\ref{11}) is
\begin{equation}
\partial_{t}\omega+(\vec{u}\cdot\nabla)\omega+\beta u_{2}=0,
\label{vorticity-eqn}%
\end{equation}
which is also called the quasi-geostrophic equation in geophysical fluids
(\cite{Pedlosky1987}). Consider a shear flow $\vec{u}_{0}=(U(y),0)$, $U\in
C^{2}([y_{1},y_{2}])$, which is a steady solution of (\ref{vorticity-eqn}).
The linearized equation of (\ref{vorticity-eqn}) around the shear flow
$\vec{u}_{0}$ is
\begin{equation}
\partial_{t}\omega+U\partial_{x}\omega-(\beta-U^{\prime\prime})\partial
_{x}\psi=0. \label{13}%
\end{equation}
To study the linear instability, it suffices to consider the normal mode
solution $\psi(x,y,t)=\phi(y)e^{i\alpha(x-ct)}$, where $\alpha>0$ is the wave
number in the $x$-direction and $c=c_{r}+ic_{i}$ is the complex wave speed.
Then (\ref{13}) is reduced to the Rayleigh-Kuo equation
\begin{equation}
-\phi^{\prime\prime}+\alpha^{2}\phi-{\frac{\beta-U^{\prime\prime}}{U-c}}%
\phi=0, \label{15}%
\end{equation}
with the boundary conditions
\begin{equation}
\phi(y_{1})=\phi(y_{2})=0. \label{16}%
\end{equation}
When $\beta=0$, (\ref{15}) becomes the classical Rayleigh Equation
(\cite{Rayleigh1880}), which has been studied extensively (cf.
\cite{Drazin-Reid1981,Howard1964,CLin1955,Lin2003,Lin2005,Tollmien1935}).

The shear flow $U$ is linear unstable if there exists a nontrivial solution to
(\ref{15})--(\ref{16}) with $\operatorname{Im}c>0$. This so called barotropic
instability is important for the dynamics of atmosphere and oceans. It has
been a classical problem in geophysical fluid dynamics
(\cite{Kuo1949,Kuo1974,Pedlosky1987}) since 1940s. Rossby first recognized the
nature of barotropic instability and derived the linearized vorticity equation
in \cite{Rossby1939}. Later, Kuo formulated the equation (\ref{15}%
)--(\ref{16}), and did some early studies in \cite{Kuo1949}. In particular, he
gave a necessary condition for instability that $\beta-U^{\prime\prime}$ must
change sign in the domain $[y_{1},y_{2}]$, which generalized the classical
Rayleigh criterion (\cite{Rayleigh1880}) for $\beta=0$. In \cite{Pedlosky1963}%
, Pedlosky showed that any unstable wave speed $c=c_{r}+ic_{i}$ $\left(
c_{i}>0\right)  $ must lie in the following semicircle
\begin{equation}
\left(  c_{r}-(U_{\min}+U_{\max})/{2}\right)  ^{2}+c_{i}^{2}\leq\left(
({U_{\max}-U_{\min})/{2}}+{|\beta|}/{2\alpha^{2}}\right)  ^{2},
\label{semi-circle}%
\end{equation}
which is a generalization of Howard's semicircle theorem \cite{Howard1961} for
$\beta=0$. Here, $U_{\min}=\min U$ and $U_{\max}=\max U$. Additionally, the
following characterization for the unstable wave speeds is given in
\cite{Kuo1949, Miles1964, Pedlosky1964}.

\begin{lemma}
\label{lemma unstable modes c range}If $\beta>0$, then there are no nontrivial
solutions of (\ref{15})--(\ref{16}) for $c_{r}>U_{\max}$; if $\beta<0$, then
there are no nontrivial solutions of (\ref{15})--(\ref{16}) for $c_{r}%
<U_{\min}$.
\end{lemma}

Although there are several necessary conditions as indicated above, there has
been very few sufficient conditions for the barotropic instability of shear
flows. In the fluid literature, the linear instability was studied for some
special shear flows. The barotropic instability of Bickley jet ($U\left(
y\right)  =\operatorname{sech}^{2}y$)$\ $was studied by numerical computations
and asymptotic analysis (cf.
\cite{Balmforth-Piccolo2001,Burns-Maslowe-Brown2002,Engevik2004,Howard-Drazin1964,Lipps1962,Kuo1974,Maslowe1991}%
). The stability boundary of hyperbolic-tangent shear flow was studied in
\cite{Dickinson-Clare1973,Howard-Drazin1964,Kuo1974}. Other references on the
barotropic instability include \cite{Drazin-Beaumont-Coaker1982,
Drazin-Howard1966, lindzen1988, lindzen1978, SHS1993}. In this paper, we
consider the barotropic instability of the following class of shear flows.

\begin{definition}
\label{de11} The flow $U$ is in class $\mathcal{K}$ if $U\in{C}^{3}%
([y_{1},y_{2}])$, $U$ is not a constant function on $[y_{1},y_{2}]$, and for
each $\beta\in {\rm Ran}\,(U^{\prime\prime})$, there exists $U_{\beta}\in {\rm Ran}\, (U)$
such that
\[
K_{\beta}(y):={\frac{\beta-U^{\prime\prime}(y)}{U(y)-U_{\beta}}}%
\]
is non-negative and bounded on $[y_{1},y_{2}]$. Furthermore, $U$ is said to be
in class $\mathcal{K}^{+}$ if $U$ is in class $\mathcal{K}$ and $K_{\beta}$ is
positive on $[y_{1},y_{2}]$ for each $\beta\in {\rm Ran}\, (U^{\prime\prime})$.
\end{definition}

Flows in class $\mathcal{K}^{+}\ $include $U\left(  y\right)  =\sin y,\ \tanh
y$, and more generally any $U\left(  y\right)  $ satisfying the ODE
$U^{\prime\prime}=g\left(  U\right)  $ with $g\in C^{1}({\rm Ran}\,(U))$ and
$g^{\prime}<0$ on ${\rm Ran}\,(U)$. One important property for flows in class
$\mathcal{K}^{+}$ is that there is a uniform $H^{2}$ bound for the unstable
solutions of (\ref{15})--(\ref{16}), see Lemma \ref{le22}. Neutral modes are
the solutions of (\ref{15})--(\ref{16}) with $c\in\mathbf{R}$. In the study of
stability of a shear flow $U\left(  y\right)  $, it is often important to
locate the neutral modes which are limits of a sequence of unstable modes.
These so called neutral limiting modes determine the boundary from
instability to stability. In Theorems \ref{th-neutral modes}%
--\ref{th-neutral limiting modes}, all $H^{2}\ $neutral modes for a general
shear flow, and consequently, all neutral limiting modes for a flow in class
$\mathcal{K}^{+}$, are classified into four types by their phase speed $c$: 1)
$c=U\left(  z\right)  $ such that $\beta=U^{\prime\prime}\left(  z\right)  $;
2) $c=U\left(  y_{1}\right)  $ or $U\left(  y_{2}\right)  $; 3) $c$ is a
critical value of $U$; 4) $c$ is outside the range of $U$. Here, the neutral
modes of types 2) and 3) might be singular, and type 4) is called non-resonant
since the phase speed $c$ causes no interaction with the basic flow $U\left(
y\right)  $. This contrasts greatly with the non-rotating case $\beta=0$,
where it was shown in \cite{lin-xu-2017} that for neutral modes in $H^{2}$,
$c$ must be an inflection value of $U$.

In the literature, it is common to look for unstable modes near neutral modes.
A useful approach to determine the stability boundary is to study the local
bifurcation of unstable modes near all possible neutral limiting wave numbers
and then combine these information to detect the stability/instability at any
wave number. In \cite{Lin2003}, this approach was used to show that when
$\beta=0$, any flow $U\left(  y\right)  $ in class $\mathcal{K}^{+}$ is
linearly stable if and only if $\alpha\geq\alpha_{\max}$, where $-\alpha
_{\max}^{2}$ is the principal eigenvalue of the operator $-\frac{d^{2}}%
{dy^{2}}-K_{0}(y)$. However, when $\beta\neq0$, there are several difficulties
in this approach. First, we need to deal with the subtle perturbation problem
near singular neutral modes. Second, for non-resonant neutral modes, the phase
speed $c$ is to be determined. Moreover, near these non-resonant neutral
modes, the bifurcation of unstable modes is usually non-smooth (see Remark
\ref{remark tung}). In some literature (e.g. \cite{Tung1981}), it was believed
that these non-resonant neutral modes are not adjacent to unstable modes. This
turns out to be not true from our study of the Sinus flow in Section
\ref{section-sinus}.

In this paper, we develop a new approach to study the barotropic instability
of shear flows. First, we write the linearized equation in a Hamiltonian form
$\partial_{t}\omega=JL\omega$, where $J$ is anti-self-adjoint and $L$ is
self-adjoint as defined in (\ref{defn-J-L}). For a fixed wave number $\alpha$,
by taking the ansatz $\omega=\omega_{\alpha}\left(  y,t\right)  e^{i\alpha x}%
$, the linearized equation can be written in a Hamiltonian form $\partial
_{t}\omega_{\alpha}=J_{\alpha}L_{\alpha}\omega_{\alpha}$, where $J_{\alpha}$
and $L_{\alpha}$ are defined in (\ref{defn-J-L-alpha}). Then by the
instability index theorem recently developed in \cite{Lin2017} for general
Hamiltonian PDEs, we get the index formula (\ref{index-formula}). This formula
implies that to determine the instability at any $\alpha>0$, it suffices to
count the number of neutral modes with a non-positive signature (i.e.
$\left\langle L_{\alpha}\omega_{\alpha},\omega_{\alpha}\right\rangle \leq0$).
The four types of neutral modes in $H^{2}$ are counted separately. In
particular, the counting of non-resonant neutral modes can be reduced to study
$\lambda_{n}\left(  \beta,c\right)  $, the $n$-th eigenvalue of the
Sturm-Liouville operator $-\frac{d^{2}}{dy^{2}}-{\frac{\beta-U^{\prime\prime}%
}{U-c}}$ for $c\notin {\rm Ran}\,(U)$. An important observation is that for a
non-resonant neutral mode $\left(  c,\alpha,\beta,\phi\right)  $, the sign
$\left\langle L_{\alpha}\omega_{\alpha},\omega_{\alpha}\right\rangle $ is
determined by $\partial_{c}\lambda_{n}\left(  \beta,c\right)  $, where
$\alpha^{2}=-\lambda_{n}\left(  \beta,c\right)  >0$ and $\omega_{\alpha}%
=-\phi^{\prime\prime}+\alpha^{2}\phi$. Therefore, by studying the shape of the
graph of $\lambda_{n}\left(  \beta,c\right)  $, we are able to count the
non-resonant neutral modes with a non-positive signature. Combining with the
index count for the other three types of neutral modes, we can find the
stability boundary in the whole parameter space $\left(  \alpha
,\beta\right)  $. See Subsection \ref{subsection-stability-index} for more
detailed discussions about this approach. In this approach, we avoid the study
of the bifurcation of unstable modes near neutral modes, which is particularly
tricky for singular and non-resonant neutral modes.

In Section \ref{section-sinus}, we study in details the classical Sinus flow
\[
U\left(  y\right)  =\left(  1+\cos\left(  \pi y\right)  \right)
/2,\ \ \ \ \ y\in\left[  -1,1\right]  .
\]
For the Sinus flow, $L_{\alpha}$ has at most one negative eigenvalue.
Moreover, the singular neutral wave speeds exist only at the endpoints of
${\rm Ran}\,(U)=\left[  0,1\right]  $. The set of all the eigenvalues of  corresponding singular Sturm-Liouville
operator $-\frac{d^{2}}{dy^{2}}-{\frac{\beta-U^{\prime\prime}}{U-c}}$ is bounded from below and can be computed by
using hypergeometric functions. The spectral continuity of the operators
$-\frac{d^{2}}{dy^{2}}-{\frac{\beta-U^{\prime\prime}}{U-c}\ }${can\ be shown
}at the end points $c=0,1\ $by studying the singular limits when
$c\rightarrow1^+$ and $c\rightarrow0^-$. Based on these properties and the above
approach, we obtain a simple characterization of the stability boundary in the
parameter space. By the index formula and the relation of $\partial_{c}%
\lambda_{1}\left(  \beta,c\right)  $ with the $\left\langle L_{\alpha}%
\cdot,\cdot\right\rangle $ sign for a fixed $\beta$, the graph of $\lambda_{1}^-\left(
\beta,c\right)  $ has at most one hump, more precisely, monotone or single
humped respectively for negative or positive sign of $\left\langle L_{\alpha
}\cdot,\cdot\right\rangle $ at singular neutral modes. Here $\lambda^-_{1}$ is  the negative part of $\lambda_{1}$. The lower part of the
stability boundary is given exactly by  $\sup_{c\in(-\infty,0]\cup[1,+\infty)}\lambda^-_{1}\left(
\beta,c\right)$. In \cite{Kuo1974} and \cite{Pedlosky1987}, it was concluded
from numerical computations that the lower stability boundary is the curve of
singular neutral modes ($\beta>0$) and the modes with zero wave number ($\beta<0$). Our results give a correction to this commonly
accepted picture. In fact, only part of the lower stability boundary consists
of singular neutral modes with negative $\left\langle L_{\alpha}\cdot
,\cdot\right\rangle $ and the modes with zero wave number, while the other part consists of non-resonant neutral
modes. The new stability boundary is confirmed by more accurate numerical
results. Same results on the stability boundary can be obtained for more
general flows similar to Sinus flow. Moreover, we count the exact number of
non-resonant neutral modes in each stability region. As we discuss below, this
has important implication on the nonlinear dynamics near shear flows.

Lastly, we study some dynamical behaviors near the shear flows. First, the
existence of nontrivial traveling wave solutions is shown near shear flows
with non-resonant neutral modes. These traveling waves, which have fluid
trajectories moving in one direction, do not exist when there is no rotation
(i.e. $\beta=0$) and therefore are purely due to rotating effects. We expect
the nonlinear dynamics is much richer due to the existence of these traveling
waves. Second, the Hamiltonian structure of the linearized equation is used to
prove the linear inviscid damping for stable shears with no neutral modes
(Theorem \ref{thm damping stable}) and in the center space for the unstable
shears (Theorem \ref{thm-damping-center}). These results are useful for the
further study of nonlinear dynamics near the shear flows, such as nonlinear
inviscid damping (for stable flows without neutral modes) and the construction
of invariant manifolds (for unstable flows).

This paper is organized as follows. In Section \ref{section-neutral mode}, we
classify all the neutral modes in $H^{2}$ for general shear flows. For shear
flows in class $\mathcal{K}^{+}$, by proving a uniform $H^{2}$ bound for
unstable modes, we obtain a classification of neutral limiting modes. In
Section 3, for flows in class $\mathcal{K}^{+}$, we derive an instability
index formula by using the Hamiltonian structure of the linearized fluid
equation. Then a general approach is developed to find the stability boundary
for flows in class $\mathcal{K}^{+}$. In Section \ref{section-sinus}, we find
the stability boundary for the Sinus flow in details. In Sections
\ref{section bifurcation} and \ref{section-damping}, the bifurcation of
nontrivial traveling waves and the linear inviscid damping are studied,
respectively. Section \ref{section discussions} contains the summary and
discussion of the results for Sinus flow.

\section{Neutral modes in $H^{2}$}

\label{section-neutral mode}

In this section, we first classify neutral modes in $H^{2}$ for a general
shear flow. For flows in class $\mathcal{K}^{+}$, we prove that unstable modes
have a uniform $H^{2}$ bound and therefore any neutral limiting mode is in
$H^{2}$. As a result, a classification of neutral limiting modes is obtained.

\subsection{Classification of neutral modes in $H^{2}$}

In this subsection, we give a classification of $H^{2}\ $neutral modes for a
general shear flow. First, we give the precise definition of neutral modes.

\begin{definition}
\label{defn neutral mode} $(c_{s},\alpha_{s},\beta_{s},\phi_{s})$ is said to
be a neutral mode if $c_{s}\in\mathbf{R}$, $\alpha_{s}>0$, $\beta_{s}%
\in\mathbf{R}$, and $\phi_{s}$ is a nontrivial solution to the Sturm-Liouville
equation
\begin{align}\label{Sturm-Liouville equation-def-neutral mode}
-\phi_{s}^{\prime\prime}+\alpha_{s}^{2}\phi_{s}-{\frac{\beta_{s}%
-U^{\prime\prime}}{U-c_{s}}}\phi_{s}=0,\;\;on\;\;(y_{1},y_{2}),
\end{align}
with the boundary conditions $\phi_{s}(y_{1})=\phi_{s}(y_{2})=0$. If $\phi
_{s}\in H^{2}(y_{1},y_{2})$, we call $(c_{s},\alpha_{s},\beta_{s},\phi_{s})$
to be a $H^{2}\ $neutral mode.
\end{definition}

If the equation (\ref{Sturm-Liouville equation-def-neutral mode}) is singular, by a solution $\phi$ we mean $\phi$ solves  (\ref{Sturm-Liouville equation-def-neutral mode}) on $(y_1,y_2)\setminus\{U=c_s\}$.
For convenience, we make the following assumption:
\begin{hypothesis}\label{hypothesis-U}
Let $U\in C^{2}([y_{1},y_{2}])$. Assume that
for any
$c\in( U_{\min},U_{\max})$, $\{U-c=0\}$ is a
finite set.
\end{hypothesis}

\begin{remark}
\label{remark-hypothesis-U} Hypothesis \ref{hypothesis-U} is true for generic
$C^{2}\ $flows. Suppose there exists $c_{0}\in(U_{\min},$ $U_{\max})$ such that
$\{U-c_{0}=0\}$ is an infinite set. Then for any accumulation point $x_{0}$ of
$\{U-c_{0}=0\}$, we have $U^{(n)}(x_{0})=0$ for all $1\leq n\leq k$ if $U\in
C^{k}$. This implies that Hypothesis \ref{hypothesis-U} is satisfied for
analytic flows and for flows in class $\mathcal{K}^{+}$. In fact, it is true
for any flow $U\left(y\right)  $ satisfying the 2nd order ODE $U^{\prime
\prime}=k\left(  y\right)  g\left(  U\right)  $, where \thinspace$k>0$ is bounded and
$g\in C^{1}$ by the uniqueness of ODE solutions.
\end{remark}

Assume that $U$ satisfies Hypothesis \ref{hypothesis-U}. Then for any
$c\in(U_{\min},U_{\max})$, the set $\{U-c=0\}\cap(y_{1},y_{2})$ is non-empty,
which we denote by%
\begin{equation}
\{z_{i}\ |\ 1\leq i\leq k_{c},\ z_{1}<z_{2}<\cdots<z_{k_{c}}%
\}.\label{arrange-z}%
\end{equation}
Set $z_{0}:=y_{1}$ and $z_{k_{c}+1}:=y_{2}$.

\begin{lemma}
\label{le23} Assume that $U$ satisfies Hypothesis \ref{hypothesis-U}. Let
$\phi$ be a solution of (\ref{15})--(\ref{16}) with $\alpha>0$, $\beta
\in\mathbf{R}$ and $c\in(U_{\min},U_{\max})$. If there exists $1\leq i_{0}\leq
k_{c}$ such that $\phi\in H^{1}(z_{i_{0}-1},\ z_{i_{0}+1})$, and $(U(y)-c)(U(z)-c)<0$ for all $y\in(z_{i_{0}-1},\ z_{i_{0}})$ and all $z\in(z_{i_{0}},\ z_{i_{0}+1})$, then $\phi$ can
not vanish at $z_{i_{0}-1},\ z_{i_{0}}$ and $z_{i_{0}+1}$ simultaneously
unless it vanishes identically on at least one of the intervals $(z_{i_{0}%
-1},z_{i_{0}})$ and $(z_{i_{0}},z_{i_{0}+1})$.
\end{lemma}

\begin{remark}
We refer the readers to Lemma \ref{lemma neutral c range} for the cases
$c=U_{\min}$ or $c=U_{\max}$.
\end{remark}

\begin{proof}
Suppose $\phi(z_{j})=0$ for $j=i_{0}-1,i_{0},i_{0}+1$. If $\beta\geq0$ and
$U-c<0$ on $(z_{i_{0}-1},z_{i_{0}})$, or $\beta<0$ and $U-c>0$ on
$(z_{i_{0}-1},z_{i_{0}})$, we define $x_{1}=z_{i_{0}-1}$ and $x_{2}=z_{i_{0}}%
$. If $\beta\geq0$ and $U-c<0$ on $(z_{i_{0}},z_{i_{0}+1})$, or $\beta<0$ and
$U-c>0$ on $(z_{i_{0}},z_{i_{0}+1})$, we define $x_{1}=z_{i_{0}}$ and
$x_{2}=z_{i_{0}+1}$. Then we get
\[
\int_{x_{1}}^{x_{2}}\left(  |\phi^\prime|^{2}+\alpha^{2}|\phi|^{2}-{\frac
{\beta-U^{\prime\prime}}{U-c}}|\phi|^{2}\right)  dy=0.
\]
Note that
\[\left|{U'(y)|\phi(y)|^2\over U(y)-c}\right|\leq {|U'(y)|\|\phi'\|_{L^2(x_1,y)}^2(y-x_1)\over |U'(\xi_y)|(y-x_1)}\to0, \;\;\text{as} \;\;y\to x_1^+,
\]
where $\xi_y\in(x_1,y)$. Similarly, ${U'(y)|\phi(y)|^2\over U(y)-c}\to0$ as $y\to x_2^-$.
Thus by integration by parts, we obtain
\[
\int_{x_{1}}^{x_{2}}\Big|\phi^{\prime}-\frac{U^{\prime}\phi}{U-c}%
\Big|^{2}dy+\int_{x_{1}}^{x_{2}}\left(  \alpha^{2}-{\frac{\beta}{U-c}}\right)
|\phi|^{2}dy=0.
\]
Note that $\alpha^{2}-{\beta/(U-c)}>0$ on $(x_{1},x_{2})$ in all cases. Thus,
$\phi\equiv0$ on $(x_{1},x_{2})$.
\end{proof}

Note that $(U(y)-c)(U(z)-c)<0$ for all $y\in(z_{i_{0}-1},\ z_{i_{0}})$ and all $z\in(z_{i_{0}},\ z_{i_{0}+1})$  holds true unless $U'(z_{i_0})=0$.
The following lemma will be used later in the classification of neutral modes.

\begin{lemma}
\label{le24} Let $\phi$ be a solution of (\ref{15}) with $\alpha>0$, $\beta
\in\mathbf{R}$ and $c\in {\rm Ran}\, (U)$. Assume that $y_{1}\leq x_{0}<x_{1}%
<x_{2}\leq y_{2}$ satisfy $U(x_{1})-c=0$ and $U-c\neq0$ on $(x_{0},x_{1}%
)\cup(x_{1},x_{2})$. If $U^{\prime}(x_{1})\neq0$ and $\phi\in C^{1}%
(x_{0},x_{2})$ satisfies the initial conditions $\phi(x_{1})=a_{1}$,
$\phi^{\prime}(x_{1})=a_{2}$ for some $a_{1},a_{2}\in\mathbf{R}$, then $\phi$
is unique on the interval $(x_{0},x_{2})$.
\end{lemma}

\begin{proof}
It suffices to show that $\phi(x_{1})=\phi^{\prime}(x_{1})=0$ implies
$\phi\equiv0$ on $(x_{1},x_{2})$, and similarly $\phi\equiv0$ on $(x_{0}%
,x_{1})$.

Let $v:=\phi\in C^{1}([x_{1},x_{2}))$ and $u:=\phi^{\prime}\in C^{0}%
([x_{1},x_{2}))$. Then (\ref{15}) can be written as a first order ODE system
\begin{equation}
\left\{
\begin{array}
[c]{l}%
{\frac{dv}{dt}}=u,\\
{\frac{du}{dt}}=\alpha^{2}v-{\frac{\beta-U^{\prime\prime}}{U-c}}v,
\end{array}
\right.  \label{29}%
\end{equation}
with the initial data $v(x_{1})=u(x_{1})=0$. For a fixed $z\in\lbrack
{x_{1},x_{2})}$ and any $s\in\lbrack x_{1},z]$,
\begin{equation}
|v(s)|\leq\int_{x_{1}}^{s}|u(\tau)|d\tau\leq(z-x_{1})|u|_{L^{\infty}}\left(
z\right)  , \label{210}%
\end{equation}
where $|u|_{L_{\infty}}\left(  z\right)  :=\sup\limits_{x_{1}\leq s\leq
z}|u(s)|$. Thus $|v|_{L^{\infty}}(z)\leq(z-x_{1})|u|_{L^{\infty}}(z)$. Since
$U^{\prime}(x_{1})\neq0$, there exists $\delta_{0}>0$ such that $U^{\prime
}(s)\neq0$ for $s\in\lbrack x_{1},x_{1}+\delta_{0}]$. Let
\[
\delta_{1}:=\min\limits_{x_{1}\leq s\leq x_{1}+\delta_{0}}\left\vert
U^{\prime}(s)\right\vert ,\ \ \ \ \ \ \ \ C_{0}:=\max\limits_{y_{1}\leq s\leq
y_{2}}\left\vert \beta-U^{\prime\prime}(s)\right\vert .
\]
Then for each $z\in\lbrack x_{1},x_{1}+\delta_{0}]$,
\begin{equation}
\left\vert {\frac{[(\beta-U^{\prime\prime})v](z)}{U(z)-c}}\right\vert
=\left\vert (\beta-U^{\prime\prime})(z){\frac{v(z)-v(x_{1})}{U(z)-U(x_{1})}%
}\right\vert \leq{\frac{C_{0}|u|_{L^{\infty}}(z)}{\delta_{1}}}, \label{212}%
\end{equation}
and thus by (\ref{29})--(\ref{212})
\[
|u|_{L^{\infty}}(z)\leq\int_{x_{1}}^{z}\left(  |\alpha^{2}v(\tau)|+\left\vert
{\frac{[(\beta-U^{\prime\prime})v](\tau)}{U(\tau)-c}}\right\vert \right)
d\tau\leq\left(  \alpha^{2}(x_{2}-x_{1})+{\frac{C_{0}}{\delta_{1}}}\right)
\int_{x_{1}}^{z}|u|_{L^{\infty}}(\tau)d\tau.
\]
Therefore by Gronwall inequality, we have $u\equiv0$ and thus $v=\phi
\equiv0\ $on $[x_{1},x_{1}+\delta_{0}]$. This implies that $\phi\equiv0$ on
$\left(  x_{1},x_{2}\right)  $, since the ODE (\ref{15}) is regular in
$\left(  x_{1},x_{2}\right)  $. This completes the proof.
\end{proof}

Now we classify all the wave speeds of neutral modes in $H^{2}$ for a general
shear flow.

\begin{theorem}
\label{th-neutral modes} Assume that $U$ satisfies Hypothesis
\ref{hypothesis-U}. Let $(\phi_{s},\alpha_{s},\beta,c_{s})$ be a neutral mode
in $H^{2}$. Then the wave speed $c_{s}$ must be one of the following:

(i) $c_{s}=U(z)$ such that $\beta=U^{\prime\prime}(z)$;

(ii) $c_{s}=U(y_{1})$ or $c_{s}=U(y_{2})$;

(iii) $c_{s}$ is a critical value of $U$;

(iv) $c_{s}\notin {\rm Ran}\, (U)$.
\end{theorem}

\begin{proof}
It suffices to show that if $c_{s}\in {\rm Ran}\, (U)$, then one of cases (i)-(iii)
is true. Suppose that $c_{s}\in {\rm Ran}\, (U)$ and $\{U(z)-c_{s}=0\}\cap\left\{
y_{1},y_{2}\right\}  \neq\emptyset. $ Then $c_{s}=U(y_{1})$ or $c_{s}%
=U(y_{2})$, that is, case (ii) is true. Otherwise, $c_{s}\neq U(y_{i}),i=1,2$.
We consider two cases below.

\textbf{Case 1.} There exists $z_{s}\in\{U-c_{s}=0\}$ such that $\beta
=U^{\prime\prime}(z_{s})$. Then (i) is true.

\textbf{Case 2.} $\beta\neq U^{\prime\prime}(z)$ for all $z\in\{U-c_{s}=0\}$.
We divide it into two subcases.

\textbf{Case 2.1.} $U^{\prime}(z)\neq0$ for all $z\in\{U-c_{s}=0\}$. Then
${c_{s}}\in(U_{\min},U_{\max})$.

In this subcase, $\{U-c_{s}=0\}$ is non-empty and finite, so we use the
notation in (\ref{arrange-z}). We claim that there exists $1\leq i_{1}\leq
k_{c_{s}}$ such that $\phi_{s}(z_{i_{1}})\neq0$. Suppose otherwise, $\phi
_{s}(z_{i})=0$ for any $1\leq i\leq k_{c_{s}}$. For any fixed $1\leq i_{0}\leq
k_{c_{s}}$, by the fact that $U^{\prime}(z_{i_{0}})\neq0$ and by Lemma
\ref{le23}, $\phi_{s}\equiv0$ on at least one of the intervals $[z_{i_{0}%
-1},z_{i_{0}}]$ and $[z_{i_{0}},z_{i_{0}+1}]$. Since $\phi_{s}\in H^{2}%
(y_{1},y_{2})$, it follows that $\phi_{s}\in C^{1}([y_{1},y_{2}])$ and by
Lemma \ref{le24}, $\phi\equiv0$ on $[z_{i_{0}-1},z_{i_{0}+1}]$ and hence on
$[y_{1},y_{2}]$. Thus, there exists $1\leq i_{1}\leq k_{c_{s}}$ such that
$\phi(z_{i_{1}})\neq0$. Then near $z_{i_{1}}$,
\[
\phi_{s}^{\prime\prime}=\alpha_{s}^{2}\phi_{s}-{\frac{\beta-U^{\prime\prime}%
}{U-c_{s}}}\phi_{s}\notin L_{{\rm loc}}^{2}(y_{1},y_{2}),
\]
which is a contradiction to $\phi_{s}\in H^{2}(y_{1},y_{2})$.

\textbf{Case 2.2.} There exists $z_{0}\in\{U-c_{s}=0\}$ such that $U^{\prime
}(z_{0})=0$. In this subcase, $c_{s}=U(z_{0})$ is a critical value of $U$.
This finishes the proof of Theorem \ref{th-neutral modes}.
\end{proof}

\begin{remark}
For the neutral modes in Theorem \ref{th-neutral modes}, we call \textrm{(i)}
to be regular, \textrm{(ii)}--\textrm{(iii)} to be singular, and \textrm{(iv)}
to be non-resonant. For $\beta=0$, it can be shown that only \textrm{(i)} is
true, that is, for all neutral modes in $H^{2}$, the phase speed must be an
inflection value of $U$ (see Appendix in \cite{lin-xu-2017}).
\end{remark}

\subsection{Neutral limiting modes for flows in class $\mathcal{K^{+}}$}

First, we prove the uniform $H^{2}\ $bound for unstable solutions for flows in
class $\mathcal{K^{+}}$. First, we need the following two identities about
unstable modes, of which (\ref{integral-identity-2}) was used in
\cite{Kuo1949} to show Rayleigh's criterion for barotropic instability.

\begin{lemma}
\label{le21} Let $\phi$ be a solution of (\ref{15})--(\ref{16}) with
$c=c_{r}+ic_{i}$ $(c_{i}>0)$. Then
\begin{equation}
\int_{y_{1}}^{y_{2}}{\frac{(\beta-U^{\prime\prime})}{|U-c|^{2}}}|\phi
|^{2}dy=0, \label{integral-identity-2}%
\end{equation}
and
\begin{equation}
\int_{y_{1}}^{y_{2}}\left[  |\phi^{\prime}|^{2}+\alpha^{2}|\phi|^{2}%
-{\frac{(\beta-U^{\prime\prime})(U-q)}{|U-c|^{2}}}|\phi|^{2}\right]  dy=0
\label{integral-identity-1}%
\end{equation}
for any $q\in\mathbf{R}$.
\end{lemma}

\begin{proof}
Thanks to (\ref{15})--(\ref{16}), we get by integration by parts that
\[
\int_{y_{1}}^{y_{2}}\left[  |\phi^{\prime}|^{2}+\alpha^{2}|\phi|^{2}%
-{\frac{(\beta-U^{\prime\prime})(U-c_{r})}{|U-c|^{2}}}|\phi|^{2}\right]
dy=0,\;\; \int_{y_{1}}^{y_{2}}{\frac{(\beta-U^{\prime\prime})c_{i}}{|U-c|^{2}%
}}|\phi|^{2}dy=0,
\]
which yields (\ref{integral-identity-2})--(\ref{integral-identity-1}).
\end{proof}

\begin{lemma}
\label{le22} Let $U$ be in class $\mathcal{K}^{+}$ and $\beta\in(\min
U^{\prime\prime},\max U^{\prime\prime})$. If $\phi$ is a solution of
(\ref{15})--(\ref{16}) with $c=c_{r}+ic_{i}$ $(c_{i}>0)$, then
\begin{equation}
\int_{y_{1}}^{y_{2}}(|\phi^{\prime}|^{2}+\alpha^{2}|\phi|^{2})dy\leq
\int_{y_{1}}^{y_{2}}K_{\beta}|\phi|^{2}dy, \label{H1 bound}%
\end{equation}%
\begin{equation}
\int_{y_{1}}^{y_{2}}(|\phi^{\prime\prime}|^{2}+2\alpha^{2}|\phi^{\prime}%
|^{2}+\alpha^{4}|\phi|^{2})dy\leq\Vert K_{\beta}\Vert_{L^{\infty}}\int_{y_{1}%
}^{y_{2}}K_{\beta}|\phi|^{2}dy. \label{H2 bound}%
\end{equation}

\end{lemma}

\begin{proof} Since $\beta\in(\min
U^{\prime\prime},\max U^{\prime\prime})$, $K_{\beta}$ is 
positive and bounded on $[y_{1},y_{2}]$. 
Let $q=U_{\beta}-2(U_{\beta}-c_{r})$ in (\ref{integral-identity-1}). Then
\begin{align*}
\int_{y_{1}}^{y_{2}}(|\phi^{\prime}|^{2}+\alpha^{2}|\phi|^{2})dy =  &
\int_{y_{1}}^{y_{2}}{\frac{(\beta-U^{\prime\prime})[(U-U_{\beta})+2(U_{\beta
}-c_{r})]}{|U-c|^{2}}}|\phi|^{2}dy\\
=  &  \int_{y_{1}}^{y_{2}}K_{\beta}{\frac{(U-c_{r})^{2}-(U_{\beta}-c_{r})^{2}%
}{|U-c|^{2}}}|\phi|^{2}dy\\
\leq &  \int_{y_{1}}^{y_{2}}K_{\beta}{\frac{(U-c_{r})^{2}}{|U-c|^{2}}}%
|\phi|^{2}dy \leq\int_{y_{1}}^{y_{2}}K_{\beta}|\phi|^{2}dy.
\end{align*}
This proves (\ref{H1 bound}).

Now we consider (\ref{H2 bound}). Multiplying (\ref{15}) by $\bar{\phi
}^{\prime\prime}$ and integrating it, we have
\begin{align*}
\int_{y_{1}}^{y_{2}}(|\phi^{\prime\prime}|^{2}+\alpha^{2}|\phi^{\prime}%
|^{2})dy = -\alpha^{2}\int_{y_{1}}^{y_{2}}{\frac{(\beta-U^{\prime\prime
})(U-c_{r}+ic_{i})}{|U-c|^{2}}}|\phi|^{2}dy+\int_{y_{1}}^{y_{2}}{\frac
{(\beta-U^{\prime\prime})^{2}}{|U-c|^{2}}}|\phi|^{2}dy.
\end{align*}
Taking the real part in the above equality and using Lemma \ref{le21}, we get
\begin{align*}
\int_{y_{1}}^{y_{2}}(|\phi^{\prime\prime}|^{2}+\alpha^{2}|\phi^{\prime}%
|^{2})dy =  &  -\alpha^{2}\int_{y_{1}}^{y_{2}}{\frac{(\beta-U^{\prime\prime
})(U-c_{r})}{|U-c|^{2}}}|\phi|^{2}dy+\int_{y_{1}}^{y_{2}}{\frac{(\beta
-U^{\prime\prime})^{2}}{|U-c|^{2}}}|\phi|^{2}dy\\
=  &  -\alpha^{2}\int_{y_{1}}^{y_{2}}(|\phi^{\prime}|^{2}+\alpha^{2}|\phi
|^{2})dy+\int_{y_{1}}^{y_{2}}{\frac{(\beta-U^{\prime\prime})^{2}}{|U-c|^{2}}%
}|\phi|^{2}dy.
\end{align*}
Thus, by Lemma \ref{le21} and (\ref{H1 bound}) we obtain
\begin{align*}
&  \int_{y_{1}}^{y_{2}}(|\phi^{\prime\prime}|^{2}+2\alpha^{2}|\phi^{\prime
}|^{2}+\alpha^{4}|\phi|^{2})dy \leq\Vert K_{\beta}\Vert_{L^{\infty}}%
\int_{y_{1}}^{y_{2}}{\frac{(\beta-U^{\prime\prime})(U-U_{\beta})}{|U-c|^{2}}%
}|\phi|^{2}dy\\
=  &  \Vert K_{\beta}\Vert_{L^{\infty}}\int_{y_{1}}^{y_{2}}(|\phi^{\prime
}|^{2}+\alpha^{2}|\phi|^{2})dy \leq\Vert K_{\beta}\Vert_{L^{\infty}}%
\int_{y_{1}}^{y_{2}}K_{\beta}|\phi|^{2}dy.
\end{align*}
This completes the proof of (\ref{H2 bound}).
\end{proof}

Next, we consider neutral limiting modes defined below.

\begin{definition}
\label{defn limiting neutral} Let $\beta\in(\min U^{\prime\prime},\max
U^{\prime\prime})$. We call $(c_{s},\alpha_{s},\beta,\phi_{s})$ to be a
neutral limiting mode if $c_{s}\in\mathbf{R}$, $\alpha_{s}>0$ and there exists
a sequence of unstable modes $\{(c_{k},\alpha_{k},\beta,\phi_{k})\}$ (with
$\operatorname{Im}(c_{k})=c_{k}^{i}>0$ and $\Vert\phi_{k}\Vert_{L^{2}}=1$) to
(\ref{15})--(\ref{16}) such that $c_{k}\rightarrow c_{s},\alpha_{k}%
\rightarrow\alpha_{s}$, $\phi_{k}$ converges uniformly to $\phi_{s}$ on any
compact subset of $S_{0}$ as $k\rightarrow\infty$, $\phi_{s}^{\prime\prime}$
exists on $S_{0}$, and $\phi_{s}$ satisfies
\begin{equation}
(U-c_{s})(-\phi_{s}^{\prime\prime}+\alpha_{s}^{2}\phi_{s})-(\beta
-U^{\prime\prime})\phi_{s}=0\label{Rayleigh-Kuo-2}%
\end{equation}
on $S_{0}$, where $S_{0}$ denotes the complement of the set $\{U-c_{s}=0\}$ in
the interval $[y_{1},y_{2}]$. Here $c_{s}$ is called the neutral limiting
phase speed and $\alpha_{s}$ is called the neutral limiting wave number.
\end{definition}

Then we prove that any neutral limiting mode is in $H^{2}$ for flows in class
$\mathcal{K}^{+}$.

\begin{lemma}
\label{le26} Let $U$ be in class $\mathcal{K}^{+}$ and $\beta\in(\min
U^{\prime\prime},\max U^{\prime\prime})$. Suppose that $(\phi_{s},\alpha
_{s},\beta,c_{s})$ is a neutral limiting mode. Then $\phi_{s}\in H^{2}%
(y_{1},y_{2})$.
\end{lemma}

\begin{proof}
Let $\{(\phi_{k},\alpha_{k},\beta,c_{k})\}$ be a sequence of unstable modes
converging to $(\phi_{s},\alpha_{s},\beta,c_{s})$ in the sense of Definition
\ref{defn limiting neutral}. Note that $\phi_{k}$ has been normalized such
that $\Vert\phi_{k}\Vert_{L^{2}}=1$. Since $U$ is in class $\mathcal{K}^{+}$,
by Lemma \ref{le22}, there exists $C>0$ such that $\Vert\phi_{k}\Vert_{H^{2}%
}\leq C$ for all $k\geq1$. Thus there exists $\phi_{0}\in H^{2}(y_{1},y_{2})$
such that, up to a subsequence, $\phi_{k}\rightharpoonup\phi_{0}$ in
$H^{2}(y_{1},y_{2})$, $\phi_{k}\to\phi_{0}$ in $C^{1}([y_{1},y_{2}])$ and
$\left\Vert \phi_{0}\right\Vert _{H^{2}}\leq C$. Recall that $S_{0}$ denotes
the complement of the set $\{U-c_{s}=0\}$ in the interval $[y_{1},y_{2}]$. For
any compact subset $S_{1}\subset S_{0}$, $\phi_{0}$ solves
(\ref{Rayleigh-Kuo-2}) on $S_{1}$ and thus by Definition
\ref{defn limiting neutral}, $\phi_{0}\equiv\phi_{s}\in H^{2}(y_{1},y_{2})$.
This completes the proof.
\end{proof}

Combining Remark \ref{remark-hypothesis-U} (ii), Theorem
\ref{th-neutral modes}, and Lemma \ref{le26}, we get the classification of
neutral limiting modes for flows in class $\mathcal{K}^{+}$.

\begin{theorem}
\label{th-neutral limiting modes} Assume that $U$ is in class $\mathcal{K}%
^{+}$. Let $(\phi_{s},\alpha_{s},\beta,c_{s})$ be a neutral limiting mode.
Then the neutral limiting phase speed $c_{s}$ must be one of the following:

(i) $c_{s}=U_{\beta}$;

(ii) $c_{s}=U(y_{1})$ or $c_{s}=U(y_{2})$;

(iii) $c_{s}$ is a critical value of $U$;

(iv) $c_{s}\notin {\rm Ran}\, (U)$.
\end{theorem}

\begin{remark}
\label{remark tung}In Theorem IV of \cite{Tung1981}, Tung showed that for a
general $C^{2}$ shear flow $U\left(  y\right)  $, the phase speed $c_{s}\ $of
any neutral limiting mode $\left(  c_{s},\alpha_{s},\beta,\phi_{s}\right)  $
must lie in ${\rm Ran}\,(U)$. His proof is under the assumption that for fixed
$\beta$, the dispersion relation $c\left(  \alpha\right)  $ is an analytic
function of $\alpha$ near $\alpha_{s}$ when $c\left(  \alpha_{s}\right)
=c_{s}\notin {\rm Ran}\,(U)$. However, as suggested in \cite{Maslowe1991}, the
analytic assumption might not always hold and it is possible that $c_{s}\notin
{\rm Ran}\,(U)$. In Theorem \ref{transition}, we give the sharp stability boundary
for the Sinus flow, part of which consists of non-resonant neutral modes. This
shows that the phase speed of neutral limiting modes can indeed lie outside
the range of $U$.

Below we give some explanation why the analytic assumption of $c\left(
\alpha\right)  \ $could fail. Assume that $(\phi_{s},\alpha_{s},\beta,c_{s})$
is a neutral mode and $c_{s}\notin {\rm Ran}\, (U)$. From the Rayleigh-Kuo equation
(\ref{15})--(\ref{16}), the perturbation of the eigenvalue $c$ near $c_{s}%
\ $appears to be analytic in $\alpha\ $when $c_{s}$ is not in the range of
$U$. However, we should consider the operator associated with the linearized
equation (\ref{13}) with the wave number $\alpha$ ($\beta$ is fixed):
\[
B_{\alpha}\omega:=U\omega-(\beta-U^{\prime\prime})(-{\frac{d^{2}}{dy^{2}}%
}+\alpha^{2})^{-1}\omega.
\]
Then $c_{s}$ is an isolated eigenvalue of $B_{\alpha_{s}}$. Define the Riesz
projection operator
\[
P_{\alpha_{s}}:=-{\frac{1}{2\pi i}}\int_{\Gamma}(B_{\alpha_{s}}-\zeta
)^{-1}d\zeta,
\]
where $\Gamma$ is a circle in $\rho(B_{\alpha_{s}})$ enclosing $c_{s}$ and no
other spectral points of $B_{\alpha_{s}}$. Note that ${\rm Ran}\,(P_{\alpha_{s}})$
is the generalized eigenspace of the eigenvalue $c_{s}$ and $\dim
({\rm Ran}\,(P_{\alpha_{s}}))$ is the algebraic multiplicity of $c_{s}$ (see P. 181
in \cite{Kato1980}). Although the geometric multiplicity of $c_{s}$ is $1$,
the algebraic multiplicity of $c_{s}$ may be larger than $1$. In such case,
there might be more than one branches of eigenvalues emanating from $c_{s}$
when we perturb the parameter $\alpha$ in a neighborhood of $\alpha_{s}$. As a
consequence, the expansion of $c\left(  \alpha\right)  -c\left(  \alpha
_{s}\right)  $ near $\alpha=\alpha_{s}$ is given by the Puiseux series (see P.
65 in \cite{Kato1980}) instead of the power series in the analytic case. This
suggests that we can not exclude the possibility that for a neutral limiting
mode, $c_{s}\ $is outside ${\rm Ran}\, (U)$.
\end{remark}

Similar to the proof of Theorem 4.1 in \cite{Lin2003}, we get the existence of
unstable modes when the wave number is slightly to the left of a regular
neutral wave number.

\begin{lemma}
\label{lemma bifurcation regular neutral} Let $U$ be in class $\mathcal{K}%
^{+}$, $\beta\in(\min U^{\prime\prime},\max U^{\prime\prime})$, and
$(c_{s},\alpha_{s},\beta,\phi_{s})$ be a regular neutral mode with
$c_{s}=U_{\beta}$. Then there exists $\varepsilon_{0}<0$ such that if
$\varepsilon_{0}<\varepsilon<0$, there is a nontrivial solution $\phi
_{\varepsilon}$ to the equation
\[
(U-U_{\beta}-c(\varepsilon))(\phi_{\varepsilon}^{\prime\prime}-\alpha
(\varepsilon)^{2}\phi_{\varepsilon})+(\beta-U^{\prime\prime})\phi
_{\varepsilon}=0
\]
with $\phi_{\varepsilon}(y_{1})=\phi_{\varepsilon}(y_{2})=0$. Here
$\alpha(\varepsilon)=\sqrt{\varepsilon+\alpha_{s}^{2}}$ is the perturbed wave
number and $U_{\beta}+c(\varepsilon)$ is an unstable wave speed with
$\operatorname{Im}(c(\varepsilon))>0$.
\end{lemma}

The next lemma comes from \cite{Kuo1949, Tung1981}.

\begin{lemma}
\label{lemma neutral c range}Let $U\in C^{2}(\left[  y_{1},y_{2}\right]  ) $.
When $\beta\geq0$ and $c_{s}\geq U_{\max}$ (or $\beta\leq0$ and $c_{s}\leq
U_{\min}$), for any $\alpha>0\ $there exist no neutral modes in $H^{2}$.
\end{lemma}

\section{Hamiltonian formulation, index formula and instability criteria}

\label{Hamiltonian formulation index formula and instability criteria}

In this section, we write the linearized fluid equation for flows in class
$\mathcal{K}^{+}$ in a Hamiltonian form and derive an instability index
formula to be used later to find the stability condition. Also, we compute the
associated energy quadratic forms for unstable modes and neutral limiting
modes. Then we establish an important relation between signs of the energy
quadratic form for non-resonant neutral modes and the graph of eigenvalues of
the Sturm-Liouville operators (\ref{defn-L_c}). Combining above, we give a new
approach to study the instability of shear flows in class $\mathcal{K}^{+}$.

\subsection{Hamiltonian formulation and instability index formula}

\label{Hamiltonian formulation and index formula}

In this subsection, we write the linearized equation (\ref{13}) as a
Hamiltonian system and use the index theory developed in \cite{Lin2017} to
derive an instability index formula for flows in class $\mathcal{K}^{+}$. Fix
$\beta\in(\min U^{\prime\prime},\max U^{\prime\prime})$. In the traveling
frame $(x-U_{\beta}t,y,t)$, the linearized equation (\ref{13}) becomes
\begin{equation}
\partial_{t}\omega+(U-U_{\beta})\partial_{x}\omega-(\beta-U^{\prime\prime
})\partial_{x}\psi=0. \label{linearized-Euler-shear-frame}%
\end{equation}
Note that the change of coordinates $(x,y,t)\rightarrow(x-U_{\beta}t,y,t)$
does not affect the stability of the shear flows. Recall that for flows in
class $\mathcal{K}^{+}$, $K_{\beta}=\frac{\beta-U^{\prime\prime}}{U-U_{\beta}%
}>0.$ Let the $x$ period be $2\pi/\alpha$ for some $\alpha>0$. Define the
non-shear space on the periodic channel $S_{2\pi/\alpha}\times\left[
y_{1},y_{2}\right]  $ by
\begin{equation}
X=\left\{  \omega=\sum_{k\in\mathbf{Z},\ k\neq0}e^{ik\alpha x}\omega
_{k}\left(  y\right)  ,\ \Vert\omega\Vert_{X}^{2}=\Vert\frac{1}{\sqrt
{K_{\beta}}}\omega\Vert_{L^{2}}^{2}<\infty\right\}  . \label{non-shear space}%
\end{equation}
Clearly, $X=L^{2}$. The equation (\ref{linearized-Euler-shear-frame}) can be
written in a Hamiltonian form
\begin{equation}
\omega_{t}=-(\beta-U^{\prime\prime})\partial_{x}\left(  {\omega}/{K_{\beta}%
}-\psi\right)  =JL\omega, \label{linearized Euler-Hamiltonian}%
\end{equation}
where
\begin{equation}
J=-(\beta-U^{\prime\prime})\partial_{x}:X^{\ast}\rightarrow X,\ \ \ \ L={1}%
/{K_{\beta}}-\left(  -\Delta\right)  ^{-1}:X\rightarrow X^{\ast},
\label{defn-J-L}%
\end{equation}
are anti-self-adjoint and self-adjoint, respectively. Denote $n^{-}\left(
L\right)  $ $\left(  n^{0}\left(  L\right)  \right)  \ $to be the number of
negative (zero) directions of $L$ on $X$. Define the operator%
\begin{equation}
A_{0}=-\Delta-K_{\beta}:H^{2}\rightarrow L^{2} \label{defn-A0}%
\end{equation}
and
\begin{equation}
\tilde{L}_{0}=-\frac{d^{2}}{dy^{2}}-K_{\beta}:H^{2}\left(  y_{1},y_{2}\right)
\rightarrow L^{2}\left(  y_{1},y_{2}\right)  , \label{defn-L0}%
\end{equation}
with the Dirichlet boundary conditions. Then by Lemma 11.3 in \cite{Lin2017},
we have
\[
n^{0}\left(  L\right)  =n^{0}\left(  A_{0}\right)  =2\sum_{l\geq1}n^{0}\left(
\tilde{L}_{0}+l^{2}\alpha^{2}\right)  ,\;\;n^{-}\left(  L\right)
=n^{-}\left(  A_{0}\right)  =2\sum_{l\geq1}n^{-}\left(  \tilde{L}_{0}%
+l^{2}\alpha^{2}\right)  .
\]
If $n^{-}(\tilde{L}_{0})\neq0$, let $-\alpha_{\max}^{2}$ be the principal
eigenvalue of $\tilde{L}_{0}$ and $\phi_{0}$ be the eigenfunction. When
$\tilde{L}_{0}\geq0$, let $\alpha_{\max}=0$. Then by the above relations, when
$\alpha\geq\alpha_{\max},$ $L$ is non-negative and the stability holds.

Below, we consider the case when $\alpha<\alpha_{\max}$. Let $Y=L_{\frac
{1}{K_{\beta}}}^{2}\left(  y_{1},y_{2}\right)  $. The space $X$ has an
invariant decomposition $X=\bigoplus_{l\in\mathbf{Z},\ l\neq0}X^{l}$, where
\begin{equation}
\ \ X^{l}=\left\{  e^{i\alpha lx}\omega_{l}\left(  y\right)  ,\ \omega_{l}\in
Y\right\}  . \label{defn-X-l}%
\end{equation}
The linearized equation can be studied on each $X^{l}$ separately. To simplify
notations, we only consider $l=\pm1$ below. On the subspace
\[
X_{\alpha}:=\left\{  e^{i\alpha x}\omega\left(  y\right)  ,\ \omega\in
Y\right\}  ,
\]
the operator $JL$ is reduced to the ODE operator $J_{\alpha}L_{\alpha}$ acting
on the weighted space $Y,$ where
\begin{equation}
J_{\alpha}=-i\alpha(\beta-U^{\prime\prime}),\ \ \ \ L_{\alpha}=\frac
{1}{K_{\beta}}-\left(  -\frac{d^{2}}{dy^{2}}+\alpha^{2}\right)  ^{-1}.
\label{defn-J-L-alpha}%
\end{equation}
By the same proof of Lemma 11.3 in \cite{Lin2017}, we have
\[
n^{-}\left(  L_{\alpha}\right)  =n^{-}\left(  \tilde{L}_{0}+\alpha^{2}\right)
,\ n^{0}\left(  L_{\alpha}\right)  =n^{0}\left(  \tilde{L}_{0}+\alpha
^{2}\right)  .
\]

Since $J_{\alpha}$ is not a real operator on $Y$, we define the invariant
subspace
\begin{align*}
X^{\alpha} =X_{\alpha}\oplus X_{-\alpha} =\left\{  \cos\left(  \alpha
x\right)  \omega_{1}\left(  y\right)  +\sin\left(  \alpha x\right)  \omega
_{2}\left(  y\right)  ,\ \omega_{1},\omega_{2}\in Y \right\}  ,
\end{align*}
which is isomorphic to the real space $Y\times Y=( L_{\frac{1}{K_{\beta} }%
}^{2}\left(  y_{1},y_{2}\right)  ) ^{2}$. For any
\[
\omega=\cos\left(  \alpha x\right)  \omega_{1}\left(  y\right)  +\sin\left(
\alpha x\right)  \omega_{2}\left(  y\right)  \in X^{\alpha},
\]
we have
\[
JL\omega=\left(  \cos\left(  \alpha x\right)  ,\sin\left(  \alpha x\right)
\right)  J^{\alpha}L^{\alpha}\left(
\begin{array}
[c]{c}%
\omega_{1}\\
\omega_{2}%
\end{array}
\right)  ,
\]
where%
\[
J^{\alpha}=\left(
\begin{array}
[c]{cc}%
0 & -\alpha(\beta-U^{\prime\prime})\\
\alpha(\beta-U^{\prime\prime}) & 0
\end{array}
\right)  ,\ \ L^{\alpha}=\left(
\begin{array}
[c]{cc}%
L_{\alpha} & 0\\
0 & L_{\alpha}%
\end{array}
\right)  .
\]
In the above, the operator $L_{\alpha}$ is defined in (\ref{defn-J-L-alpha}).
Thus to study the spectra of $JL$ on $X^{\alpha}$, we study the spectra of
$J^{\alpha}L^{\alpha}$ on $Y\times Y$. We note that
\begin{equation}
\sigma\left(  J^{\alpha}L^{\alpha}|_{Y\times Y}\right)  =\sigma\left(
J_{\alpha}L_{\alpha}|_{Y}\right)  \cup\sigma\left(  J_{-\alpha}L_{-\alpha
}|_{Y}\right)  ,\ \label{217}%
\end{equation}
and $\sigma\left(  J_{\alpha}L_{\alpha}|_{Y}\right)  $ is the complex
conjugate of $\sigma\left(  J_{-\alpha}L_{-\alpha}|_{Y}\right)  $.

By the instability index Theorem 2.3 in \cite{Lin2017} for linear Hamiltonian
PDEs, we have
\begin{equation}
2\tilde{k}_{i}^{\leq0}+2\tilde{k}_{c}+\tilde{k}_{0}^{\leq0}+\tilde{k}%
_{r}=n^{-}(L^{\alpha})=2n^{-}(L_{\alpha}), \label{218}%
\end{equation}
where $n^{-}(L^{\alpha})$ denotes the sum of multiplicities of negative
eigenvalues of $L^{\alpha}$, $\tilde{k}_{r}$ is the sum of algebraic
multiplicities of positive eigenvalues of $J^{\alpha}L^{\alpha}$, $\tilde
{k}_{c}$ is the sum of algebraic multiplicities of eigenvalues of $J^{\alpha
}L^{\alpha}$ in the first quadrant, $\tilde{k}_{i}^{\leq0}$ is the total
number of non-positive dimensions of $\langle L^{\alpha}\cdot,\cdot\rangle$
restricted to the generalized eigenspaces of purely imaginary eigenvalues of
$J^{\alpha}L^{\alpha}$ with positive imaginary parts, and $\tilde{k}_{0}%
^{\leq0}$ is the number of non-positive dimensions of $\langle L^{\alpha}%
\cdot,\cdot\rangle$ restricted to the generalized kernel of $J^{\alpha
}L^{\alpha}$ modulo $\ker L^{\alpha}$. By the next lemma, we have $\tilde
{k}_{0}^{\leq0}=0$, from which it follows that
\begin{equation}
2\tilde{k}_{i}^{\leq0}+2\tilde{k}_{c}+\tilde{k}_{r}=2n^{-}(L_{\alpha}).
\label{218'}%
\end{equation}

\begin{lemma}
\label{lemma generalized kernel}Let $E_{0}$ be the generalized zero eigenspace
of $J^{\alpha}L^{\alpha}$. Then $E_{0}=\ker L^{\alpha}$.
\end{lemma}

\begin{proof}
It suffices to show that the generalized zero eigenspace of $J_{\alpha
}L_{\alpha}$ on $Y$ coincides with $\ker L_{\alpha}$. Suppose there exists
$\omega\in Y$ such that
\begin{equation}
J_{\alpha}L_{\alpha}\omega=-i\alpha(U-U_{\beta})\left(  \omega-K_{\beta}
\psi\right)  =\tilde{\omega}\in\ker L_{\alpha}. \label{generalized kernel}%
\end{equation}
Let $\tilde{\psi}=\left(  -\frac{d^{2}}{dy^{2}}+\alpha^{2}\right)  ^{-1}%
\tilde{\omega}$. Then $-\tilde{\psi}^{\prime\prime2}\tilde{\psi}-K_{\beta}
\tilde{\psi}=0. $ Using the fact $\beta\in(\min U^{\prime\prime},\max
U^{\prime\prime})$ and by Lemma \ref{le23}, $\tilde{\psi}$ is not all zero on
the set $\left\{  U=U_{\beta}\right\}  $, which implies the same for
$\tilde{\omega}=K_{\beta} \tilde{\psi}$. Thus (\ref{generalized kernel})
gives
\[
\omega-K_{\beta} \psi=\frac{\tilde{\omega}}{-i\alpha(U-U_{\beta})}\notin
L^{2}\left(  y_{1},y_{2}\right)  \text{. }%
\]
This contradiction shows that the generalized kernel of $J_{\alpha}L_{\alpha}$
on $Y$ is the same as $\ker L_{\alpha}$.
\end{proof}

Now we derive the index formula for $J_{\alpha}L_{\alpha}$ on $Y$. Let
${k}_{r}$ be the sum of algebraic multiplicities of positive eigenvalues of
${J}_{\alpha}{L}_{\alpha}$, ${k}_{c}$ be the sum of algebraic multiplicities
of eigenvalues of ${J}_{\alpha}{L}_{\alpha}$ in the first and the forth
quadrants, ${k}_{i}^{\leq0}$ be the total number of non-positive dimensions of
$\langle{L}_{\alpha}\cdot,\cdot\rangle$ restricted to the generalized
eigenspaces of nonzero purely imaginary eigenvalues of ${J}_{\alpha}%
{L}_{\alpha}$. By (\ref{217}), we have the following relation
\begin{equation}
2k_{i}^{\leq0}=2\tilde{k}_{i}^{\leq0},\;\;2k_{c}=2\tilde{k}_{c},\;\;2k_{r}%
=\tilde{k}_{r}. \label{219}%
\end{equation}

Combining (\ref{218'}) and (\ref{219}), we get the following index formula for
${J}_{\alpha}{L}_{\alpha}$.

\begin{theorem}
\label{thm-index} Let $U$ be in class $\mathcal{K}^{+}$ and $\beta\in(\min
U^{\prime\prime},\max U^{\prime\prime})$. Then the following index formula
holds for the operator ${J}_{\alpha}{L}_{\alpha}$ on $Y$:
\begin{equation}
k_{c}+k_{r}+k_{i}^{\leq0}=n^{-}(L_{\alpha}). \label{index-formula}%
\end{equation}

\end{theorem}

From the index formula (\ref{index-formula}), the stability of shear flows is
reduced to determine $k_{i}^{\leq0}$. This corresponds to consider neutral
modes in $H^{2}$ with the wave speed $c_{s}\neq U_{\beta}$.

\begin{remark}
When $\beta=0$, it can be shown that $k_{i}^{\leq0}=0$ and the index formula
(\ref{index-formula}) reduces to $k_{c}+k_{r}=n^{-}(L_{\alpha})$ (see
\cite{lin-xu-2017}). When $\beta\neq0$, in general we have $k_{i}^{\leq0}%
\neq0$ as seen from Sinus flow in Section \ref{section-sinus}.
\end{remark}

\subsection{Computation of the quadratic form $\langle{L}_{\alpha}\cdot
,\cdot\rangle$}

\label{Computation of the quadratic form}

Firstly, we compute the quadratic form $\langle L_{\alpha}\cdot,\cdot\rangle$
for unstable modes and neutral limiting modes.

\begin{lemma}
\label{le-L-quadratic form} Let $(c,\alpha,\beta,\phi)$ be a solution of
Rayleigh-Kuo equation (\ref{15})-(\ref{16}) with $\phi\in H^{2}$. Then

(i)
\begin{equation}
\langle L_{\alpha}\omega,\omega\rangle=(c-U_{\beta})\int_{y_{1}}^{y_{2}}%
{\frac{(\beta-U^{\prime\prime})}{|U-c|^{2}}}|\phi|^{2}%
dy,\label{identity-L-quadratic}%
\end{equation}
where $\omega=-\phi^{\prime\prime}+\alpha^{2}\phi$.

(ii) If $(c,\alpha,\beta,\phi)$ is an unstable mode, then $\langle L_{\alpha
}\omega,\omega\rangle=0$.

(iii) If $(c,\alpha,\beta,\phi)$ is a regular or non-resonant neutral limiting
mode, then $\langle L_{\alpha}\omega,\omega\rangle=0$. If $(c,\alpha
,\beta,\phi)$ is a singular neutral limiting mode, then $\langle L_{\alpha
}\omega,\omega\rangle\leq0$.
\end{lemma}

\begin{proof}
We first show (\ref{identity-L-quadratic}). From (\ref{15}), we get
$(U-c)\omega=(\beta-U^{\prime\prime})\phi.$ Therefore%
\[
{\frac{\omega}{K_{\beta}}}-\phi={\frac{\left(  U-U_{\beta}\right)  \omega
}{\beta-U^{\prime\prime}}-\frac{\left(  U-c\right)  \omega}{\beta
-U^{\prime\prime}}=\frac{(c-U_{\beta})\omega}{\beta-U^{\prime\prime}}},
\]
and
\[
\langle L_{\alpha}\omega,\omega\rangle=\int_{y_{1}}^{y_{2}}\left(
{\frac{\omega}{K_{\beta}}}-\phi\right)  \bar{\omega}\ dy=\int_{y_{1}}^{y_{2}%
}{\frac{(c-U_{\beta})}{\beta-U^{\prime\prime}}}|\omega|^{2}dy=(c-U_{\beta
})\int_{y_{1}}^{y_{2}}{\frac{(\beta-U^{\prime\prime})}{|U-c|^{2}}}|\phi
|^{2}dy.
\]

The conclusion (ii) follows from (i) and Lemma \ref{le21}.

Next, we show conclusions in (iii). For a regular neutral limiting mode
$(c,\alpha,\beta,\phi)$, $\langle L_{\alpha}\omega,\omega\rangle=0$ by (i) and
the fact that $c=U_{\beta}$.

Let $\{(c_{k},\alpha_{k},\beta,\phi_{k})\}$ be a sequence of unstable modes
converging to a neutral limiting mode $(c,\alpha,\beta,\phi)$ in the sense of
Definition \ref{defn limiting neutral}. By Lemma \ref{le22}, $\Vert\phi
_{k}\Vert_{H^{2}}\leq C$ for some $C>0$ independent of $k$. Thus up to a
subsequence, $\phi_{k}\rightarrow\phi$ in $C^{1}([y_{1},y_{2}])$. When
$(c,\alpha,\beta,\phi)$ is non-resonant, by noting that $c\notin {\rm Ran}(U)$ we
have
\[
\langle L_{\alpha}\omega,\omega\rangle=(c-U_{\beta})\int_{y_{1}}^{y_{2}}%
{\frac{(\beta-U^{\prime\prime})}{|U-c|^{2}}}|\phi|^{2}dy=\lim
\limits_{k\rightarrow\infty}(c_{k}-U_{\beta})\int_{y_{1}}^{y_{2}}{\frac
{(\beta-U^{\prime\prime})}{|U-c_{k}|^{2}}}|\phi_{k}|^{2}dy=0.
\]
When $(c,\alpha,\beta,\phi)$ is singular, using the uniform bound of
$\left\Vert \omega_{k}\right\Vert _{L^{2}}$ with $\omega_{k}=-\phi_{k}%
^{\prime\prime}+\alpha_{k}^{2}\phi_{k}$, we have, up to a subsequence,
$\omega_{k}\rightharpoonup\omega$ in $L^{2}$, and thus
\begin{align*}
\langle L_{\alpha}\omega,\omega\rangle &  =\int_{y_{1}}^{y_{2}}\left(
{{\omega}/{K_{\beta}}}-\phi\right)  \bar{\omega}\ dy=\int_{y_{1}}^{y_{2}%
}\left[  {{|\omega|^{2}}/{K_{\beta}}}-(\left\vert \phi^{\prime}\right\vert
^{2}+\alpha^{2}\left\vert \phi\right\vert ^{2})\right]  \ dy\\
&  \leq\lim\limits_{k\rightarrow\infty}\int_{y_{1}}^{y_{2}}\left[
{{|\omega_{k}|^{2}}/{K_{\beta}}}-(\left\vert \phi_{k}^{\prime}\right\vert
^{2}+\alpha_{k}^{2}|\phi_{k}|^{2})\right]  \ dy=\lim\limits_{k\rightarrow
\infty}\langle L_{\alpha_{k}}\omega_{k},\omega_{k}\rangle=0.
\end{align*}
This completes the proof of the lemma.
\end{proof}

Next, we consider non-resonant neutral modes, which naturally correspond to
the following regular Sturm-Liouville operators. For $\beta\in\mathbf{R}$ and
$c\in\lbrack-\infty,U_{\min})\cup(U_{\max},+\infty]$, we define the operator
on $L^{2}\left(  y_{1},y_{2}\right)  $:
\begin{align}
\mathcal{L}_{\beta,c} &  =-\frac{d^{2}}{dy^{2}}-{\frac{\beta-U^{\prime\prime}%
}{U-c}},\label{defn-L_c}\\
D(\mathcal{L}_{\beta,c}) &  =\{\phi\in L^{2}(y_{1},y_{2}):\phi,\phi^{\prime
}\in AC([y_{1},y_{2}]),\mathcal{L}_{\beta,c}\phi\in L^{2}(y_{1},y_{2}%
),\phi(y_{1})=\phi(y_{2})=0\},\nonumber
\end{align}
where $AC([y_{1},y_{2}])$ is the space of absolutely continuous functions on
$[y_{1},y_{2}]$. For $c=U_{\min}$ or $U_{\max}$, $\mathcal{L}_{\beta,c}$ is
defined as (\ref{defn-L_c}) with $AC([y_{1},y_{2}])$ replaced by
$AC_{{\rm loc}}([y_{1},y_{2}]\setminus U^{-1}(c))$.

To determine the sign of quadratic form for non-resonant neutral modes, we
need to compute the derivative of the $n$-th eigenvalue of $\mathcal{L}%
_{\beta,c}$ with respect to $\beta$ and $c$ separately.

\begin{lemma}
\label{le28}

For $\beta\in\mathbf{R}$ and $c \in(-\infty, U _{\min})\cup(U _{\max},
+\infty)$, let $\lambda_{n} = \lambda_{n} (\beta, c)$ (here $n \geq1$) be the
$n$-th eigenvalue of $\mathcal{L}_{\beta,c}$, and $\phi_{n} ^{(\beta, c)}$ be
the corresponding eigenfunction with $\|\phi_{n} ^{(\beta, c)}\|_{L^{2}}=1$.
Then $\lambda_{n} (\beta, c)$ has partial derivatives
%
\begin{align}
{\frac{\partial\lambda_{n}}{\partial\beta}} ( \beta, c)  &  =-\int_{y_{1}%
}^{y_{2}}{\frac{1}{U-c}}|\phi_{n} ^{(\beta, c)}| ^{2}\ dy,\label{221}\\
{\frac{\partial\lambda_{n}}{\partial c}} ( \beta, c)  &  =-\int_{y_{1}}%
^{y_{2}}{\frac{\beta-U^{\prime\prime}}{(U-c)^{2}}}|\phi_{n} ^{(\beta, c)}%
|^{2}\ dy. \label{223}%
\end{align}

\end{lemma}

\begin{proof}
We first show that \eqref{221} holds. By Theorem 2.1 in \cite{Kong1999}, for a
fixed $c$, $\lambda_{n}$ is continuous as a function of $\beta\in\mathbf{R}$.
For any $\beta,\tilde{\beta}\in\mathbf{R}$, $\phi=\phi_{n}^{(\beta,c)}$ and
$\tilde{\phi}=\phi_{n}^{(\tilde{\beta},c)}$ satisfy
\[
-\phi^{\prime\prime}-{\frac{{\beta}-U^{\prime\prime}}{U-c}}\phi=\lambda
_{n}(\beta,c)\phi,\;\;-\tilde{\phi}^{\prime\prime}-{\frac{{\tilde{\beta}%
}-U^{\prime\prime}}{U-c}}\tilde{\phi}=\lambda_{n}(\tilde{\beta},c)\tilde{\phi
},
\]
with $\phi(y_{1})=\phi(y_{2})=\tilde{\phi}(y_{1})=\tilde{\phi}(y_{2})=0$.
Thus
\[
{\frac{\lambda_{n}(\beta,c)-\lambda_{n}(\tilde{\beta},c)}{\beta-\tilde{\beta}%
}}\int_{y_{1}}^{y_{2}}\phi\tilde{\phi}dy=-\int_{y_{1}}^{y_{2}}{\frac{1}{U-c}%
}\phi\tilde{\phi}dy.
\]
Taking the limit $\tilde{\beta}\rightarrow\beta$ in the above, we prove
(\ref{221}).

The formula (\ref{223}) can be proved in a similar way and we skip the details.
\end{proof}

The following is a straightforward consequence of Lemma \ref{le28}.

\begin{corollary}
\label{co21} (i) For fixed $c_{0}\in(-\infty,U_{\min})$, $\lambda_{n}%
(\beta,c_{0})$ is strictly decreasing for $\beta\in\mathbf{R}$.

(ii) For fixed $c_{0}\in(U_{\max},+\infty)$, $\lambda_{n}(\beta,c_{0})$ is
strictly increasing for $\beta\in\mathbf{R}$.

(iii) For fixed $\beta_{0}\in(-\infty,U_{\min}^{\prime\prime}]$, $\lambda
_{n}(\beta_{0},c)$ is strictly increasing for $c\in(-\infty,U_{\min})$ and
$c\in(U_{\max},+\infty)$, respectively.

(iv) For fixed $\beta_{0}\in\lbrack U_{\max}^{\prime\prime},+\infty)$,
$\lambda_{n}(\beta_{0},c)$ is strictly decreasing for $c\in(-\infty,U_{\min})$
and $c\in(U_{\max},+\infty)$, respectively.

\end{corollary}


Now we can determine the sign of $\langle L_{\alpha}\cdot,\cdot\rangle$ for a
non-resonant neutral mode by combining (\ref{identity-L-quadratic}) and
(\ref{223}).

\begin{theorem}
\label{non-resonant quadratic form} Let $(c,\alpha,\beta,\phi)$ be a
non-resonant neutral mode. Then $\alpha^{2}=-\lambda_{n_{0}}(\beta,c)>0$ for
some $n_{0}\geq1$ and
\begin{equation}
\langle L_{\alpha}\omega_{\alpha},\omega_{\alpha}\rangle=-(c-U_{\beta}%
)\frac{\partial\lambda_{n_{0}}}{\partial c}(\beta
,c)\label{equality-L-form-derivative}%
\end{equation}
where $\omega_{\alpha}=-\phi^{\prime\prime}+\alpha^{2}\phi$.

\end{theorem}

\subsection{Stability criteria}

\label{subsection-stability-index}

In this subsection, we give a new method to study the instability of a shear
flow $U$ in class $\mathcal{K}^{+}$. Fix $\beta\in\left(  \min U^{\prime
\prime},\max U^{\prime\prime}\right)  $ and $\alpha>0$. We determine the
barotropic instability of the shear flow $U$ in the following steps.

First, recall that $n^{-}\left(  L_{\alpha}\right)  =n^{-}(\tilde{L}%
_{0}+\alpha^{2})$, where $\tilde{L}_{0}$ is defined in (\ref{defn-L0}). By
(\ref{index-formula}), linear stability at the wave number $\alpha\ $is
equivalent to the condition $n^{-}\left(  L_{\alpha}\right)  =k_{i}^{\leq0}$.
To determine $k_{i}^{\leq0}$, we need to study the neutral modes in $H^{2}$.
By Theorem \ref{th-neutral modes}, the neutral wave speed $c$ must be one of
the following four types: (i) $U_{\beta}$; (ii) $U\left(  y_{1}\right)  $ or
$U\left(  y_{2}\right)  $; (iii) critical values of $U$; (iv) outside
${\rm Ran}\,(U)$. Since $c=U_{\beta}$ corresponds to the zero eigenvalue of
$J_{\alpha}L_{\alpha}$ (defined in (\ref{defn-J-L-alpha})), it has no
contribution to $k_{i}^{\leq0}$. To find neutral modes of types (ii) and
(iii), we need to solve a (possibly) singular eigenvalue problem for the
operator $\mathcal{L}_{\beta,c}$ defined in (\ref{defn-L_c}) with $c$ to be
$U\left(  y_{1}\right)  ,U\left(  y_{2}\right)  $ or a critical value of $U$.
For such $c$, if $-\alpha^{2}$ is a negative eigenvalue of $\mathcal{L}%
_{\beta,c}$ with the eigenfunction $\phi\in H^{2}(y_{1},y_{2})$, then
$\lambda=-i\alpha(c-U_{\beta})$ is a nonzero and purely imaginary eigenvalue
of $J_{\alpha}L_{\alpha}$ with the eigenfunction $\omega=-\phi^{\prime\prime
}+\alpha^{2}\phi\in L^{2}(y_{1},y_{2})$. Denote $k_{i}^{-}\left(
\lambda\right)  $ $(k_{i}^{\leq0}(\lambda))$ to be the number of negative
(non-positive) dimensions of $\left\langle L_{\alpha}\cdot,\cdot\right\rangle
$ restricted to the generalized eigenspace of $\lambda$ for $J_{\alpha
}L_{\alpha}$. Then when $\langle L_{\alpha}\omega,\omega\rangle\neq0$, noting
that $\lambda$ is purely imaginary, we infer that $\lambda$ is a simple
eigenvalue of $J_{\alpha}L_{\alpha}$ and
\[
k_{i}^{\leq0}\left(  \lambda\right)  =k_{i}^{-}\left(  \lambda\right)
=\left\{
\begin{array}
[c]{cc}%
1 & \text{if }\langle L_{\alpha}\omega,\omega\rangle<0,\;\\
0 & \text{if }\langle L_{\alpha}\omega,\omega\rangle>0.
\end{array}
\right.
\]
When $\langle L_{\alpha}\omega,\omega\rangle=0$, we have $k_{i}^{\leq0}\left(
\lambda\right)  \geq1$ and $\lambda$ might be a multiple eigenvalue of
$J_{\alpha}L_{\alpha}.$

For the case (iv) of non-resonant neutral modes, $\mathcal{L}_{\beta,c}$ is a
regular Sturm-Liouville operator but $c\notin {\rm Ran}\,\left(  U\right)  $ is not
given explicitly. By the definition of $\lambda_{n}$ in Lemma \ref{le28}, we
obtain that for any given $\alpha>0$, the number of non-resonant neutral modes
is exactly the number of solutions of $\lambda_{n}\left(  \beta,c\right)
=-\alpha^{2}$ for all $n\geq1$. Let $c^{\ast}$ be a solution of $\lambda
_{n_{0}}\left(  \beta,c\right)  =-\alpha^{2}$ for some $n_{0}\geq1$. Then
$\lambda_{n_{0}}\left(  \beta,
c^{\ast}\right)  <0$ is the $n_{0}$-th eigenvalue of $\mathcal{L}%
_{\beta,c^{\ast}}$ with the eigenfunction $\phi^{\ast}$, and correspondingly,
$\lambda^{\ast}=-i\alpha(c^{\ast}-U_{\beta})$ is a nonzero and purely
imaginary eigenvalue of $J_{\alpha}L_{\alpha}$ with the eigenfunction
$\omega^{\ast}=-\phi^{\ast{\prime\prime}}+\alpha^{2}\phi^{\ast}$. If
$\partial_{c}\lambda_{n_{0}}(\beta,c^{\ast})\neq0$, then by
(\ref{equality-L-form-derivative}),
\[
\langle L_{\alpha}\omega^{\ast},\omega^{\ast}\rangle=-(c^{\ast}-U_{\beta
})\partial_{c}\lambda_{n_{0}}(\beta,c^{\ast})\neq0,
\]
which implies that $\lambda^{\ast}$ is a simple eigenvalue of $J_{\alpha
}L_{\alpha}$ with
\[
k_{i}^{-}\left(  \lambda^{\ast}\right)  =\left\{
\begin{array}
[c]{cc}%
1 & \text{if }(c^{\ast}-U_{\beta})\partial_{c}\lambda_{n_{0}}(\beta,c^{\ast
})>0\;\\
0 & \text{if }(c^{\ast}-U_{\beta})\partial_{c}\lambda_{n_{0}}(\beta,c^{\ast
})<0.
\end{array}
\right.
\]
If $\partial_{c}\lambda_{n_{0}}(\beta,c^{\ast})=0$, then $\langle L_{\alpha
}\omega^{\ast},\omega^{\ast}\rangle=0$ and $\lambda^{\ast}$ might be a
multiple eigenvalue of $J_{\alpha}L_{\alpha}$. In this case, we have
$k_{i}^{\leq0}\left(  \lambda^{\ast}\right)  \geq1$. Note that by Lemma
\ref{le-L-quadratic form}, only points with $\partial_{c}\lambda_{n_{0}}%
(\beta,c^{\ast})=0$ could be a neutral limiting mode, i.e., possibly be the
boundary for stability/instability.

\begin{remark}
\label{rm-upper bound}For fixed $\beta$, suppose the operator $\tilde{L}_{0}$
has at least one negative eigenvalue and recall that   the lowest
one is denoted by $-\alpha_{\max}^{2}<0$. By (\ref{H1 bound}) and Lemma \ref{lemma bifurcation regular neutral},
$\alpha_{\max}>0$ gives the upper bound for the unstable wave numbers in the
sense that linear stability holds when $\alpha\geq\alpha_{\max}$ and there
exist unstable modes for $\alpha$ slightly less than $\alpha_{\max}$. When
$\beta=0$, it was shown in \cite{Lin2003} that $\left(  0,\alpha_{\max
}\right)  $ is exactly the interval of unstable wave numbers. When $\beta
\neq0$, the situation becomes more subtle as seen from the study of Sinus flow
in the next section. In particular, when $\beta>0\ $there is always a set of
stable wave numbers in $\left(  0,\alpha_{\max}\right)  $.
\end{remark}

\section{Sharp stability criteria for the Sinus flow}

\label{section-sinus}

In this section, we consider the barotropic instability of the Sinus flow
\[
U(y)={({1+\cos(\pi y)})/{2}},\ \ \ \ y\in\lbrack-1,1].
\]
We will use the approach outlined in Subsection
\ref{subsection-stability-index} to determine the sharp stability boundary for
the Sinus flow in the parameter space $\left(  \alpha,\beta\right)  $. Our
results correct the stability boundary given in the classical references
\cite{Kuo1974, Pedlosky1987}. Moreover, the new stability boundary is
confirmed by more accurate numerical results.

\subsection{ Sharp stability boundary}

\label{Sharp stability transition}

In this subsection, we determine the stability boundary for the Sinus flow. In
particular, we show that the lower part of stability boundary  is given by the
supremum  of wave numbers of non-resonant neutral modes. Here, the supremum is set to be zero if there are no non-resonant neutral modes.

Since for any $\beta\in\mathbf{R}$,
\[
K_{\beta}={\frac{\beta-U^{\prime\prime}}{U-({{1}/{2}}-{{\beta}/{\pi^{2}}})}%
}=\pi^{2},
\]
the Sinus flow belongs to class $\mathcal{K}^{+}$ with $U_{\beta}={{1}/{2}%
}-{{\beta}/{\pi^{2}}}.$ Rayleigh-Kuo criterion ensures that a necessary
condition for instability is $\beta\in(-{{\pi^{2}}/{2}},{{\pi^{2}}/{2}})$. Fix
$\alpha>0$. The linearized equation around the Sinus flow is written in the
Hamiltonian form
\[
\partial_{t}\omega=J_{\alpha}L_{\alpha}\omega,\ \omega\in L^{2}[-1,1],\
\]
where
\begin{equation}
J_{\alpha}=-i\alpha(\beta-U^{\prime\prime}),\;\;L_{\alpha}={{1}/{\pi^{2}}%
}-(-{{d^{2}}/{dy^{2}}}+\alpha^{2})^{-1}. \label{defn-J-L-sinus}%
\end{equation}
Clearly,
\[
\sigma(L_{\alpha})=\{{{1}/{\pi^{2}}}-{{1}/({{{k^{2}\pi^{2}}/{4}}+\alpha^{2}}%
)},\;\;k=1,2,\cdots\}.
\]

{By Theorem \ref{thm-index}, we get the following instability index formula
for the Sinus flow. }

\begin{theorem}
\label{th31} Consider the Sinus flow and $\beta\in\mathbf{R}$. For any
$\alpha\in\lbrack\sqrt{3}\pi/2,+\infty)$, $L_{\alpha}$ is non-negative and the
flow is linearly stable for perturbations of period ${2\pi}/{\alpha}$. For any
$\alpha\in(0,\sqrt{3}\pi/2)$, the index formula
\begin{equation}
k_{c}+k_{r}+k_{i}^{\leq0}=1 \label{index-formula-sinus}%
\end{equation}
is satisfied for the eigenvalues of $J_{\alpha}L_{\alpha}$, where $J_{\alpha}$
and $L_{\alpha}$ are defined in (\ref{defn-J-L-sinus}).
\end{theorem}

\begin{remark}
By Rayleigh-Kuo criterion, the Sinus flow is linearly stable for any
$\alpha>0\ $when $\left\vert \beta\right\vert \geq{{\pi^{2}}/{2}}$. In this
case, when $\alpha\in(0,\sqrt{3}\pi/2)$, the index formula
(\ref{index-formula-sinus}) implies that $k_{i}^{\leq0}=1$, that is, there
exists one nonzero and purely imaginary eigenvalue of $J_{\alpha}L_{\alpha}%
\ $with non-positive signature. This corresponds to a $H^{2}\ $neutral mode
with non-positive signature, which by results in Section 4.2, must be of
non-resonant type.
\end{remark}

When $\beta\in\left(  -{{\pi^{2}}/{2,\pi^{2}}/{2}}\right)  $ and $\alpha
\in(0,\sqrt{3}\pi/2)$, the index formula (\ref{index-formula-sinus}) implies
that $k_{i}^{\leq0}=1$ and linear stability holds if and only if $k_{i}%
^{\leq0}=1$. Thus the study of linear stability is reduced to count $H^{2}%
\ $neutral modes $\left(  c,\alpha,\beta,\phi\right)  $ with $\left\langle
L_{\alpha}\omega_{\alpha},\omega_{\alpha}\right\rangle \leq0$, where
$\omega_{\alpha}=-\phi^{\prime\prime}+\alpha^{2}\phi$. By Theorem
\ref{th-neutral modes}, the possible wave speeds of $H^{2}\ $neutral modes
are: (i) $c=U_{\beta}=1/2-\beta/\pi^{2}$; (ii) $c=0$; (iii) $c=1$ and (iv)
$c\in(-\infty,0)\cup(1,\infty)$. The regular neutral mode with $c=U_{\beta}$
corresponds to zero eigenvalue of $J_{\alpha}L_{\alpha}$ and has no
contribution to $k_{i}^{\leq0}$. For Sinus flow, there are no singular neutral
modes for $c\ $in the interior of ${\rm Ran}\,(U)=\left[  0,1\right]  $, that is,
$c\in\left(  0,1\right)  $. This simplifies our discussion below. The neutral
modes are solutions of the eigenvalue problem
\begin{equation}
-\phi^{\prime\prime}-\frac{\beta-U^{\prime\prime}}{U-c}\phi=\lambda
\phi,\;\;\phi(\pm1)=0\label{eq44}%
\end{equation}
with $c\in(-\infty,0]\cup\lbrack1,\infty)$. The corresponding operators
$\mathcal{L}_{\beta,c}$ are defined in (\ref{defn-L_c}). We present the
following spectral properties for (\ref{eq44}) and leave the proofs to the
next two subsections:

(1) The spectral points of (\ref{eq44}) with $c=0,1$ or $\pm\infty$ are
computed explicitly. In particular, they are bounded from below and are all
discrete eigenvalues. See Proposition \ref{pro-41}.

(2) Spectral continuity at the boundary: Recall that $\lambda_{n}(\beta,c)$
denotes the $n$-th eigenvalue of (\ref{eq44}).
There are two types of spectral continuity according to the boundary points.
The first one is that $\lambda_{n}(\beta,c)\rightarrow n^{2}\pi^{2}/4$ as
$c\rightarrow\pm\infty$. This type is generic for any flow $U\in C^{2}%
([y_{1},y_{2}])$. See Lemma \ref{lem-cinfty} and Remark
\ref{general-infty-limits}. The second type is that $\lambda_{n}(\beta,\cdot)$
is left continuous at $0$ for $\beta\in\left(  0,{\pi^{2}}/{2}\right)  $ and
right continuous at $1$ for $\beta\in\left(  -{\pi^{2}}/{2},0\right]  $. See
Proposition \ref{prop-singular limit}.

Assuming these spectral properties, we give a simple characterization of the
stability boundary for Sinus flow. Similar results can be obtained for other
flows $U\left(  y\right)  \ $in class $\mathcal{K^{+}}$with $n^{-}\left(
\tilde{L}_{0}\right)  =1$, where $\tilde{L}_{0}$ is defined in (\ref{defn-L0}).

\begin{lemma}
\label{2nd eigenvalue positive} Let $\beta\in\left[  -{\pi^{2}}/{2},{\pi^{2}%
}/{2}\right]  $. Then $\lambda_{n}(\beta,c)>(n^{2}/4-1)\pi^{2}$ for any
$n\geq1$ and any $c\in(-\infty,0)\cup(1,+\infty)$. Here, $\lambda_{n}%
(\beta,c)$ is the $n$-th eigenvalue of the operator $\mathcal{L}_{\beta,c}$.
\end{lemma}

\begin{proof}
If $c\in(-\infty,0)$, then $(1/2-c)\pi^{2}>\beta$, and by Corollary \ref{co21}
(i) we have $\lambda_{n}(\beta,c)>\lambda_{n}((1/2-c)\pi^{2},c)=(n^{2}%
/4-1)\pi^{2}$. If $c\in(1,+\infty)$, then $(1/2-c)\pi^{2}<\beta$, and by
Corollary \ref{co21} (ii) we have $\lambda_{n}(\beta,c)>\lambda_{n}%
((1/2-c)\pi^{2},c)=(n^{2}/4-1)\pi^{2}$.
\end{proof}

The above lemma implies that when $n\geq2$, $\lambda_{n}(\beta,c)>0$ and there
are no non-resonant neutral modes associated with such eigenvalues. Then by
Proposition \ref{pro-41}, there are no singular neutral modes
associated with $\lambda_{n}(\beta,1)$ and $\lambda_{n}(\beta,0)$ for $n\geq
2$. Moreover, for Sinus flow there are no neutral modes for $c\in\left(
0,1\right)  $. Therefore, to count the index $k_{i}^{\leq0}$, we only need to
study the first eigenvalue $\lambda_{1}(\beta,c)$ for $c\in(-\infty
,0]\cup\lbrack1,+\infty)$. For convenience, we denote $\lambda_{\beta
}(c)=\lambda_{1}({\beta},c)$.

\begin{proposition}
\label{lower transition} Consider Sinus flow and $\beta\in(-\frac{\pi^{2}}%
{2},\frac{\pi^{2}}{2})$. The lower bound of unstable wave numbers is given by
$\Lambda_{\beta}=\sup_{c\notin(0,1)}\lambda_{\beta}^{-}$, where $\lambda
_{\beta}^{-}(c)=\max\{-\lambda_{\beta}(c),0\}$ is the negative part of
$\lambda_{\beta}(c)$. More precisely, we have $\Lambda_{\beta}<3\pi^{2}/4$ and
\end{proposition}

(i) for $\alpha^{2}\in(\Lambda_{\beta},3\pi^{2}/4)$, $k_{i}^{\leq0}=0$ (linear instability);

(ii) for $\alpha^{2}\in(0,\Lambda_{\beta}]$, $k_{i}^{\leq0}=1$ (linear stability).

\begin{proof}
First, we show $\Lambda_{\beta}<3\pi^{2}/4$ when $\beta\in(-\frac{\pi^{2}}%
{2},\frac{\pi^{2}}{2})$. Indeed, $\lambda_{\beta}^-(c)<3\pi^{2}/4$ is true for
$c\in(-\infty,0)\cup(1,+\infty)$ by Lemma \ref{2nd eigenvalue positive} and
for $c=0,1,\pm\infty $ by Proposition \ref{pro-41}.

Proof of (i): For any $\alpha^{2}\in(\Lambda_{\beta},3\pi^{2}/4)$, we have
$\lambda_{\beta}(c)\geq-\Lambda_{\beta}>-\alpha^{2}$, and thus there are no
$H^{2}\ $neutral modes with the wave number $\alpha$. This implies that
$k_{i}^{\leq0}=0$ and the linear instability follows from the index formula
(\ref{index-formula-sinus}).

Proof of (ii): We assume $\Lambda_{\beta}>0$ since the conclusion is trivial
when $\Lambda_{\beta}=0$. By Lemma \ref{lemma neutral c range}, $\Lambda
_{\beta}=\sup_{c\in(-\infty,0]}\lambda_{\beta}^{-}$ for $\beta>0$, and
$\Lambda_{\beta}=\sup_{c\in\lbrack1,\infty)}\lambda_{\beta}^{-}$ for $\beta
<0$. For fixed $\beta\in\left(  0,\frac{\pi^{2}}{2}\right)  $, there exists
$c^{\ast}\in(-\infty,0]$ such that $\lambda_{\beta}(c^{\ast})=-\Lambda_{\beta
}<0.$ Note that $\lim_{c\to-\infty}\lambda_{\beta}(c)=\pi^{2}/4>0$ by Lemma
\ref{lem-cinfty} and $\lambda_{\beta}\left(  c\right)  $ is continuous on
$(-\infty,0]$. For any $\alpha^{2}\in(0,\Lambda_{\beta}]$, let $c_{1}%
\in(-\infty,0]$ be  the minimal solution of $\lambda_{\beta}(c)=-\alpha^{2}$.
Then $\partial_{c}\lambda_{\beta}\left(  c_{1}\right) \leq0$, which
implies that $k_{i}^{\leq0}\geq1$ by (\ref{equality-L-form-derivative}). This
implies that $k_{i}^{\leq0}=1$ and the linear stability by
(\ref{index-formula-sinus}). The case when $\beta\in\left(  -\frac{\pi^{2}}%
{2},0\right)  $ can be treated similarly.
\end{proof}

\begin{remark}
\label{c-star-unique-and-beta-sign} By the index formula
(\ref{index-formula-sinus}) and the above proof, the lower bound of unstable
wave numbers $\Lambda_{\beta}=\sup_{c\notin(0,1)}\lambda_{\beta}^{-}$ is
achieved at exactly one point $c^{\ast}\in(-\infty,0]$ for $\beta\in\left(
0,\frac{\pi^{2}}{2}\right)  $ and $c^{\ast}\in\lbrack1,\infty)$ for $\beta
\in\left(  -\frac{\pi^{2}}{2},0\right)  $. Moreover, for each $\alpha^{2}%
\in(0,\Lambda_{\beta}]$, when $\beta\in(0,\frac{\pi^{2}}{2})$, there exists
exactly one point $c_{1}\in(-\infty,0]$ satisfying $\lambda_{\beta}^{-}\left(
c_{1}\right)  =\alpha^{2}$ and $\partial_{c}\lambda_{\beta}^{-}\left(
c_{1}\right)  \geq0$; when $\beta\in(-\frac{\pi^{2}}{2},0)$, there exists
exactly one point $c_{1}\in\lbrack1,+\infty)$ satisfying $\lambda_{\beta}%
^{-}\left(  c_{1}\right)  =\alpha^{2}$ and $\partial_{c}\lambda_{\beta}%
^{-}\left(  c_{1}\right)  \leq0$. Thus, when $\beta>0$, the function
$\lambda_{\beta}^{-}\left(  c\right)  $ $\left(  c\in(-\infty,0]\right)  \ $is
either  increasing with $c^{\ast}=0$ or single humped at $c^{\ast}%
\in\left(  -\infty,0\right)  $ (i.e. increasing in $\left(  -\infty,c^{\ast
}\right)  $ and decreasing in $\left(  c^{\ast},0\right)  $). Similarly, when
$\beta<0$, the function $\lambda_{\beta}^{-}\left(  c\right)  $ $\left(
c\in\lbrack1,+\infty)\right)  \ $is either decreasing with $c^{\ast
}=1$ or single humped at $c^{\ast}\in\left(  1,+\infty\right)  $. These two
cases (monotone or single humped) are determined by the information of
$\lambda_{\beta}^{-}\left(  c\right)  $ at the boundary points $c=0$ or
$1\ $as in the following remark.
\end{remark}

\begin{remark}
\label{rm-42} Assume $\Lambda_{\beta}>0$. When $\beta\in\left(  -\frac{\pi
^{2}}{2},0\right)  $, we divide it into three cases:

\textbf{Case 1.} $\lambda_{\beta}^{-}(1)=0$. In this case, $\lambda_{\beta
}^{-}\left(  c\right)  $ is single humped and $\Lambda_{\beta}=\sup
_{c\in\lbrack1,+\infty)}\lambda_{\beta}^{-}\left(  c\right)  $ is achieved in
the interior $(1,\infty)$.

\textbf{Case 2.} $\lambda_{\beta}^{-}(1)>0$ and $\lambda_{\beta}^{-^{\prime}%
}(1^{+})\leq0$. Then $\lambda_{\beta}^{-}\left(  c\right)  $ is
decreasing in $[1,+\infty)\ $and $\Lambda_{\beta}=\lambda_{\beta}^{-}(1)$.

\textbf{Case 3. }$\lambda_{\beta}^{-}(1)>0$ and $\lambda_{\beta}^{-^{\prime}%
}(1^{+})>0$. Then $\lambda_{\beta}^{-}\left(  c\right)  $ is single humped and
$\Lambda_{\beta}=\sup_{c\in\lbrack1,+\infty)}$ $\lambda_{\beta}^{-}\left(
c\right)  >\lambda_{\beta}^{-}(1)$ is achieved in the interior $(1,\infty)$.

Similarly, when $\beta\in\left(  0,\frac{\pi^{2}}{2}\right)  $, we have three
cases:

\textbf{Case 1.} $\lambda_{\beta}^{-}(0)=0$. Then $\lambda_{\beta}^{-}\left(
c\right)  $ is single humped and $\Lambda_{\beta}=\sup_{c\in(-\infty
,0]}\lambda_{\beta}^{-}\left(  c\right)  $ is achieved in the interior
$(-\infty,0)$.

\textbf{Case 2.} $\lambda_{\beta}^{-}(0)>0$ and $\lambda_{\beta}^{-^{\prime}%
}(0^{-})\geq0$. Then $\lambda_{\beta}^{-}\left(  c\right)  $ is
increasing in $(-\infty,0]\ $and $\Lambda_{\beta}=\lambda_{\beta}^{-}(0)$.

\textbf{Case 3. }$\lambda_{\beta}^{-}(0)>0$ and $\lambda_{\beta}^{-^{\prime}%
}(0^{-})<0$. Then $\lambda_{\beta}^{-}\left(  c\right)  $ is single humped and
$\Lambda_{\beta}=\sup_{c\in(-\infty,0]}$ $\lambda_{\beta}^{-}\left(  c\right)
>\lambda_{\beta}^{-}(0)$ is achieved in the interior $(-\infty,0)$.
\end{remark}

The computation of $\lambda_{\beta}^{-}$ $\ $and $\lambda_{\beta}^{-^{\prime}%
}$ at endpoints $0,1$ is done in Corollary \ref{value-lambda-} and
Proposition \ref{derivative}. Then by Proposition \ref{lower transition} and
Remarks \ref{c-star-unique-and-beta-sign}--\ref{rm-42}, we are in a position
to give the sharp stability boundary for Sinus flow.

\begin{theorem}
\label{transition} Let $\beta\in(-\frac{\pi^{2}}{2},\frac{\pi^{2}}{2})$. Then
the Sinus flow is linearly unstable if and only if $\alpha^{2}\in
(\Lambda_{\beta},3\pi^{2}/4)$. The lower bound $\Lambda_{\beta}$ for unstable
wave numbers is described as follows: there exist $\beta_{-}\in\left(
-\frac{\pi^{2}}{2},0\right)  $ and $\beta_{+}=\frac{\sqrt{3}-1}{4}\pi^{2}%
\in\left(  0,\frac{\pi^{2}}{2}\right)  $ such that

\begin{enumerate}
\item for $\beta\in(-\frac{\pi^{2}}{2},\beta_{-})$, $\Lambda_{\beta}%
=\lambda_{\beta}^{-}(c^{\ast})>0$ for some $c^{\ast}\in(1,\infty)$;

\item for $\beta\in[\beta_{-}, 0]$, $\Lambda_{\beta}= 0$;

\item for $\beta\in(0,\beta_{+}]$,%
\[
\Lambda_{\beta}=\lambda_{\beta}^{-}(0)=\pi^{2}\left(  1-\left(  \sqrt
{-\frac{\beta}{\pi^{2}}+\frac{9}{16}}+\frac{1}{4}\right)  ^{2}\right)  >0;
\]

\item for $\beta\in(\beta_{+},\frac{\pi^{2}}{2})$, $\Lambda_{\beta}%
=\lambda_{\beta}^{-}(c^{\ast})>\lambda_{\beta}^{-}(0)>0$ for some $c^{\ast}%
\in(-\infty,0)$.
\end{enumerate}
\end{theorem}

The following Figure \ref{fig:plot-thm42} illustrates these cases.

\begin{figure}[th]
\centering
\begin{tabular}
[c]{cc}%
\includegraphics[width=2.7in]{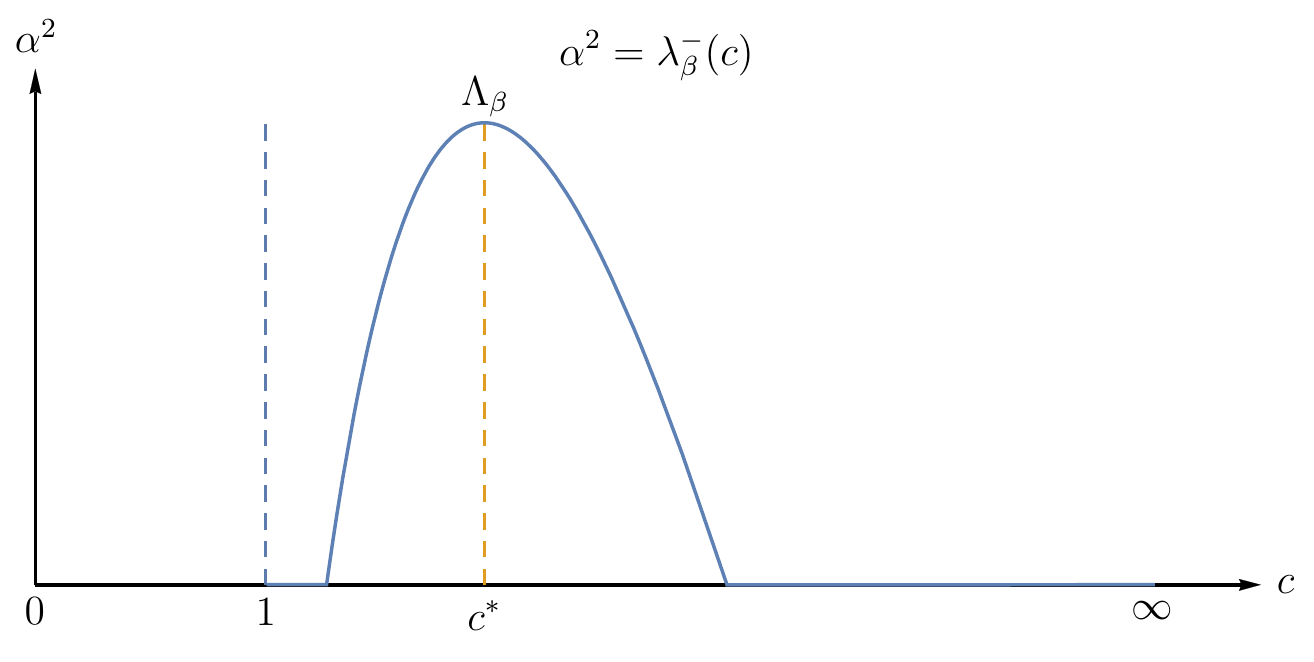} &
\includegraphics[width=2.7in]{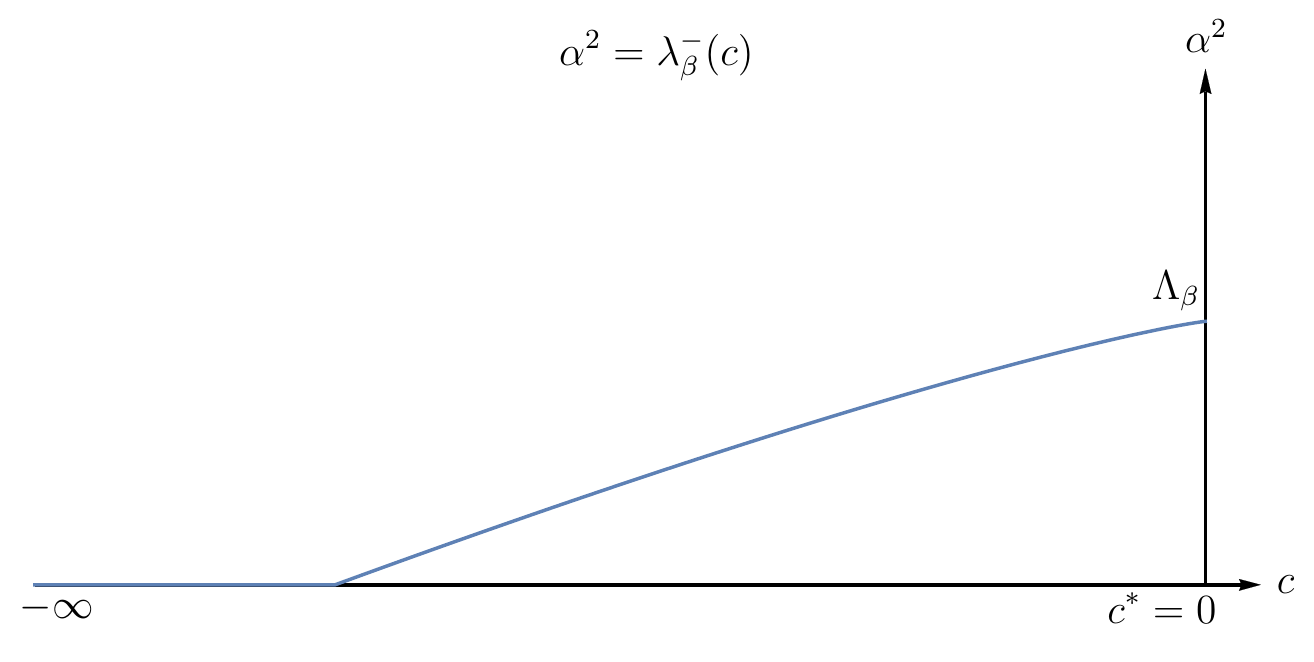}\\
(1) Case $\beta\in(-\frac{\pi^{2}}{2}, \beta_{-})$ & (2) Case $\beta\in(0,
\beta_{+}]$\\
\includegraphics[width=2.7in]{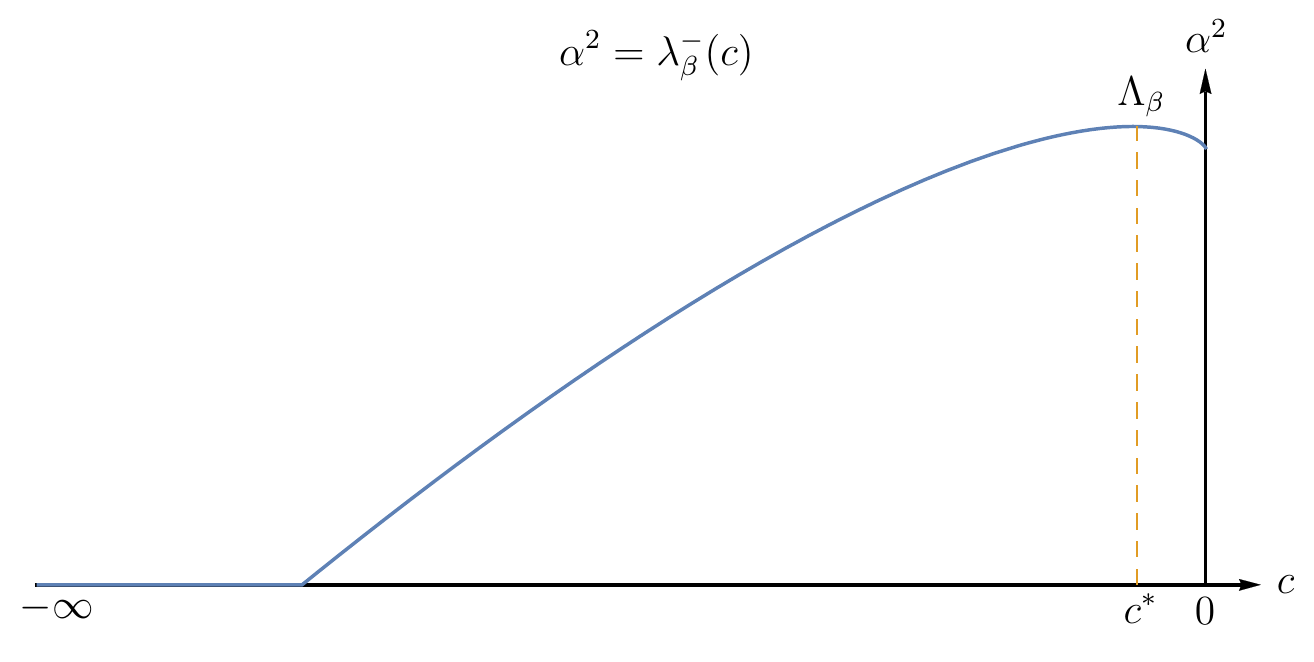} &
\includegraphics[width=2.7in]{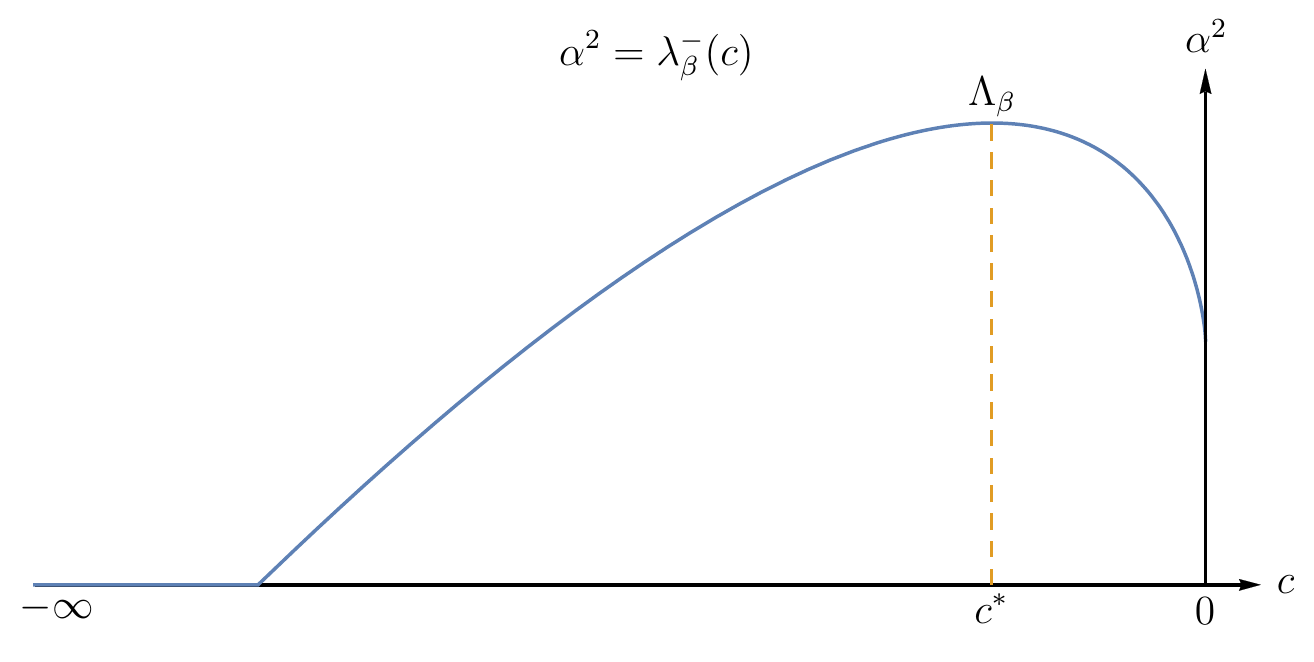}\\
(3) Case $\beta\in(\beta_{+}, \frac{5\pi^{2}}{16})$ with $\lambda_{\beta
}^{-\prime}(0^{-})$ finite & (4) Case $\beta\in[\frac{5\pi^{2}}{16}, \frac
{\pi^{2}}{2})$ with $\lambda_{\beta}^{-\prime}(0^{-})$ infinite
\end{tabular}
\caption{Graph of $\lambda_{\beta}^-(c)$ for several fixed $\beta$}%
\label{fig:plot-thm42}%
\end{figure}

\begin{proof}
The upper bound $3\pi^{2}/4$ for unstable wave numbers is given in Theorem
\ref{th31}. Below we determine the lower bound $\Lambda_{\beta}$. First,
consider $\beta\in(0,\frac{\pi^{2}}{2})$. By Corollary \ref{value-lambda-},
$\lambda_{\beta}^{-}(0)>0$. By Proposition \ref{derivative}, $\lambda_{\beta
}^{-\prime}(0^{-})=-\lambda_{\beta}^{\prime}(0^{-})\geq0$ for $\beta
\in(0,\beta_{+}]$, and $\lambda_{\beta}^{-\prime}(0^{-})=-\lambda_{\beta
}^{\prime}(0^{-})<0$ for $\beta\in(\beta_{+},\frac{\pi^{2}}{2})$. Then
conclusions 3 and 4 follow from Remark \ref{rm-42}. For $\beta=0$,
$\Lambda_{0}=0$ since the interval of unstable wave numbers is $\left(
0,\sqrt{3}\pi/2\right)  $ (\cite{Lin2003}).

For $\beta\in(-{\frac{\pi^{2}}{2}},0)$, we first determine whether
$\Lambda_{\beta}>0$ or not. To this end, we consider $\inf_{c\in
\lbrack1,\infty)}\lambda_{\beta}(c)$ as a function of $\beta$ and obverse that
it is non-decreasing in $\beta\in\lbrack-{\frac{\pi^{2}}{2}},0].$ In fact,
$\lambda_{\beta}(c)$ is strictly increasing in $\beta$ for any given
$c\in\lbrack1,\infty)$ by Corollary \ref{co21} and Proposition \ref{pro-41}
(4). Furthermore,
\[
\inf_{c\in\lbrack1,\infty)}\lambda_{-\frac{\pi^{2}}{2}}(c)=\lambda_{-\frac
{\pi^{2}}{2}}\left(  1\right)  =-\frac{3}{4}\pi^{2}<0,\;\;\;\;\inf
_{c\in\lbrack1,\infty)}\lambda_{0}(c)>0.
\]
Here, the second inequality is true because $\lambda_{0}(c)>0$ for any
$c\in\lbrack1,\infty]$ by (4.8) in \cite{Tung1981} and Proposition
\ref{pro-41}, and $\lambda_{0}$ is continuous on $c\in\lbrack1,\infty]$ by
Lemma \ref{lem-cinfty} and Proposition \ref{prop-singular limit}. Hence by
continuity, there exists $\beta_{-}\in\left(  -\frac{\pi^{2}}{2},0\right)  $
such that
\[
\Lambda_{\beta}=\sup_{c\in\lbrack1,\infty)}\lambda_{\beta}^{-}(c)=-\min
\{0,\inf_{c\in\lbrack1,\infty)}\lambda_{\beta}(c)\}>0
\]
if $\beta\in(-{\frac{\pi^{2}}{2}},\beta_{-})$ and $\Lambda_{\beta}=0$ if
$\beta\in\left[  \beta_{-},0\right]  $. This proves conclusion 2. By Corollary
\ref{value-lambda-}, $\lambda_{\beta}^{-}(1)=0$, thus conclusion 1 follows
from Remark \ref{rm-42}.
\end{proof}

\begin{remark}
\label{numerical beta-} By numerical calculation, $\beta_{-}\approx
-0.41224\pi^{2}\approx-4.06867$.
\end{remark}

\begin{remark}
\label{boundary-strictly decreasing} The lower bound $\Lambda_{\beta}$ is
strictly decreasing to $\beta\in(-\frac{\pi^{2}}{2},\beta_{-})$. In fact, for
any $-\frac{\pi^{2}}{2}<\beta_{2}<\beta_{1}<\beta_{-}$, there exists $c_{1}%
\in(1,\infty)$ such that $\Lambda_{\beta_{1}}=-\lambda_{\beta_{1}}(c_{1})>0$,
and by Corollary \ref{co21} (ii), we have $\lambda_{\beta_{2}}(c_{1}%
)<\lambda_{\beta_{1}}(c_{1})$. This gives $\Lambda_{\beta_{2}}>\Lambda
_{\beta_{1}}$.
\end{remark}

\begin{remark}
\label{rm-general} Consider a general class $\mathcal{K^{+}}$ flow $U\left(
y\right)  $. For $\beta\in(\min U^{\prime\prime},\max U^{\prime\prime})$, we
define $\sqrt{\Lambda_{\beta}}\geq0$ to be the supremum  of wave numbers for neutral
modes with non-positive $L_{\alpha}\ $signature. Equivalently, $\Lambda
_{\beta}$ is the maximum of $\sup_{c\notin {\rm Ran}\,(U)}\lambda_{\beta}^{-}(c)$
(defined as in the Sinus flow) and the negative part of eigenvalues of
$-\frac{d^{2}}{dy^{2}}-\frac{\beta-U^{\prime\prime}}{U-c}$ ($c=U\left(
y_{1}\right)  ,U\left(  y_{2}\right)  $ or a critical value of $U$) with
non-positive signature. Then $\sqrt{\Lambda_{\beta}}<\alpha_{\max}$ (defined in
Subsection 3.1) and there is linear instability for $\alpha
\in\left(  \sqrt{\Lambda_{\beta}},\alpha_{\max}\right)  $. Indeed, $\sqrt{\Lambda_{\beta
}}\geq\alpha_{\max}$ would imply that $k_{i}^{\leq0}\geq1\ $for $\alpha
=\sqrt{\Lambda_{\beta}}\ $which is a contradiction to the index formula
(\ref{index-formula}) and the fact that $n^{-}\left(  L_{\alpha}\right)  =0$
for $\alpha\geq\alpha_{\max}$. The linear instability again follows from the
index formula since when $\alpha\in\left(  \sqrt{\Lambda_{\beta}},\alpha_{\max
}\right)  $,  we have $k_{i}^{\leq0}=0$ and $n^{-}\left(  L_{\alpha}\right)
>0$.

Moreover, the interval $\left(  \sqrt{\Lambda_{\beta}},\alpha_{\max}\right)  $ gives
the sharp range of unstable wave numbers if the flow shares the properties of
Sinus flow. More precisely, this is true for flows satisfying: i)
$n^{-}\left(  \tilde{L}_{0}\right)  =1$ ($\tilde{L}_{0}$ defined by
(\ref{defn-L0})); ii) the singular neutral modes only exist with $c$ to be the
endpoints of ${\rm Ran}\,(U)$; iii) $E(\mathcal{L}_{\beta,U_{\max}})$ and $E(\mathcal{L}_{\beta,U_{\min}})$ are bounded from below, where
$E(\mathcal{L}_{\beta,U_{\min}})$ denotes  the set of all the eigenvalues of  $\mathcal{L}_{\beta,U_{\min}}$; iv)
the weak continuity of the principal eigenvalues holds in the sense that
$\lim_{c\rightarrow U_{\min}^{-}}\lambda_{\beta}(c)=\inf E(\mathcal{L}%
_{\beta,U_{\min}})$ and $\lim_{c\rightarrow U_{\max}^{+}}\lambda_{\beta
}(c)=\inf E(\mathcal{L}_{\beta,U_{\max}})$. The last condition is not
required if $\inf E(\mathcal{L}_{\beta,U_{\min}})>0$ and $\liminf
_{c\rightarrow U_{\min}^{-}}\lambda_{\beta}(c)>0$, and similarly for $U_{\max
}$.
\end{remark}

\subsection{Eigenvalue computation}

\label{Eigenvalue computation}

In this subsection we explicitly solve all the eigenvalues of \eqref{eq44}
with $c=U_{\beta},0,1,\pm\infty$. Then we compute the derivative of
$\lambda_{\beta}$ when $\lambda_{\beta}^{-}>0$.

\begin{figure}[th]
\centering
\includegraphics[width=2.9in]{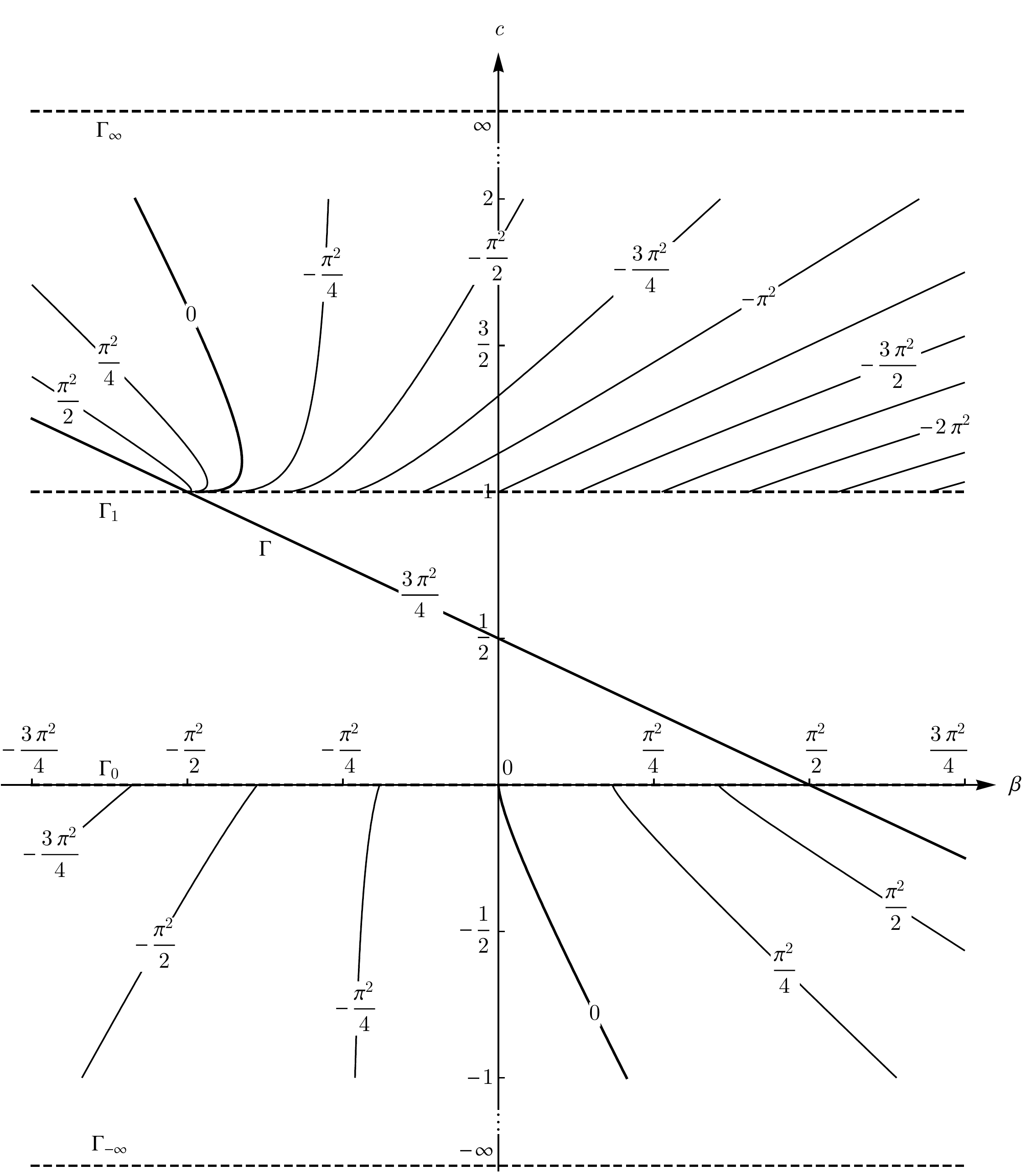} \caption{Contour plot for
$-\lambda_{1} (\beta, c)$}%
\label{fig:contour}%
\end{figure}

To see how $\lambda_{n}$'s behave as a function of $\beta$ and $c$, we first
present the numerical contour plots in Figure $\ref{fig:contour}$ for
different values of $-\lambda_{1}$. The abscissa axis is $\beta$, and the
vertical axis is $c$. The numbers on the contours are values of $-\lambda_{1}%
$. Several important lines divide the $\beta$-$c$ plane into pieces.

\begin{itemize}
\item The slant $\Gamma$ represents the line $c = U _{\beta}$ (regular neutral
wave speed).

\item The dotted line $\Gamma_{0}$ represents $c = 0$ and $\Gamma_{1}$
represents $c = 1$ (singular neutral wave speeds).

\item We denote the open region above $\Gamma$ and $\Gamma_{1}$ by $K _{1}$,
and denote the open region below $\Gamma$ and $\Gamma_{0}$ by $K _{0}$. We add
two dotted lines $\Gamma_{\pm\infty}$ at infinite far away for $c \to\pm
\infty$.
\end{itemize}

We now compute the spectrum of $\mathcal{L}_{\beta,c}$ with $(\beta,c)$ on the
boundary of $K_{0}$ and $K_{1}$.

\begin{proposition}
\label{pro-41} Let $(\beta,c)\in\Gamma\cup\Gamma_{0}\cup\Gamma_{1}\cup
\Gamma_{\pm\infty}$. Then  all the eigenvalues of $\sigma(\mathcal{L}_{\beta,c})$  are listed as follows.

{\rm(1)} For $\beta\in\mathbf{R}$ $(\text{on}\ \Gamma)$, $\lambda_{n} (\beta,U
_{\beta})= \left(  {\frac{n ^{2} }{4 }}- 1 \right)  \pi^{2}$ and $\phi
_{n}^{(\beta,U _{\beta})}(y)=\sin({\frac{n\pi}{2}}(y+1))$, $n\geq1$.

{\rm(2)} For $\beta\in\mathbf{R}$ $(\text{on}\ \Gamma_{\pm\infty})$, $\lambda
_{n}(\beta,\pm\infty)= {\frac{n ^{2} }{4}} \pi^{2}$ and $\phi_{n}^{(\beta
,\pm\infty)}(y)=\sin({\frac{n\pi}{2}}(y+1))$, $n\geq1$.

{\rm(3)} For $\beta< {\frac{9\pi^{2} }{16}}$ $(\text{on}\ \Gamma_{0})$,
$\lambda_{n}(\beta,0) = \left(  \left(  \gamma+ \frac{n-1}{2} \right)  ^{2} -
1 \right)  \pi^{2}$ and $\phi_{n}^{(\beta,0)}(y)=\cos^{2\gamma}({\frac{\pi}%
{2}}y)$ $P_{n-1}(\sin({\frac{\pi}{2}}y))$, $n\geq1$, where $\gamma= \frac
{1}{4}+ \sqrt{-\frac{\beta}{\pi^{2}} + \frac{9}{16}}$ and $P_{n-1}(\cdot)$ is
a polynomial with order $n-1$.

{\rm(4)} For $\beta> -{\frac{9\pi^{2} }{16}}$ $(\text{on}\ \Gamma_{1})$,
$\lambda_{n} (\beta,1)= \left(  \left(  \tilde\gamma-{\frac{1}{2}} +
\lceil\frac{n}{2} \rceil\right)  ^{2} - 1 \right)  \pi^{2}$ and $\phi
_{n}^{(\beta,1)}(y) = \text{sign}(y)$ $|\sin({\frac{\pi}{2}}y)|^{2\tilde
\gamma}$ $P_{n-1}(\cos({\frac{\pi}{2}}y))$ when $n$ is even; $\phi_{n}%
^{(\beta,1)}(y) = $ $|\sin({\frac{\pi}{2}}y)|^{2\tilde\gamma}$ $P_{n}%
(\cos({\frac{\pi}{2}}y)) $ when $n$ is odd, $n\geq1$, where $\tilde\gamma=
\frac{1}{4}+ \sqrt{\frac{\beta}{\pi^{2}} + \frac{9}{16}}$.\\
Here $\phi_{n}^{(\beta,c)}$ is, up to a nonzero constant, the corresponding
eigenfunction of $\lambda_{n} (\beta,c)$. Consequently, the set of all the eigenvalues of $\sigma(\mathcal{L}_{\beta,c})$ is bounded from below. Moreover, we have

{\rm(5)} $\sigma_e(\mathcal{L}_{\beta,0})=\emptyset$ if $\beta\in(-\infty,{5\pi^2\over 16}]$; $\sigma_e(\mathcal{L}_{\beta,1})=\emptyset$ if $\beta\in[-{5\pi^2\over16},+\infty)$.
\end{proposition}

\begin{proof}
The computation of  eigenvalues in (1) and (2) is straightforward.
\eqref{eq44} is singular when $c=0$ and $c=1$. We  solve all the
eigenvalues in (3) and (4) by transforming \eqref{eq44} into two
hypergeometric equations as follows.



\noindent\textbf{Case $c=0$.} The equation \eqref{eq44} becomes
\begin{equation}
-\phi^{\prime\prime}-\frac{\beta+\frac{\pi^{2}}{2}\cos(\pi y)}{\cos^{2}\left(
\frac{\pi y}{2}\right)  }\phi=\lambda\phi,\;\;\phi(\pm1)=0.\label{eqn-0}%
\end{equation}
We make a change of variable for $y\in(0,1)$. Set $z=\cos^{2}\left(  \frac{\pi
y}{2}\right)  \in(0,1)$ so $y=\frac{2\arccos\sqrt{z}}{\pi}$. Define
\[
\psi(z):=\phi\left(  \frac{2\arccos\sqrt{z}}{\pi}\right)  =\phi(y).
\]
Then the equation for $\psi$ is
\begin{equation}
\frac{\pi^{2}}{2}\left[  \left(  -2z+2z^{2}\right)  \psi^{\prime\prime
}(z)+(-1+2z)\psi^{\prime}(z)\right]  -\frac{\beta+\frac{\pi^{2}}{2}\left(
2z-1\right)  }{z}\psi(z)=\lambda\psi(z).\label{eqn-z}%
\end{equation}
Suppose $\psi(z)=z^{\gamma}G(z)$ for some $G(z)$ where $\gamma$ is a constant
to be determined. Then the equation for $G(z)$ is
\begin{align}
z(1-z)G^{\prime\prime}(z) &  +\left[  \frac{1}{2}+2\gamma-\left(
1+2\gamma\right)  z\right]  G^{\prime}(z)\nonumber\\
&  -\left[  \frac{-\frac{\beta}{\pi^{2}}+\frac{1}{2}-\gamma\left(
\gamma-\frac{1}{2}\right)  }{z}+\gamma^{2}-\left(  1+\frac{\lambda}{\pi^{2}%
}\right)  \right]  G(z)=0.\label{G}%
\end{align}
Set $\gamma=\frac{1}{4}+\sqrt{-\frac{\beta}{\pi^{2}}+\frac{9}{16}}$, so for
$\beta<{9\pi^{2}\over16}$, $\gamma>\frac{1}{4}$. Denote $r=\sqrt{1+\frac{\lambda}%
{\pi^{2}}}$, which could be purely imaginary, but non-negative. (\ref{G}) is
simplified to
\begin{equation}
z(1-z)G^{\prime\prime}(z)+\left[  \frac{1}{2}+2\gamma-\left(  1+2\gamma
\right)  z\right]  G^{\prime}(z)-\left(  \gamma+r\right)  \left(
\gamma-r\right)  G(z)=0.\label{hypergeo}%
\end{equation}
It is the Euler's hypergeometric differential equation
\[
z(1-z)G^{\prime\prime}(z)+\left[  c-(a+b+1)z\right]  G^{\prime}(z)-abG(z)=0
\]
for $a=\gamma-r,b=\gamma+r,c=\frac{1}{2}+2\gamma$. A nontrivial solution of
(\ref{hypergeo}) is
\begin{align}\label{G1Z}
G_{1}(z)=\ _{2}F_{1}\left(  \gamma-r,\gamma+r;\frac{1}{2}+2\gamma;z\right)  .
\end{align}
Here
\begin{align}
_{2}F_{1}(a,b;c;z) &  =\sum_{n=0}^{\infty}{\frac{(a)_{n}(b)_{n}}{(c)_{n}}%
}{\frac{z^{n}}{n!}},\label{F}\\
(q)_{n}=1\ \text{if}\ n &  =0,\ (q)_{n}=q(q+1)\cdots(q+n-1)\ \text{if}%
\ n>0.\label{qn}%
\end{align}
$_{2}F_{1}$ is analytic in $z\in(0,1)$ if we choose the branch cut to be
$\{z>1\}$. The corresponding solution of (\ref{eqn-z}) is
\begin{align}\label{psi1z}
\psi_{1}(z)=z^{\gamma}\ _{2}F_{1}\left(  \gamma-r,\gamma+r;\frac{1}{2}%
+2\gamma;z\right).
\end{align}
The other linearly independent solution to $\psi_{1}$ is
\[
\psi_{2}(z)=\psi_{1}(z)\int_{{\frac{1}{2}}}^{z}{{{\psi_{1}^{-2}(s)}}}%
e^{-\int_{{\frac{1}{2}}}^{s}{\frac{1}{t}}\cdot{\frac{1-2t}{2(1-t)}}dt}ds.
\]
Direct computation deduces $\lim\limits_{z\rightarrow0^{-}}|\psi
_{2}(z)|=\infty,$ while $\lim\limits_{z\rightarrow0^{-}}\psi_{1}(z)=0$.

Therefore the only possible solution to (\ref{eqn-0}) is
\[
\phi(y)=%
\begin{cases}
c_{1}\psi_{1}\left(  \cos^{2}\left(  \frac{\pi y}{2}\right)  \right)   &
y\in(0,1),\\
c_{2}\psi_{1}\left(  \cos^{2}\left(  \frac{\pi y}{2}\right)  \right)   &
y\in(-1,0).
\end{cases}
\]
Series expansion near $y=0$ gives
\[
\psi_{1}\left(  \cos^{2}\left(  \frac{\pi y}{2}\right)  \right)  =\frac
{\pi^{1/2}\Gamma\left(  2\gamma+\frac{1}{2}\right)  }{\Gamma\left(
\gamma-r+\frac{1}{2}\right)  \Gamma\left(  \gamma+r+\frac{1}{2}\right)
}-\frac{\pi^{3/2}\Gamma\left(  2\gamma+\frac{1}{2}\right)  }{\Gamma\left(
\gamma-r\right)  \Gamma\left(  \gamma+r\right)  }|y|+O\left(  y^{2}\right)  .
\]
Since (\ref{eqn-0}) is regular and $\phi$ is smooth at $y=0$, we infer that
the constant term and $|y|$ term cannot be non-zero simultaneously. Note that
$\gamma>\frac{1}{4}$ is real, and $r$ is non-negative or purely imaginary.
Since $\Gamma(z)$ only has poles at non-positive integers, either
$\gamma-r+\frac{1}{2}$ or $\gamma-r$ equals to a non-positive integer, so
$r=\gamma+\frac{n-1}{2}>\frac{1}{4}$ for a positive integer $n$. Therefore,
$\lambda_{n}(\beta,0)=\pi^{2}(r^{2}-1)=\pi^{2}\left(  \left(  \gamma
+\frac{n-1}{2}\right)  ^{2}-1\right)  $. Inserting (\ref{F})--(\ref{qn}) into
(\ref{G1Z}), direct computation gives the $\phi_{n}^{(\beta,0)}%
=\cos^{2\gamma}({\frac{\pi}{2}}y)P_{n-1}(\sin({\frac{\pi}{2}}y))$.


\noindent\textbf{Case $c=1$.} \eqref{eq44} is
\begin{align}\label{c=1explicit}
-\phi^{\prime\prime}-\frac{\beta+\frac{\pi^{2}}{2}\cos(\pi y)}{-\sin
^{2}\left(  \frac{\pi y}{2}\right)  }\phi=\lambda\phi,\;\;\phi(\pm1)=0.
\end{align}
We make a change of variable for $y\in(0,1)$. Set $z=\sin^{2}\left(  \frac{\pi
y}{2}\right)  \in(0,1)$ so $y=\frac{2\arcsin\sqrt{z}}{\pi}$. Define
\[
\psi(z):=\phi\left(  \frac{2\arcsin\sqrt{z}}{\pi}\right)  =\phi(y).
\]
Then the equation for $\psi$ is
\[
\frac{\pi^{2}}{2}\left[  \left(  -2z+2z^{2}\right)  \psi^{\prime\prime
}(z)+(-1+2z)\psi^{\prime}(z)\right]  -\frac{-\beta+\frac{\pi^{2}}{2}\left(
2z-1\right)  }{z}\psi(z)=\lambda\psi(z).
\]
This is almost the same as equation (\ref{eqn-z}), only with $-\beta$ in the
position of $\beta$. For $-\beta<{9\pi^{2}\over16}$, we set $\tilde{\gamma}=\frac
{1}{4}+\sqrt{\frac{\beta}{\pi^{2}}+\frac{9}{16}}>\frac{1}{4}$ and again denote
$r=\sqrt{1+\frac{\lambda}{\pi^{2}}}$. The solution to (\ref{c=1explicit}) is
\[
\phi(y)=%
\begin{cases}
c_{1}\psi_{1}\left(  \sin^{2}\left(  \frac{\pi y}{2}\right)  \right)  , &
y\in(0,1),\\
c_{2}\psi_{1}\left(  \sin^{2}\left(  \frac{\pi y}{2}\right)  \right)  , &
y\in(-1,0),
\end{cases}
\]
where $\psi_{1}(z)$ is defined as (\ref{psi1z}) with $\gamma$
replaced by $\tilde{\gamma}$.

To satisfy the boundary condition, $\phi(1) = c _{1} \psi(1) = 0$, plugging in
$y = 1$ we have
\begin{align*}
\psi_{1} \left(  \sin^{2} \left(  \frac{\pi}{2}\right)  \right)  = \psi_{1}
(1) = \ _{2} F _{1} \left(  \tilde\gamma- r, \tilde\gamma+ r; \frac{1}{2} +
2\tilde\gamma; 1 \right)  = \frac{\pi^{1/2} \Gamma\left(  2 \tilde\gamma
+\frac{1}{2}\right)  }{\Gamma\left(  \tilde\gamma-r+\frac{1}{2}\right)
\Gamma\left(  \tilde\gamma+r+\frac{1}{2}\right)  }.
\end{align*}
Therefore non-trivial eigenfunction exists if and only if $\tilde\gamma- r +
\frac{1}{2}$ equals to a non-positive integer $-m$, so $r = \tilde\gamma+
\frac{1}{2} + m > {3\over4}$ for a non-negative integer $m$. However, since $\psi_{1}
\left(  \sin^{2} \left(  \frac{\pi y}{2}\right)  \right)  = O(y ^{2
\tilde\gamma}) = o(y^{1\over2})$ as $y \to0$, there are two linearly independent
eigenfunctions in $H^{1}$, that is, $\phi(y) = \psi_{1} \left(  \sin^{2}
\left(  \frac{\pi y}{2}\right)  \right)  $ and $\phi(y) = \mathrm{sign} (y)
\psi_{1} \left(  \sin^{2} \left(  \frac{\pi y}{2}\right)  \right)  $.
Therefore, $\lambda_{n} (\beta,1)= \pi^{2} (r ^{2} - 1) = \pi^{2} \left(
\left(  \tilde\gamma- \frac{1}{2} + \lceil\frac{n}{2} \rceil\right)  ^{2} - 1
\right)  $.

Finally, we show that $\sigma_{e}(\mathcal{L}_{\beta,0})=\emptyset$ if $\beta\in(-\infty,{{5\pi^{2}}\over{16}%
}]$. Since the two
linearly independent solutions of (\ref{eqn-0}) with $\lambda_{1}(\beta
,0)=\pi^{2}\left(  \gamma^{2}-1\right)  $ are
\begin{equation}
\phi_{0}(y):=\phi_{1}^{(\beta,0)}(y)=\cos^{2\gamma}\left(  {{\pi y}/{2}%
}\right)  ,\;\;\phi_{1}(y)=\cos^{2\gamma}\left(  {{\pi y}/{2}}\right)
\int_{0}^{y}{{\cos^{-4\gamma}({{\pi s}/{2}})}}ds.\label{defn-phi-0-1}%
\end{equation}
Then $\phi_{0}\in L^{2}(-1,1)$, and $\phi_{1}\notin L^{2}(-1,1)$. The eigenvalue problem (\ref{eqn-0}) is
in the limit point cases at $\pm1$.
In the limit point cases, we get by Remark 10.8.1 in \cite{Zettl2005} that the
starting point of the essential spectrum, i.e. $\sigma_{0}=\inf\sigma
_{e}(\mathcal{L}_{\beta,0})$, is exactly the oscillation point of
(\ref{eqn-0}). More precisely, (\ref{eqn-0}) is non-oscillatory for
$\lambda<\sigma_{0}$ and (\ref{eqn-0}) is oscillatory for $\lambda>\sigma_{0}%
$. Since $\lambda_{n}(\beta,0)\rightarrow\infty$ as $n\rightarrow\infty$ and
$\phi_{n}^{(\beta,0)}=\cos^{2\gamma}({\frac{\pi}{2}}y)P_{n-1}(\sin({\frac{\pi
}{2}}y))$ has finite zeros on $(-1,1)$, we get $\sigma_{0}=\infty$ and thus
$\sigma_{e}(\mathcal{L}_{\beta,0})=\emptyset$ for $\beta\in(-\infty,{{5\pi
^{2}}\over{16}}].$ The proof of $\sigma_e(\mathcal{L}_{\beta,1})=\emptyset$ when $\beta\in[-{5\pi^2\over16},+\infty)$ is similar by
considering the half interval $y \in(0,1)$.
\end{proof}


Proposition \ref{pro-41} indicates that for $c\in\lbrack0,1]$, there exist
exactly two families of neutral modes. One is a family of regular neutral
modes:
\[
c=U_{\beta}={{1}/{2}}-{{\beta}/{\pi^{2}}},\;\alpha={{\sqrt{3}\pi}/{2}}%
,\;\beta\in\mathbf{R},\;\phi(y)=\cos({{\pi y}/{2}}),
\]
the other is a curve of singular neutral modes (SNM curve) for $c=0$ and
$\beta\in\left(  0,{{\pi^{2}}/{2}}\right)  $:
\begin{equation}
c=0,\;\alpha=\pi\sqrt{1-\gamma^{2}},\;\beta=\pi^{2}(-\gamma^{2}+{{\gamma}/{2}%
}+{{1}/{2}}),\;\phi_{0}(y)=\cos^{2\gamma}\left(  {{\pi y}/{2}}\right)
,\label{37}%
\end{equation}
with $\gamma\in\left(  {{1}/{2}},1\right)  $ defined in Proposition
\ref{pro-41}. They were also found by Kuo \cite{Kuo1974}.

\begin{remark}
\label{Kuo1} Based on numerical results, Kuo in \cite{Kuo1974} claimed that
the above SNM curve (\ref{37}) is the lower stability boundary for $\beta
\in\left(  0,{{\pi^{2}}/{2}}\right)  $. More precisely, the Sinus flow is
linearly unstable when $\alpha\in(\pi\sqrt{1-\gamma^{2}},\sqrt{3}\pi/2)$ and
it is linearly stable when $\alpha\in(0,\pi\sqrt{1-\gamma^{2}}]$. The same
stability picture for the Sinus flow also appeared in \cite{Pedlosky1987}. By
Lemma \ref{le26}, any neutral limiting solution must lie in $H^{2}(-1,1)$.
However, $\phi_{0},\phi_{1}\notin H^{2}(-1,1)$ and $\lim_{y\rightarrow\pm
1}|\phi_{1}(y)|=\infty$ when $\beta\in\lbrack{{5\pi^{2}}/{16}},{{\pi^{2}}/{2}%
})$, where $\phi_{0},\phi_{1}$ are defined in (\ref{defn-phi-0-1}). This
implies that (\ref{37}) cannot be a neutral limiting mode, and therefore,
Kuo's claim on the stability boundary is incorrect, at least for the case
$\beta\in\lbrack{{5\pi^{2}}/{16}},{{\pi^{2}}/{2}})$.
\end{remark}

Recall that $\lambda_{\beta}( c)=\lambda_{1}({\beta},c)$ and $\lambda_{\beta}
^{-}( c)= \max\{-\lambda_{\beta}( c), 0\}$. By Proposition \ref{pro-41}, we
get the value of $\lambda_{\beta}^{-} $ on the boundaries.

\begin{corollary}
\label{value-lambda-} The value of $\lambda_{\beta}^{-}$ at the points
$c=0,1,\pm\infty$ is given as follows.

\begin{enumerate}
\item $\lambda_{\beta}^{-} (1) =0$ for $\beta>-\frac{\pi^{2}}{2}$, and
$\lambda_{\beta}^{-} (\pm\infty) = 0$ for $\beta\in\mathbf{R}$.

\item $\lambda_{\beta}^{-} (0) = \pi^{2} \left(  1-\left(  \sqrt{-\frac{\beta
}{\pi^{2}} + \frac{9}{16}} +\frac{1}{4} \right)  ^{2} \right)  > 0$ for
$0<\beta<\frac{\pi^{2}}{2}$, while $\lambda_{\beta}^{-} (0) = 0$ for
$\beta\leq0$.
\end{enumerate}
\end{corollary}

Since $\lambda_{\beta}^{-} (0) > 0$ for $\beta\in(0,\frac{\pi^{2}}{2}) $, we
need to compute the left derivative of $\lambda_{\beta}$ at $0$ according to
Remark \ref{rm-42}. Here we assume the left continuity of $\lambda_{\beta}$ at
$0$, which will be verified in Proposition \ref{prop-singular limit}.

\begin{proposition}
\label{derivative}
For $\beta\in(0,\frac{\pi^{2}}{2})$, we have
\[
{\frac{\partial\lambda_{1}}{\partial c}}(\beta,0^{-})=\lambda_{\beta}^{\prime}(0^-)=\frac{\pi^{2}\gamma(\gamma-1)\left(  \gamma^{2}-\frac{3}%
{4}\right)  }{\left(  \gamma-\frac{1}{4}\right)  \left(  \gamma-\frac{3}%
{4}\right)  }%
\begin{cases}
\leq0 & \frac{\sqrt{3}}{2}\leq\gamma<1\Leftrightarrow0<\beta\leq\frac{\sqrt
{3}-1}{4}\pi^{2},\\
>0 & \frac{3}{4}<\gamma<\frac{\sqrt{3}}{2}\Leftrightarrow\frac{\sqrt{3}-1}%
{4}\pi^{2}<\beta<\frac{5}{16}\pi^{2},
\end{cases}
\]
and $\lambda_{\beta}^{\prime}(0^-)=+\infty$ when $\frac{1}{2}<\gamma\leq\frac
{3}{4}\Leftrightarrow\frac{{}}{{}}\frac{5}{16}\pi^{2}\leq\beta<\frac{1}{2}%
\pi^{2}$, where $\gamma$ is defined in Proposition \ref{pro-41}.
\end{proposition}

\begin{proof}

The eigenfunction for $\lambda_{\beta}(0)$ is $\phi(y) = \frac{1}{\sqrt{C
_{4\gamma}}} \cos^{2\gamma} \left(  \frac{\pi y}{2} \right)  $, where
\begin{align*}
C _{s} = \int_{-1} ^{1} \cos^{s} \left(  \frac{\pi y}{2} \right)  dy =
\begin{cases}
\frac{2}{\sqrt{\pi}} \frac{\Gamma\left(  \frac{s + 1}{2} \right)  }%
{\Gamma\left(  \frac{s}{2} + 1 \right)  } & s > -1,\\
+\infty & s \leq-1,
\end{cases}
\end{align*}
by Beta function $B (p, q) = \frac{\Gamma(p)\Gamma(q)}{\Gamma(p + q)} = 2
\int_{0} ^{\frac{\pi}{2}} \sin^{2p - 1} (x) \cos^{2q - 1}( x ) \;\mathrm{d}
x$, $p, q > 0$. Therefore
\begin{align*}
\lambda_{\beta}^{\prime}(0^-) = -\int_{-1}^{1} {\frac{\beta-U^{\prime\prime}}{U
^{2}}} {{\phi}^{2}}\ dy  &  = -\frac{1}{C _{4\gamma}}\int_{-1}^{1} {\left(
\beta-{\pi^{2}}\left(  \cos^{2} \left(  \frac{\pi y}{2}\right)  - \frac{1}%
{2}\right)  \right)  } \cos^{4\gamma-4} \left(  \frac{\pi y}{2} \right)  dy\\
&  = \frac{1}{C _{4\gamma}} \left[  \left(  -\beta+ \frac{\pi^{2}}{2} \right)
C _{4\gamma- 4} - \pi^{2} C _{4\gamma- 2}\right]  .
\end{align*}
If ${\frac{1}{2}}<\gamma\leq\frac{3}{4}$, then $C _{4\gamma- 4} = +\infty$ and
$\lambda_{\beta}^{\prime}(0^-)= +\infty$. If $\gamma> \frac{3}{4}$, then
\begin{align*}
C _{4\gamma} = \frac{2}{\sqrt{\pi}} \frac{\Gamma\left(  2\gamma+ \frac{1}{2}
\right)  }{\Gamma\left(  2\gamma+ 1 \right)  } = \frac{2}{\sqrt{\pi}}
\frac{\Gamma\left(  2\gamma- \frac{1}{2} \right)  }{\Gamma\left(
2\gamma\right)  } \frac{2\gamma- \frac{1}{2}}{2\gamma} = \frac{2\gamma-
\frac{1}{2}}{2\gamma} C _{4\gamma- 2} = \frac{2\gamma- \frac{3}{2}}{2\gamma-
1} \frac{2\gamma- \frac{1}{2}}{2\gamma} C _{4\gamma- 4}.
\end{align*}
Therefore
\begin{align}
\lambda_{\beta}^{\prime}(0^-)= \frac{(2\gamma- 1)(2\gamma)}{\left(  2\gamma-
\frac{1}{2}\right)  \left(  2\gamma- \frac{3}{2}\right)  } \pi^{2} \left(
\gamma^{2} - \frac{1}{2}\gamma- \frac{2\gamma- \frac{3}{2}}{2\gamma-1}
\right)  = \frac{\pi^{2} \gamma(\gamma- 1)\left(  \gamma^{2} - \frac{3}%
{4}\right)  }{\left(  \gamma- \frac{1}{4}\right)  \left(  \gamma- \frac{3}%
{4}\right)  } .\nonumber
\end{align}

\end{proof}

\begin{remark}
\label{Kuo2} Proposition \ref{derivative} may seem to be counter-intuitive, as
Figure \ref{fig:contour} indicates that the partial derivative $\lambda_{\beta}^{\prime}(0^-)={\frac{\partial\lambda_{1}}{\partial c}}(\beta,0^{-})$
should be negative for all $\beta>0$, different from what we claimed in the
last three plots of Figure \ref{fig:plot-thm42}. However, if we zoom in near
the $\beta$-axis, numerical results will be consistent with Proposition
\ref{derivative}. Near $c=0$ and $\beta=0.1\pi^{2}$, $0.25\pi^{2}$ and
$0.4\pi^{2}$, we have the following contour plots. \begin{figure}[th]
\centering
\includegraphics[scale=0.4]{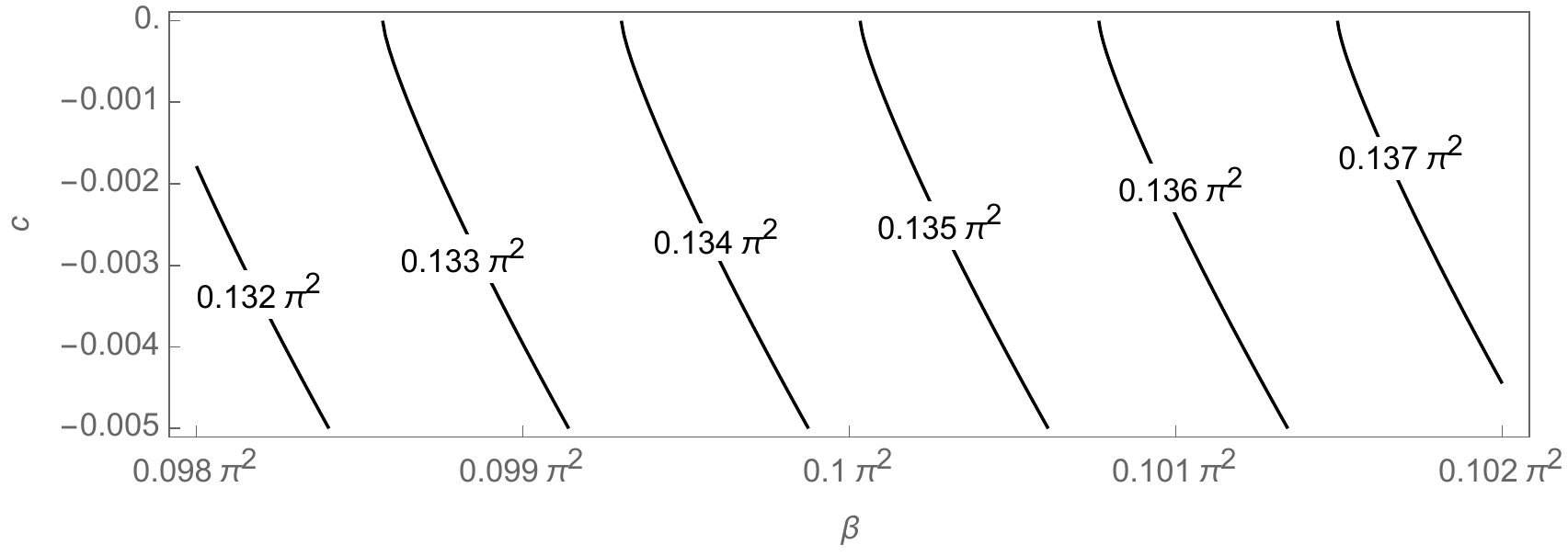} \includegraphics[scale=0.4]{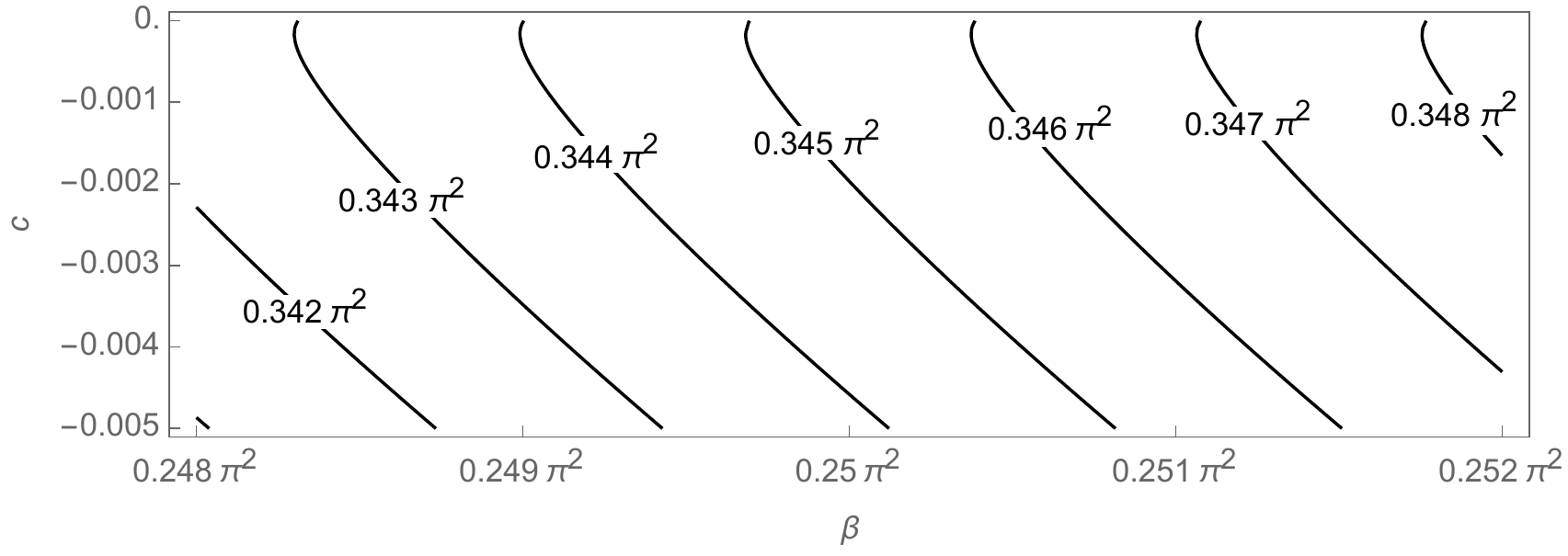}
\includegraphics[scale=0.4]{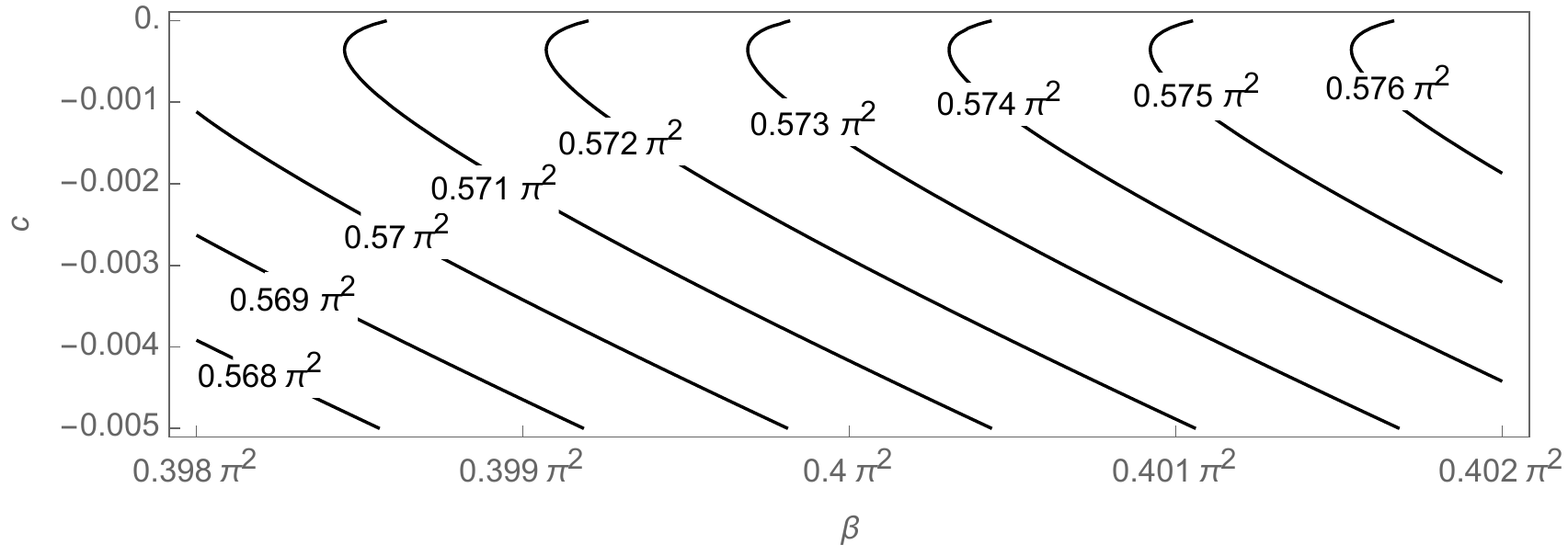}\caption{Contour plot for $-\lambda
_{1}(\beta,c)$ near $c=0$}%
\end{figure}

\noindent It can be seen that for $\beta$ near $0.25\pi^{2}\in\left(
\frac{\sqrt{3}-1}{4}\pi^{2},\frac{5}{16}\pi^{2}\right)  $, $\frac
{\partial\lambda_{1}}{\partial c}$ changes sign when $c$ is very close to $0$.
For $\beta$ near $0.4\pi^{2}\in\left(  \frac{5}{16}\pi^{2},\frac{1}{2}\pi
^{2}\right)  $, contours are tangent to $\beta$ axis, which indicates
$\frac{\partial\lambda_{1}}{\partial c}=\infty$ as $c\rightarrow0^{-}$.
Therefore, for $\beta>\frac{\sqrt{3}-1}{4}\pi^{2}$, $\lambda_{\beta}^{-}$ does
not attain its supremum at $c=0$, but at some $c^{\ast}<0$ which is really
close to $0$ (with about a distance smaller than $0.001$ based on the
observation from these plots). This may be the reason why Kuo failed to find
the correct lower stability boundary (\cite{Kuo1974}) for $\beta>\frac
{\sqrt{3}-1}{4}\pi^{2}$.

From the theoretical perspective, we compute the quadratic form $\langle
L_{\alpha}\omega_{\alpha},\omega_{\alpha}\rangle$ for the neutral modes
(\ref{37}), where $\omega_{\alpha}=-\phi^{\prime\prime}+\alpha^{2}\phi$. By
Lemma \ref{le-L-quadratic form} (i), direct computation shows
\[
\langle L_{\alpha}\omega_{\alpha},\omega_{\alpha}\rangle=-U_{\beta}\int%
_{-1}^{1}{\frac{\beta-U^{\prime\prime}}{U^{2}}}|\phi_{0}|^{2}dy=U_{\beta
}\lambda_{\beta}^{\prime}(0^-)
\begin{cases}
\leq0 & 0 < \beta\leq\frac{\sqrt{3} - 1}{4} \pi^{2},\\
> 0 & \frac{\sqrt{3} - 1}{4} \pi^{2} < \beta< \frac{5}{16} \pi^{2},\\
= +\infty & \frac{5}{16} \pi^{2} \leq\beta< \frac{1}{2} \pi^{2}.
\end{cases}
\]
Hence, Lemma \ref{le-L-quadratic form} (iii) also ensures that (\ref{37}) is
not a neutral limiting mode when $\beta\in({{(\sqrt{3}-1)\pi^{2}}/{4}}%
,{{\pi^{2}}/{2}})$. In particular, for $\beta\in({{(\sqrt{3}-1)\pi^{2}}/{4}%
},{5\pi^{2}}/{16})$, the SNM curve (\ref{37}) is not part of the stability
boundary even if the singular neutral modes are in $H^{2}$.
\end{remark}

\subsection{Spectrum Continuity}

\label{subsection-spectrum-continuity} In this subsection, we show the
continuity of the $n$-th eigenvalue of \eqref{eq44} up to the boundaries
$c=0,1,\pm\infty$.

First, we consider $c = \pm\infty$.

\begin{lemma}
\label{lem-cinfty} Let $\beta\in\mathbf{R}$. $\lim_{c \to\pm\infty}
\lambda_{n}(\beta,c) = n ^{2} \pi^{2} / 4$.
\end{lemma}

\begin{proof}
By Theorem 2.1 in \cite{Kong1999} and Proposition \ref{pro-41} (2), the
conclusion follows from
\begin{align*}
\label{potential-L1}\left\|  {(\beta-U^{\prime\prime})/(U-c)}\right\|
_{L^{1}(-1,1)}\to0,\;\text{as}\;\;c\rightarrow\pm\infty.
\end{align*}

\end{proof}

\begin{remark}
\label{general-infty-limits} Clearly, $\lim_{c\to\pm\infty}\lambda_{n}%
(\beta,c) = n ^{2} \pi^{2} /(y_{2}-y_{1})^{2}$ for a general flow $U\in
C^{2}([y_{1},y_{2}])$.
\end{remark}

\begin{remark}
One should identify $\Gamma_{\infty}$ with $\Gamma_{-\infty}$ to have a better
understanding about the change of eigenvalues since they correspond to the
same regular Sturm-Liouville problem. For instance, one can take inversion
$\tilde{c}=\frac{1}{c-1/2}$ so that the domain $c\notin(0,1)$ becomes
$\tilde{c}\in\lbrack-2,2]$. The contour plot will look like the following.
\begin{figure}[th]
\centering
\includegraphics[width=3in]{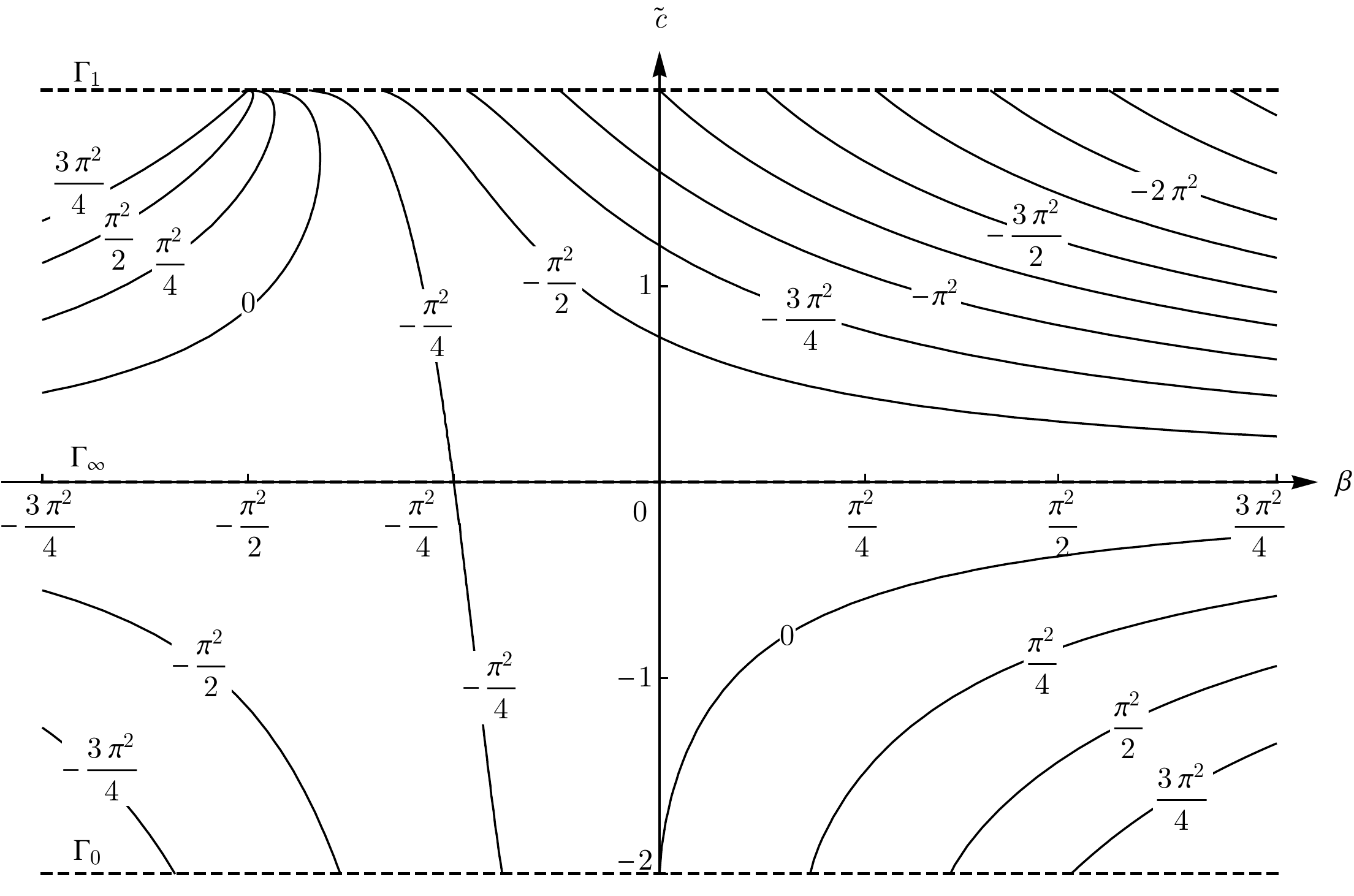} \caption{Contour plot for
$-\lambda_{1}(\beta,\tilde{c})$}%
\label{fig:ctilde}%
\end{figure}


\end{remark}

Next, we consider the finite endpoints $c=0,1$. Our method is based on regular
approximations of singular Sturm-Liouville problems.

Let $T$ be a self-adjoint operator in a Hilbert space $H$. Recall that for any
closable operator $S$ such that $\bar{S}=T$, its domain $D(S)$ is called a
core of $T$. The sequence of self-adjoint operators $\{T_{j}\}_{j=1}^{\infty}$
is said to be spectral included for $T$, if for any $\lambda\in\sigma(T)$,
there exists a sequence $\{\lambda(T_{j})\}_{j=1}^{\infty}$ with
$\lambda(T_{j})\in\sigma(T_{j})$ $(j\geq1)$ such that $\lim
\limits_{j\rightarrow\infty}\lambda(T_{j})=\lambda$.

\begin{proposition}
\label{prop-singular limit}(i) Let $\beta\in(0,{\frac{\pi^{2}}{2}})$. Then for
any $n\geq1$,
\begin{equation}
\lambda_{n}(\beta,c)\rightarrow\lambda_{n}(\beta,0),\;\;as\;\;c\rightarrow
0^{-}. \label{323}%
\end{equation}

(ii) Let $\beta\in(-{\frac{\pi^{2}}{2}},0]$. Then for any $n\geq1$,
\begin{align*}
\lambda_{n}(\beta,c)\rightarrow\lambda_{n}(\beta,1),\;\;as\;\;c\rightarrow
1^{+}.
\end{align*}

\end{proposition}

\begin{proof}
First, we prove (i). We begin to show that the limit (\ref{323}) holds for
$\beta\in(0,{{5\pi^{2}}/{16}}]$. Define
\[
\mathcal{L}_{\beta,\min}^{\prime}\phi=\mathcal{L}_{\beta,0}\phi,D(\mathcal{L}%
_{\beta,\min}^{\prime})=\{\phi\in D(\mathcal{L}_{\beta,0}):\phi
\;\text{has\ compact\ support\ in}\;(-1,1)\}.
\]
We denote the closure of $\mathcal{L}_{\beta,\min}^{\prime}$ by $\mathcal{L}%
_{\beta,\min}$. Since (\ref{eqn-0}) is in the limit point cases at $\pm1$, we
infer from Theorem 10.4.1 and Remark 10.4.2 that ${\mathcal{L}_{\beta,\min}%
}=\mathcal{L}_{\beta,0}$ and it is a self-adjoint operator on $L^{2}(-1,1)$.
Then $D(\mathcal{L}_{\beta,\min}^{\prime})$ is a core of $\mathcal{L}%
_{\beta,0}$. It is obvious that $D(\mathcal{L}_{\beta,\min}^{\prime})\subset
D(\mathcal{L}_{\beta,c})$ for any $c<0$, where $\mathcal{L}_{\beta,c}$ is
defined in (\ref{defn-L_c}). Furthermore, for any $\phi\in D(\mathcal{L}%
_{\beta,\min}^{\prime})$, by setting $\mathrm{supp}(\phi)=[a,b]\subset(-1,1)$
we get
\begin{align*}
\Vert\mathcal{L}_{\beta,c}\phi-\mathcal{L}_{\beta,0}\phi\Vert_{L^{2}(-1,1)} &
=\left\Vert {\frac{\beta-U^{\prime\prime}}{U-c}}\phi-{\frac{\beta
-U^{\prime\prime}}{U}}\phi\right\Vert _{L^{2}(a,b)}=\left\Vert {\frac
{(\beta-U^{\prime\prime})c}{U(U-c)}}\phi\right\Vert _{L^{2}(a,b)}\\
&  =\int_{a}^{b}{\frac{\left(  \beta-U^{\prime\prime}\right)  ^{2}c^{2}}%
{U^{2}(U-c)^{2}}}\phi^{2}\ dy\rightarrow0,\;\;as\;\;c\rightarrow0^{-}.
\end{align*}
Thus, by Theorem VIII 25 (a) in \cite{Reed1972} or Theorem 9.16 (i) in
\cite{Weidmann1980} we have $\{\mathcal{L}_{\beta,c},c<0\}$ is strongly
resolvent convergent to $\mathcal{L}_{\beta,0}$ in $L^{2}(-1,1)$.
Then it follows from Theorem VIII 24 (a) in \cite{Reed1972} that
$\{\mathcal{L}_{\beta,c},\ c<0\}$ is spectral included for $\mathcal{L}%
_{\beta,0}$. We then show that (\ref{323}) holds by induction. Note that
$\lambda_{n}(\beta,0)\in((n^{2}/4-1)\pi^{2},((n+1)^{2}/4-1)\pi^{2})$. Since
$\lambda_{1}(\beta,0)<0$ and $\lambda_{2}(\beta,c)>0$ for all $c<0$ by Lemma
\ref{2nd eigenvalue positive}, we have $\lim_{c\rightarrow0^{-}}\lambda
_{1}(\beta,c)=\lambda_{1}(\beta,0)$. Suppose $\lim_{c\rightarrow0^{-}}%
\lambda_{n}(\beta,c)=\lambda_{n}(\beta,0)$. Since $\lambda_{n}(\beta
,0)\in((n^{2}/4-1)\pi^{2},((n+1)^{2}/4-1)\pi^{2})$ and $\lambda_{n+2}%
(\beta,c)>((n+2)^{2}/4-1)\pi^{2}$ for all $c<0$ by Lemma
\ref{2nd eigenvalue positive}, we have $\lim_{c\rightarrow0^{-}}\lambda
_{n+1}(\beta,c)=\lambda_{n+1}(\beta,0)$.

Next, we show that (\ref{323}) holds for $\beta\in({5\pi^{2}}/{16},{\pi^{2}%
}/{2})$. The above conclusion, Corollary \ref{co21} (i) and Lemma
\ref{2nd eigenvalue positive} ensure that for any given $n\geq1$, there exists
$\delta>0$ such that $\lambda_{n}(\beta,c)\in((n^{2}/4-1)\pi^{2}%
,((n+1)^{2}/4-1)\pi^{2})$ for any $c\in(-\delta,0)$.

Let $c\in(-\delta,0)$, $\phi_{n,c}:=\phi_{n}^{(\beta,c)}$ and recall that
$\|\phi_{n,c}\|_{L^{2}}=1$. We get by integration by parts that
\begin{align}
\label{325}\int_{-1}^{1} \left\vert \phi_{n,c}^{\prime}\right\vert ^{2} dy  &
=\int_{-1}^{1}\left[  {\frac{\beta-U^{\prime\prime}}{U-c}}+\lambda_{n}%
(\beta,c)\right]  |\phi_{n,c}|^{2}dy\\
&  =\pi^{2}\int_{-1}^{1}\left[  {\frac{U-c+c-U_{\beta}}{U-c}}+\frac
{\lambda_{n}(\beta,c)}{\pi^{2}}\right]  |\phi_{n,c}|^{2\ }dy\leq
{\frac{(n+1)^{2}\pi^{2}}{4}},\nonumber
\end{align}
since $U-c>0$ on $y\in(-1,1)$ and $c-U_{\beta}=c-({{1}/{2}}-{{\beta}/{\pi^{2}%
}})<0$. Hence we get $\Vert\phi_{n,c}\Vert_{H^{1}}^{2}\leq{(n+1)^{2}\pi^{2}%
/4}+1$. Therefore, up to a subsequence, we have $\phi_{n,c}\rightharpoonup
\tilde\phi_{n,0}$ in $H^{1}$ and $\phi_{n,c}\rightarrow\tilde\phi_{n,0}\ $in
$C^{0}(\left[  -1,1\right]  ) $ for some $\tilde\phi_{n,0}\in H^{1}$.
Moreover, $\left\Vert \tilde\phi_{n,0}\right\Vert _{L^{2}}=1$ and $\tilde
\phi_{n,0}\left(  \pm1\right)  =0$. Up to a subsequence, let
\[
\lim_{c\rightarrow0^{-}} \lambda_{n}(\beta,c)=\tilde\lambda_{n}(\beta
,0)\in[(n^{2}/4-1)\pi^{2},((n+1)^{2}/4-1)\pi^{2}].
\]
We claim that $\tilde\phi_{n,0}$ solves
\begin{equation}
-\phi^{\prime\prime}-{\frac{\beta-U^{\prime\prime}}{U}}\phi=\tilde\lambda
_{n}(\beta,0)\phi\;\;\text{on}\;(-1,1),\;\; \label{eqn-phi-0}%
\end{equation}
with $\phi(\pm1)=0$. Assuming this is true, then $\tilde\lambda_{n}%
(\beta,0)=\lambda_{n}(\beta,0)$, which is the unique eigenvalue in
$[(n^{2}/4-1)\pi^{2},((n+1)^{2}/4-1)\pi^{2}]$. This proves (\ref{323}).

It remains to show that $\tilde{\phi}_{n,0}$ satisfies (\ref{eqn-phi-0}). Take
any closed interval $[a,b]\in(-1,1)$. There exists $\delta_{0}>0$ such that
$\left\vert U-c\right\vert \geq\delta_{0}\ $ on $\left[  a,b\right]  $ for any
$c\in(-\delta,0)$. Since $\phi_{n,c}$ solves the regular equation
(\ref{eqn-0}) on $[a,b]$, we get a uniform bound for $\Vert\phi_{n,c}%
\Vert_{H^{3}[a,b]}$. Thus, up to a subsequence, $\phi_{n,c}\rightarrow
\tilde{\phi}_{n,0}$ in $C^{2}([a,b])$. Taking the limit $c\rightarrow0^{-}$ in
the equation \eqref{eq44}, we deduce that $\tilde{\phi}_{n,0}$ solves the
equation (\ref{eqn-phi-0}) on $\left[  a,b\right]  \,$\ and also on $(-1,1)$
since $[a,b]\subset(-1,1)$ is arbitrary. This finishes the proof of (i).

Now, we prove (ii). For any $c\geq1$, we obverse that the $n$-th eigenvalue
$\mu_{n}(\beta,c)$ of
\[
-\psi^{\prime\prime}-\frac{\beta-U^{\prime\prime}}{U-c}\psi=\mu_{n}%
(\beta,c)\psi,\;\;\psi(0)=\psi(1)=0
\]
with eigenfunction $\psi_{n,c}$ is exactly the $2n$-th eigenvalue
$\lambda_{2n}(\beta,c)$ of (\ref{eq44}) with eigenfunction $\phi_{n,c}$, which
is defined by $\phi_{n,c}(y)=\psi_{n,c}(y)$ when $y\in\lbrack0,1)$ and
$\phi_{n,c}(y)=-\psi_{n,c}(-y)$ when $y\in(-1,0)$. Noticing that $\mu
_{n}(\beta,1)\in(((2n)^{2}/4-1)\pi^{2},((2n+1)^{2}/4-1)\pi^{2}]$ and $\mu
_{n}(\beta,c)>((2n)^{2}/4-1)\pi^{2}$ for all $c>1$, and similar to the proof
of (\ref{323}), we get $\mu_{n}(\beta,c)\rightarrow\mu_{n}(\beta,1)$ as
$c\rightarrow1^{+}$, which gives $\lambda_{2n}(\beta,c)\rightarrow\lambda
_{2n}(\beta,1)$ as $c\rightarrow1^{+}$. This, together with Lemma
\ref{2nd eigenvalue positive}, yields that for any given $n\geq1$, there exist
$\kappa,\nu>0$ such that $((2n-1)^{2}/4-1)\pi^{2}<\lambda_{2n-1}%
(\beta,c)<((2n+1)^{2}/4-1)\pi^{2}+\kappa$ for all $c\in(1,1+\nu)$. Using this
bound for $\lambda_{2n-1}(\beta,c)$ and similar to the proof of (\ref{325}),
we get a uniform bound for $\left\Vert \phi_{{2n-1},c}\right\Vert _{H_{1}}$,
$c\in(1,1+\nu)$. Thus there exists $\tilde{\phi}_{{2n-1},1}\in H^{1}(-1,1)$
such that, up to a subsequence, $\phi_{{2n-1},c}\rightarrow\tilde{\phi
}_{{2n-1},1}$ in $C^{0}([-1,1])$, $\Vert\tilde{\phi}_{{2n-1},1}\Vert_{L^{2}%
}=1$ and $\tilde{\phi}_{{2n-1},1}\left(  \pm1\right)  =0.$ Up to a
subsequence, let
\[
\tilde{\lambda}_{2n-1}(\beta,1)=\lim_{c\rightarrow1^{+}}\lambda_{2n-1}%
(\beta,c)\in\lbrack((2n-1)^{2}/4-1)\pi^{2},((2n+1)^{2}/4-1)\pi^{2}%
+\kappa]\text{.}%
\]
Then as in the proof that $\tilde{\phi}_{n,0}$ solves (\ref{eqn-phi-0}), we
get $\tilde{\phi}_{{2n-1},1}$ solves
\[
-\phi^{\prime\prime}-{\frac{\beta-U^{\prime\prime}}{U-1}}\phi=\tilde{\lambda
}_{2n-1}(\beta,1)\phi,\;\text{on }\left(  -1,0\right)  \cup\left(  0,1\right)
\]
with $\phi(\pm1)=0$ and $\phi\left(  0\right)  $ to be finite. We observe
$\lambda_{2n-1}(\beta,1)=\lambda_{2n}(\beta,1)$ are the only eigenvalues in
the interval $[((2n-1)^{2}/4-1)\pi^{2},((2n+2)^{2}/4-1)\pi^{2}]$. Therefore,
$\tilde{\lambda}_{2n-1}(\beta,1)=\lambda_{2n-1}(\beta,1)$. This finishes the
proof of the lemma.
\end{proof}

\subsection{Existence of unstable mode with zero wave number}

\label{Unstable mode with zero wave number}

In this subsection, we show the existence of an unstable mode with zero wave
number for any $\beta\in(\beta_{-},0)$.

\begin{proposition}
\label{prop zero wave number} For any $\beta\in(\beta_{-},0)$, there exists an
unstable mode with $\alpha=0$.
\end{proposition}

\begin{proof}
By Theorem \ref{transition}, there exists a sequence of unstable modes
$\{(c_{k},\alpha_{k},\beta,\phi_{k})\}$ with $\left\Vert \phi_{k}\right\Vert
_{L^{2}}=1,\ c_{k}^{i}=\operatorname{Im}c_{k}>0$ and $\alpha_{k}%
\rightarrow0^{+}$. We claim that $\{c_{k}^{i}\}$ has a lower bound $\delta>0$.
Suppose otherwise, there exists a subsequence $\{(c_{k_{j}},\alpha_{k_{j}%
},\beta,\phi_{k_{j}})\}$ such that $\alpha_{k_{j}}\rightarrow0^{+},\ c_{k_{j}%
}^{r}=\text{Re}\ c_{k_{j}}\rightarrow c_{s},\ c_{k_{j}}^{i}\rightarrow0^{+}$
for some $c_{s}\in\mathbf{R}\cup{\pm\infty}$. By Proposition \ref{pro-41} (2),
$\{ c_{k_{j}}\} $ is bounded and thus $c_{s}\in\mathbf{R}$. By Lemma
\ref{le22}, there is a uniform $H^{2}$ bound for the unstable solutions $\{
\phi_{k_{j}}\}$. Thus, there exists $\phi_{0}\in H^{2}(-1,1)$ such that
$\phi_{k_{j}}\rightarrow\phi_{0}$ in $C^{1}([-1,1])$ and $\left\Vert \phi
_{0}\right\Vert _{L^{2}}=1$. Since $\beta\in(\beta_{-},0)$, the only choice
for $c_{s}$ is $c_{s}=U_{\beta}$. Noting that
\begin{equation}
-\phi_{k_{j}}^{\prime\prime}+\alpha_{k_{j}}^{2}\phi_{k_{j}}-{\frac
{\beta-U^{\prime\prime}}{U-c_{k_{j}}}}\phi_{k_{j}}=0,\;\;\phi_{k_{j}}(\pm1)=0,
\label{334}%
\end{equation}
by passing to the limit $j\rightarrow\infty$ in (\ref{334}), we have
\begin{equation}
-\phi_{0}^{\prime\prime}-{\frac{\beta-U^{\prime\prime}}{U-U_{\beta}}}\phi
_{0}=-\phi_{0}^{\prime\prime}-\pi^{2}\phi_{0}=0,\;\;\phi_{0}(\pm1)=0.
\label{335}%
\end{equation}
It is clear that $\phi_{0}(y)=\sin(\pi y)$, $y\in\lbrack-1,1]$. It follows
from (\ref{334})--(\ref{335}) that
\begin{align}
\label{336}{\frac{\alpha_{k_{j}}^{2}}{c_{k_{j}}-U_{\beta}}}\int_{-1}^{1}%
\phi_{k_{j}}\phi_{0}dy  &  =\int_{-1}^{1}{\frac{\beta-U^{\prime\prime}%
}{(U-c_{k_{j}})(U-U_{\beta})}}\phi_{k_{j}}\phi_{0}dy =\pi^{2}\int_{-1}%
^{1}{\frac{1}{U-c_{k_{j}}}}\phi_{k_{j}}\phi_{0}dy.
\end{align}
Note that
\begin{equation}
\lim\limits_{j\rightarrow\infty}\int_{-1}^{1}\phi_{k_{j}}\phi_{0}dy=\int%
_{-1}^{1}\phi_{0}^{2}dy, \label{337}%
\end{equation}
and
\begin{align}
\label{338} &  \lim\limits_{j\rightarrow\infty}\int_{-1}^{1}{\frac
{1}{U-c_{k_{j}}}}\phi_{k_{j}}\phi_{0}dy = \lim\limits_{j\rightarrow\infty
}\left(  \int_{-1}^{1}{\frac{(U-c_{k_{j}}^{r})\phi_{k_{j}}\phi_{0}%
}{(U-c_{k_{j}}^{r})^{2}+{c_{k_{j}}^{i^{2}}}}}dy+i\int_{-1}^{1}{\frac{c_{k_{j}%
}^{i}\phi_{k_{j}}\phi_{0}}{(U-c_{k_{j}}^{r})^{2}+c_{k_{j}}^{i^{2}}}}dy\right)
\\
=  &  \mathcal{P}\int_{-1}^{1}{\frac{\phi_{0}^{2}}{U-U_{\beta}}}dy+i\pi
\sum_{l=1}^{2}\left(  |U^{\prime}|^{-1}\phi_{0}^{2}\right)  \big|_{y=a_{l}%
},\nonumber
\end{align}
by using (86) and (88) in \cite{Lin2003}, where $\mathcal{P}\int_{-1}^{1}$
denotes the Cauchy principal part and $a_{1},a_{2}$ are the points such that
$U(a_{1})=U(a_{2})=U_{\beta}$. Taking the imaginary part of (\ref{336}), we
have
\begin{equation}
\mathrm{\operatorname{Im}}{\frac{\alpha_{k_{j}}^{2}\int_{-1}^{1}\phi_{k_{j}%
}\phi_{0}\ dy}{\pi^{2}\int_{-1}^{1}{\frac{1}{U-c_{k_{j}}}}\phi_{k_{j}}\phi
_{0}\ dy}}=c_{k_{j}}^{i}. \label{339}%
\end{equation}
Then for sufficiently large $k$, by (\ref{337})--(\ref{338}), the LHS of
(\ref{339}) is negative, while the RHS of (\ref{339}) is positive. This
contradiction shows that $\{c_{k}^{i}\}$ has a lower bound $\delta>0$.

Now we show the existence of an unstable mode with $\alpha=0$, by taking the
limit of the sequence of unstable modes $\{(c_{k},\alpha_{k},\beta,\phi
_{k})\}$. Since $\{c_{k}\}$ is bounded, there exists $c_{0}\in\mathbf{C}$ with
$\operatorname{Im}c_{0}\geq\delta$ such that, up to a subsequence,
$c_{k}\rightarrow c_{0}$. Since $\phi_{k}$ solves (\ref{334}) with $k_{j}$
replaced by $k$ and $\{\left\vert U(y)-c_{k}\right\vert :y\in[-1,1]\} $ has a
uniform lower bound $\delta>0$ for all $k\geq1$, we therefore have a uniform
bound of $\left\Vert \phi_{k}\right\Vert _{H^{3}\left(  -1,1\right)  }$. Up to
a subsequence, let $\phi_{k}\rightarrow\tilde\phi_{0}$ in $C^{2}\left(  [
-1,1\right]  ) $. Then $\tilde\phi_{0}$ solves the equation
\[
-\tilde\phi_{0}^{\prime\prime}+{\frac{\beta-U^{\prime\prime}}{U-c_{0}}}%
\tilde\phi_{0}=0,\;\;\text{on }\left(  -1,1\right)
\]
with $\tilde\phi_{0}(\pm1)=0$. Thus $(c_{0},0,\beta,\tilde\phi_{0})$ is an
unstable mode. The proof of this proposition is finished.
\end{proof}

\section{Bifurcation of nontrivial steady solutions}

\label{section bifurcation}

In this section, we prove the bifurcation of non-parallel steady flows near
the shear flow $\left(  U\left(  y\right)  ,0\right)  $ if there exists a
non-resonant neutral mode.

\begin{proposition}
\label{prop-bifurcation} Consider a shear flow $U\in C^{3}\left(  [
-1,1]\right)  $ and fix $\beta\in\mathbf{R}$. Suppose there is a non-resonant
neutral mode $\left(  c_{0},\alpha_{0},\beta,\phi_{0}\right)  $ satisfying
(\ref{15})--(\ref{16}) with $c_{0}>U_{\max}$ or $c_{0}<U_{\min}$, and
$\alpha_{0}>0$. Then there exists $\varepsilon_{0}>0$ such that for each
$0<\varepsilon<\varepsilon_{0}$, there exists a traveling wave solution
$\vec{u}_{\varepsilon}\left(  x-c_{0}t,y\right)  =\left(  u_{\varepsilon
}\left(  x-c_{0}t,y\right)  ,v_{\varepsilon}\left(  x-c_{0}t,y\right)
\right)  $ to the equation (\ref{11}) with boundary condition (\ref{bc-Euler})
which has minimal period $T_{\varepsilon}$ in $x$,
\[
\left\Vert \omega_{\varepsilon}\left(  x,y\right)  -\omega_{0}\left(
y\right)  \right\Vert _{H^{2}\left(  0,T_{\varepsilon}\right)  \times\left(
-1,1\right)  }=\varepsilon,\ \ \omega_{\varepsilon}=\operatorname{curl}\vec
{u}_{\varepsilon},\ \omega_{0}=-U^{\prime}\left(  y\right)  ,
\]
and $T_{\varepsilon}\rightarrow{2\pi}/{\alpha_{0}}$ when$\ \varepsilon
\rightarrow0$. Moreover, $u_{\varepsilon}\left(  x,y\right)  \neq0$ and
$v_{\varepsilon}$ is not identically zero.
\end{proposition}

\begin{proof}
We assume $c_{0}>U_{\max}$ and the case $c_{0}<U_{\min}$ is similar. The proof
is similar to that of Lemma 1 in \cite{lin-zeng-couette}, we give it here for
completeness. From the vorticity equation (\ref{vorticity-eqn}), it can be
seen that $\vec{u}\left(  x-c_{0}t,y\right)  $ is a solution of (\ref{11}) if
and only if
\[
\frac{\partial\left(  \omega+\beta y,\psi-c_{0}y\right)  }{\partial\left(
x,y\right)  }=0
\]
and $\psi$ takes constant values on $\left\{  y=\pm1\right\}  $, where
$\omega$ and $\psi$ are the vorticity and stream function corresponding to
$\vec{u}$, respectively. Let $\psi_{0}$ be a stream function associated with
the shear flow $\left(  U-c_{0},0\right)  $, i.e., $\psi_{0}^{\prime}\left(
y\right)  =U\left(  y\right)  -c_{0}$. Since $U-c_{0}<0$, $\psi_{0}$ is
decreasing on $(-1,1)$. Therefore we can define a function $f_{0}\in
C^{2}({\rm Ran}\,(\psi_{0}))$ such that
\begin{equation}
f_{0}\left(  \psi_{0}\left(  y\right)  \right)  =\omega_{0}\left(  y\right)
+\beta y=-\psi_{0}^{\prime\prime}\left(  y\right)  +\beta y.
\label{eqn-f-psi-0}%
\end{equation}
Thus
\[
f_{0}^{\prime}\left(  \psi_{0}\left(  y\right)  \right)  =\frac{\beta
-U^{\prime\prime}\left(  y\right)  }{U\left(  y\right)  -c_{0}}=:\mathcal{K}%
_{c_{0}}\left(  y\right)  .
\]
Then we extend $f_{0}$ to $f\in C_{0}^{2}\left(  \mathbf{R}\right)  $ such
that $f=f_{0}$ on ${\rm Ran}\,(\psi_{0})$. We construct steady solutions near
$\left(  U-c_{0},0\right)  $ by solving the elliptic equation
\[
-\Delta\psi+\beta y=f\left(  \psi\right)  ,
\]
where $\psi\left(  x,y\right)  $ is the stream function and $\left(
u,v\right)  =\left(  \psi_{y},-\psi_{x}\right)  $ is the steady velocity. Let
$\xi=\alpha x,\ \psi\left(  x,y\right)  =\tilde{\psi}\left(  \xi,y\right)  ,$
where $\tilde{\psi}\left(  \xi,y\right)  $ is $2\pi$-periodic in $\xi.$ We use
$\alpha^{2}$ as the bifurcation parameter. The equation for $\tilde{\psi
}\left(  \xi,y\right)  $ becomes
\begin{equation}
-\alpha^{2}\frac{\partial^{2}\tilde{\psi}}{\partial\xi^{2}}-\frac{\partial
^{2}\tilde{\psi}}{\partial y^{2}}+\beta y-f(\tilde{\psi})=0,
\label{eqn-psi-tilde}%
\end{equation}
with the boundary conditions that $\tilde{\psi}$ takes constant values on
$\left\{  y=\pm1\right\}  $. Define the perturbation of the stream function
by
\[
\phi\left(  \xi,y\right)  =\tilde{\psi}\left(  \xi,y\right)  -\psi_{0}\left(
y\right)  .
\]
Then using (\ref{eqn-f-psi-0}), we reduce the equation (\ref{eqn-psi-tilde})
to
\begin{equation}
-\alpha^{2}\frac{\partial^{2}\phi}{\partial\xi^{2}}-\frac{\partial^{2}\phi
}{\partial y^{2}}-\left(  f(\phi+\psi_{0})-f\left(  \psi_{0}\right)  \right)
=0. \label{eqn-phi-traveling}%
\end{equation}
Define the spaces%
\begin{align*}
B= \{\phi(\xi,y)\in H^{3}([0,2\pi]\times\lbrack-1,1]):\text{ }\phi(\xi
,\pm1)=0,\, 2\pi\text{-periodic and even in }\xi\}
\end{align*}
and
\[
C=\left\{  \phi(\xi,y)\in H^{1}([0,2\pi]\times\lbrack-1,1]):\text{ }%
2\pi\text{-periodic and even in }\xi\right\}  .
\]
Consider the mapping
\[
F\ :B\times\mathbf{R}^{+}\rightarrow C
\]
defined by
\[
F(\phi,\alpha^{2})=-\alpha^{2}\frac{\partial^{2}\phi}{\partial\xi^{2}}%
-\frac{\partial^{2}\phi}{\partial y^{2}}-\left(  f(\phi+\psi_{0})-f\left(
\psi_{0}\right)  \right)  .
\]
We study the bifurcation near the trivial solution $\phi=0$ of the equation
$F(\phi,\alpha^{2})=0$ in $B$, whose solutions give steady flows with
$x$-period $\frac{2\pi}{\alpha}$.\ The linearized operator of $F$
around$\ \left(  0,\alpha_{0}^{2}\right)  $ has the form
\[
\mathcal{G}:=F_{\phi}(0,\alpha_{0}^{2})=-\alpha_{0}^{2}\frac{\partial^{2}%
}{\partial\xi^{2}}-\frac{\partial^{2}}{\partial y^{2}}-f^{\prime}(\psi
_{0})=-\alpha_{0}^{2}\frac{\partial^{2}}{\partial\xi^{2}}-\frac{\partial^{2}%
}{\partial y^{2}}-\mathcal{K}_{c_{0}}.
\]
By our assumption, the operator $-\frac{\partial^{2}}{\partial y^{2}%
}-\mathcal{K}_{c_{0}}:H_{0}^{1}\cap H^{2}\left(  -1,1\right)  \rightarrow
L^{2}(-1,1)$ has a negative eigenvalue $-\alpha_{0}^{2}$ with the
eigenfunction $\phi_{0}$. Therefore, the kernel of $\mathcal{G}:$
$\ B\rightarrow C\ $is given by
\[
\ker(\mathcal{G})=\left\{  \phi_{0}(y)\cos\xi\right\}  .
\]
In particular, the dimension of $\ker$({$\mathcal{G}$}) is $1$. Since
{$\mathcal{G}$} is self-adjoint, $\phi_{0}(y)\cos\xi\not \in {\rm Ran}\,($%
{$\mathcal{G}$}$)$. Note that $\partial_{\alpha^{2}}\partial_{\phi}%
F(\phi,\alpha^{2})$ is continuous and
\[
\partial_{\alpha^{2}}\partial_{\phi}F(0,\alpha_{0}^{2})\left(  \phi_{0}%
(y)\cos\xi\right)  =-\frac{\partial^{2}}{\partial\xi^{2}}\left[  \phi
_{0}(y)\cos\xi\right]  =\phi_{0}(y)\cos\xi\not \in {\rm Ran}\,({\mathcal{G}}).
\]
By the Crandall-Rabinowitz local bifurcation theorem \cite{CR71}, there exists
a local bifurcating curve $\left(  \phi_{\gamma},\alpha^{2}(\gamma)\right)  $
of $F(\phi,\alpha^{2})=0$, which intersects the trivial curve $\left(
0,\alpha^{2}\right)  $ at $\alpha^{2}=\alpha_{0}^{2}$, such that
\[
\phi_{\gamma}(\xi,y)=\gamma\phi_{0}(y)\cos\xi+o(\gamma),
\]
$\alpha^{2}(\gamma)$ is a continuous function of $\gamma$, and $\alpha
^{2}(0)=\alpha_{0}^{2}$. So the stream functions of the perturbed steady flows
in $(\xi,y)\ $coordinates take the form
\begin{equation}
\psi_{\gamma}(\xi,y)=\psi_{0}\left(  y\right)  +\gamma\phi_{0}(y)\cos
\xi+o(\gamma). \label{cats-eye}%
\end{equation}
Let the velocity $\vec{u}_{\gamma}=\left(  u_{\gamma},v_{\gamma}\right)
=\left(  \partial_{y}\psi_{\gamma},-\partial_{x}\psi_{\gamma}\right)  $. Then
\[
u_{\gamma}=U\left(  y\right)  -c_{0}+\gamma\phi_{0}^{\prime}(y)\cos
\xi+o(1)\neq0
\]
when $\gamma$ is small.
\end{proof}

By adjusting the traveling speed, we can construct traveling waves near the
Sinus flow with the period ${2\pi}/{\alpha_{0}}$.

\begin{theorem}
\label{thm-bifurcation-non-resonant}Consider the Sinus flow. Then there exists
at least one non-resonant neutral mode $\left(  c_{0},\alpha_{0},\beta
,\phi_{0}\right)  \ $in the following stable cases:

(1) $\left\vert \beta\right\vert >{\pi^{2}}/{2}$ and $0<\alpha_{0}\leq
{\sqrt{3}\pi}/{2};$

(2) $\beta\in(-{\pi^{2}}/{2},\beta_{-})\cup(0,{\pi^{2}}/{2})$ and
$0<\alpha_{0}<\sqrt{\Lambda_{\beta}}.$ \newline In these two cases, there
exists $\varepsilon_{0}>0$ such that for each $0<\varepsilon<\varepsilon_{0}$,
there exists a traveling wave solution $\vec{u}_{\varepsilon}\left(
x-c_{\varepsilon}t,y\right)  =\left(  u_{\varepsilon}\left(  x-c_{\varepsilon
}t,y\right)  ,v_{\varepsilon}\left(  x-c_{\varepsilon}t,y\right)  \right)  $
to the equation (\ref{11}) with boundary condition (\ref{bc-Euler}) which has
minimal period $T_{0}=\frac{2\pi}{\alpha_{0}}$ in $x$,
\[
\left\Vert \omega_{\varepsilon}\left(  x,y\right)  -\omega_{0}\left(
y\right)  \right\Vert _{H^{2}\left(  0,T_{0}\right)  \times\left(
-1,1\right)  }=\varepsilon, \ \ \omega_{\varepsilon}=\operatorname{curl}%
\vec{u}_{\varepsilon},\ \omega_{0}=-U^{\prime}\left(  y\right)  ,
\]
with $c_{\varepsilon}\rightarrow c_{0}$ when $\varepsilon\rightarrow0$.
Moreover, $u_{\varepsilon}\left(  x,y\right)  \neq0$ and $v_{\varepsilon}$ is
not identically zero.
\end{theorem}

\begin{proof}
First, we show the existence of non-resonant neutral modes in the two cases.
For case (1), recall that $\lambda_{\beta}(c)= \lambda_{1}(\beta,c)$ is the
principal eigenvalue of $(\ref{eq44})$. Note that $\lambda_{\beta}(U_{\beta
})=-{\frac{3\pi^{2}}{4}}$ and $\lambda_{\beta}(\pm\infty)={\frac{\pi^{2}}{4}}%
$. By continuity of $\lambda_{\beta}$, we have if $\beta>{\frac{\pi^{2}}{2}}$,
then $[0,{\frac{3\pi^{2}}{4}}]\subset\{-\lambda_{\beta}(c):c\in(-\infty
,U_{\beta}]\}$, and if $\beta<-{\frac{\pi^{2}}{2}}$, then $[0,{\frac{3\pi^{2}%
}{4}}]\subset\{-\lambda_{\beta}(c):c\in[U_{\beta},+\infty)\}$. Therefore,
there exists at least one non-resonant neutral mode for case (1). For case
(2), the existence of non-resonant neutral modes follows from Theorem
\ref{transition}.

Let $\left(  c_{0},\alpha_{0},\beta,\phi_{0}\right)  $ be a non-resonant
neutral mode. We consider the case $c_{0}>1$ and the case $c_{0}<0$ is
similar. Let $I\subset\left(  1,\infty\right)  $ be a small interval centered
at $c_{0}$. For each $c\in I$, $\lambda_{\beta}(c)$ is the negative eigenvalue
near $\lambda_{\beta}\left(  c_{0}\right)  =-\alpha_{0}^{2}\ $of the operator
$\mathcal{L}_{\beta,c} $. Let $\alpha\left(  c\right)  =\sqrt{-\lambda_{\beta
}(c)}$. By Corollary \ref{co21} and Theorem \ref{transition}, if we choose
$\left\vert I\right\vert $ to be small enough, then $\alpha\left(  c\right)  $
is strictly monotone on $I$. Assume that $\alpha\left(  c\right)  $ is
increasing on $I$. Let $c_{1}$ and $c_{2}$ in $I$ such that $c_{1}<c_{0}%
<c_{2}$. Then
\begin{equation}
\alpha\left(  c_{1}\right)  <\alpha_{0}<\alpha\left(  c_{2}\right)  .
\label{monotone-eigenvalue-c}%
\end{equation}
By Proposition \ref{prop-bifurcation}, for any $c\in\left(  c_{1}%
,c_{2}\right)  $, there exists local bifurcation of non-parallel traveling
wave solutions of the equation (\ref{11}) with boundary condition
(\ref{bc-Euler}), near the shear flow $\left(  U,0\right)  $. More precisely,
we can find $r_{0}>0$ (independent of $c\in\left(  c_{1},c_{2}\right)  $) such
that for any $0<r<r_{0}$, there exists a nontrivial traveling wave solution
\[
\vec{u}_{c,r}\left(  x-ct,y\right)  =\left(  u_{c,r}\left(  x-ct,y\right)
,v_{c,r}\left(  x-ct,y\right)  \right)
\]
with vorticity $\omega_{c,r}$ which has minimum $x$-period $T_{c,r}$ and
\[
\left\Vert \omega_{c,r}-\omega_{0}\right\Vert _{H^{2}\left(  0,T_{c,r}\right)
\times\left(  -1,1\right)  }=r.
\]
Moreover,
\[
{2\pi}/{T_{c,r}}\rightarrow\alpha\left(  c\right)  \text{ when }%
r\rightarrow0.
\]
By (\ref{monotone-eigenvalue-c}), when $r_{0}$ is chosen to be small enough,
\[
T_{c_{2},r}<{2\pi}/{\alpha_{0}}<T_{c_{1},r}\;\text{ for\ any\ }r\in(0,r_{0}).
\]
Since $T_{c,r}$ is continuous to $c,$ for each $r\in(0,r_{0})$, there exists
$c^{\ast}\left(  r\right)  \in\left(  c_{1},c_{2}\right)  \,$ such that
$T_{c^{\ast}\left(  r\right)  ,r}={2\pi}/{\alpha_{0}}$. Then the traveling
wave solution
\[
\vec{u}_{r}\left(  x-c^{\ast}\left(  r\right)  t,y\right)  :=\left(
u_{c^{\ast}\left(  r\right)  ,r}\left(  x-c^{\ast}\left(  r\right)
t,y\right)  ,v_{c^{\ast}\left(  r\right)  ,r}\left(  x-c^{\ast}\left(
r\right)  t,y\right)  \right)
\]
with the vorticity $\omega_{r}:=\omega_{c^{\ast}\left(  r\right)  ,r}$ is a
nontrivial steady solution of (\ref{11}) satisfying boundary condition
(\ref{bc-Euler}), with minimal $x$-period ${2\pi}/{\alpha_{0}}$ and
\[
\left\Vert \omega_{r}-\omega_{0}\right\Vert _{H^{2}\left(  0,\frac{2\pi
}{\alpha_{0}}\right)  \times\left(  -1,1\right)  }=r.
\]
This finishes the proof of the theorem.
\end{proof}

\begin{remark}
The non-resonant neutral mode does not exist when there is no Coriolis effects
(i.e. $\beta=0$). The traveling waves constructed above are thus purely due to
the Coriolis forces, with traveling speeds beyond the range of the basic flow.
Their existence suggests that the long time dynamics near the shear flows is
much richer. This indicates that the addition of Coriolis effects can
significantly change the dynamics of fluids.
\end{remark}

\section{Linear inviscid damping}

\label{section-damping}

In this section, we prove the linear inviscid damping using the Hamiltonian
structures of the linearized equation (\ref{linearized Euler-Hamiltonian}).
First, we show that for the Sinus flow, when $\alpha^{2}>{3\pi^{2}}/{4}$ and
$\left\vert \beta\right\vert \leq{\pi^{2}}/{2}$, there are no neutral modes in
$H^{2}$.

\begin{lemma}
\label{le-no-neutral} Consider the Sinus flow and fix any $\beta\in
\lbrack-{{\pi^{2}}/{2}},{{\pi^{2}}/{2}}]$.

(i) When $\alpha^{2}>{{3\pi^{2}}/{4}}$, there exist no neutral modes in
$H^{2}$.

(ii) When $\alpha^{2}={{3\pi^{2}}/{4}}$, $(c=U_{\beta},\alpha,\beta,\phi_{0})$
with $\phi_{0}(y)=\cos({{\pi y}/{2}})$ is the only neutral mode in $H^{2}$.
\end{lemma}

\begin{proof}
If $c=U_{\beta}$, then it follows from Proposition \ref{pro-41} (1) that there
are no neutral modes in $H^{2}$ for $\alpha^{2}>{{3\pi^{2}}/{4}}$ and exactly
one neutral solution $\phi_{0}\in H^{2}$ for $\alpha^{2}={{3\pi^{2}}/{4}}$. If
$c\in\lbrack0,1]$ and $c\neq U_{\beta}$, then the only neutral mode lies in
the SNM curve (\ref{37}), with wave number $\alpha^{2}<{{3\pi^{2}}/{4}}$. If
$c\notin\lbrack0,1]$, then by Lemma \ref{2nd eigenvalue positive}, there are
no neutral modes for $\alpha^{2}\geq{{3\pi^{2}}/{4}}$.
\end{proof}

The above lemma implies that there are no purely imaginary eigenvalues of the
linearized Euler operator $JL$ defined in (\ref{linearized Euler-Hamiltonian})
when $\alpha^{2}>{{3\pi^{2}}/{4}}$. This implies the following inviscid
damping of the velocity fields.

\begin{theorem}
\label{thm damping stable}Consider the linearized equation
(\ref{linearized-Euler-shear-frame}) with $U$ to be the Sinus flow.

(i) For any $\alpha^{2}\in({{3\pi^{2}}/{4}},+\infty)$ and $\beta\in
\lbrack-{{\pi^{2}}/{2}},{{\pi^{2}}/{2}}]$, we have
\[
\frac{1}{T}\int_{0}^{T}\left\Vert \vec{u}\left(  t\right)  \right\Vert
_{L^{2}}^{2}dt\rightarrow0,\ \text{when }T\rightarrow\infty,
\]
for any solution $\omega(t) =\operatorname{curl}\vec{u}(t) \ $of
(\ref{linearized-Euler-shear-frame}) with $\omega\left(  0\right)  $ in the
non-shear space $X$ defined in (\ref{non-shear space}).

(ii) For $\alpha^{2}={{3\pi^{2}}/{4}}$ and $\beta\in\lbrack-{{\pi^{2}}/{2}%
},{{\pi^{2}}/{2}}]$, we have
\[
\frac{1}{T}\int_{0}^{T}\left\Vert \vec{u}_{1}\left(  t\right)  \right\Vert
_{L^{2}}^{2}dt\rightarrow0,\ \text{when }T\rightarrow\infty,
\]
where $\vec{u}_{1}\left(  t\right)  $ is the velocity corresponding to the
vorticity $\left(  I-P_{1}\right)  \omega\left(  t\right)  $ with
$\omega\left(  0\right)  \in X$. Here, $P_{1}$ is the projection of $X$ to
\[
\ker L=span\left\{  e^{\pm i\frac{\sqrt{3}\pi}{2}x}\cos\left(  {\pi y}%
/{2}\right)  \right\}  .
\]

\end{theorem}

\begin{proof}
The solution of the linearized equation (\ref{linearized-Euler-shear-frame})
is written as $\omega\left(  t\right)  =e^{tJL}\omega\left(  0\right)  $,
where
\begin{equation}
J=-(\beta-U^{\prime\prime})\partial_{x},\ L={1}/{\pi^{2}}-\left(
-\Delta\right)  ^{-1} \label{defn-J-L-sinus-X}%
\end{equation}
as in (\ref{defn-J-L}). First, we note that when $\alpha^{2}>{{3\pi^{2}}/{4}}%
$, $L\ $is positive on $X$. As a consequence, $\left[  \cdot,\cdot\right]
=\left\langle L\cdot,\cdot\right\rangle $ defines an equivalent inner product
on $X$ with the $L^{2}$ inner product. For any $\omega_{1},\omega_{2}\in X$,
we have
\[
\left\langle LJL\omega_{1},\omega_{2}\right\rangle =\left\langle JL\omega
_{1},L\omega_{2}\right\rangle =-\left\langle L\omega_{1},JL\omega
_{2}\right\rangle ,
\]
and thus $JL$ is anti-self-adjoint on $\left(  X,\left[  \cdot,\cdot\right]
\right)  $. Therefore, the spectrum of $JL$ on $\left(  X,\left[  \cdot
,\cdot\right]  \right)  $ is on the imaginary axis. Since the operator $JL$ is
a compact perturbation of $-(U-U_{\beta})\partial_{x}$, whose spectrum is
clearly the whole imaginary axis, it follows from Weyl's Theorem that the
continuous spectrum of $JL$ is also the whole imaginary axis. Moreover, by
Lemma \ref{le-no-neutral}, $JL$ has no embedded eigenvalues on the imaginary
axis. Applying the RAGE theorem (\cite{CFKS}) to $e^{tJL}$, we have
\[
\frac{1}{T}\int_{0}^{T}\left\Vert B\omega\left(  t\right)  \right\Vert
_{L^{2}}^{2}dt\rightarrow0,\ \text{when }T\rightarrow\infty
\]
for any compact operator $B$ on $L^{2}\left(  S_{2\pi/\alpha}\times\left[
y_{1},y_{2}\right]  \right)  $ and for any solution $\omega\left(  t\right)
\ $of (\ref{linearized-Euler-shear-frame}) with $\omega\left(  0\right)  \in
X$. The conclusion (i) follows by choosing
\[
B\omega=\nabla^{\perp}\left(  -\Delta\right)  ^{-1}\omega=\vec{u},
\]
that is, the mapping operator from vorticity to velocity.

To prove (ii), we define $X_{1}=\left(  I-P_{1}\right)  X$. Then $L|_{X_{1}%
}>0$ and $A_{1}=\left(  I-P_{1}\right)  JL|_{X_{1}}$ is anti-self-adjoint on
$\left(  X_{1},\left[  \cdot,\cdot\right]  \right)  $. The operator $A_{1}$
has no nonzero purely imaginary eigenvalues. Moreover, the proof of Lemma
\ref{lemma generalized kernel} implies that $\ker A_{1}=\left\{  0\right\}  $.
Therefore, $A_{1}$ has purely continuous spectrum in the imaginary axis. The
conclusion again follows from the RAGE theorem to $e^{tA_{1}}$ on $X_{1}$.
\end{proof}

Next, we consider the inviscid damping for the unstable case. By Theorem
\ref{transition}, there exist exactly one unstable mode and no neutral mode in
$H^{2}$ when $\beta\in(-{{\pi^{2}}/{2}},{{\pi^{2}}/{2}})$ and $\alpha^{2}%
\in({\Lambda_{\beta}},{{3\pi^{2}}/{4}})$. As in the stable case, we consider
the linearized equation (\ref{linearized-Euler-shear-frame}) written as
Hamiltonian form $\partial_{t}\omega=JL\omega\ $in the non-shear space $X$,
where $J$ and $L$ are defined in (\ref{defn-J-L-sinus-X}). The space $X$ is
defined in (\ref{non-shear space}) with $\alpha$ to be an unstable wave number
in this case.

Denote $E^{s}\left(  E^{u}\right)  \subset X$ to be the stable (unstable)
eigenspace of $JL$. Then by Corollary 6.1 in \cite{Lin2017}, $L|_{E^{s}\oplus
E^{u}}$ is non-degenerate and
\begin{equation}
n^{-}\left(  L|_{E^{s}\oplus E^{u}}\right)  =\dim E^{s}=\dim E^{u}.
\label{negative-index-unstable}%
\end{equation}
Define the center space $E^{c}$ to be the orthogonal (in the inner product
$\left[  \cdot,\cdot\right]  $) complement of $E^{s}\oplus E^{u}$ in $X$, that
is,
\begin{equation}
E^{c}=\left\{  \omega\in X\ |\ \left\langle L\omega,\omega_{1}\right\rangle
=0,\ \forall\;\omega_{1}\in E^{s}\oplus E^{u}\right\}  . \label{defn-E-c}%
\end{equation}
Then we get the following results.

\begin{lemma}
\label{lemma-positive-center} Consider the Sinus flow and let $\alpha$ be an
unstable wave number. Then the decomposition $X=E^{s}\oplus E^{c}\oplus E^{u}$
is invariant under $JL$. Moreover, we have

(i)
\begin{equation}
\dim E^{s}=\dim E^{u}=n^{-}\left(  L\right)  \text{. } \label{negative-index}%
\end{equation}

(ii) $n^{-}\left(  L|_{E^{c}}\right)  =0$ and as a consequence, $L|_{E^{c}%
/\ker L}>0$.

(iii) The operator $JL|_{E^{c}}$ has no nonzero purely imaginary eigenvalues.
\end{lemma}

\begin{proof}
The invariance of the decomposition follows from the invariance of
$\left\langle L\cdot,\cdot\right\rangle $ under $JL$. To prove
(\ref{negative-index}), we note that $JL$ can be decomposed as the operators
$J_{l\alpha}L_{l\alpha}$ on the spaces $X^{l}$ (defined in (\ref{defn-X-l}))
with the wave number $\alpha l,\ $where $0\neq l\in\mathbf{Z}$. Then
\[
\dim E^{s}=\dim E^{u}=\sum_{l}k_{u}^{l},
\]
where $k_{u}^{l}$ is the number of unstable modes for $J_{l\alpha}L_{l\alpha}%
$. For each $l$, when $\left\vert \alpha l\right\vert $ is an unstable wave
number, there is exactly one unstable mode, and thus we have $k_{u}%
^{l}=1=n^{-}\left(  L_{l\alpha}\right)  $. If $\left\vert \alpha l\right\vert
\geq\frac{3\pi^{2}}{4}$, then we also have $k_{u}^{l}=0=n^{-}\left(
L_{l\alpha}\right)  $. Therefore
\[
\dim E^{s}=\dim E^{u}=\sum_{l}k_{u}^{l}=\sum_{l}n^{-}\left(  L_{l\alpha
}\right)  =n^{-}\left(  L\right)
\]
and (\ref{negative-index}) is proved.

To show (ii), noting that by the definition of $E^{c}$,
(\ref{negative-index-unstable}) and (\ref{negative-index}), we have
\[
n^{-}\left(  L|_{E^{c}}\right)  =n^{-}\left(  L\right)  -n^{-}\left(
L_{E^{s}\oplus E^{u}}\right)  =0,
\]
and thus $L|_{E^{c}/\ker L}>0$.

For each $l$, by Theorem \ref{transition} and Lemma \ref{le-no-neutral}, the
operator $J_{l\alpha}L_{l\alpha}\,\ $has no neutral modes except for
$c=U_{\beta}$ when $\left\vert \alpha l\right\vert ={\sqrt{3}\pi}/{2}$, which
corresponds to nontrivial $\ker L_{l\alpha}$ and $\ker L$. Thus the property
(iii) follows.
\end{proof}

Since $E^{c}$ is invariant under $JL$, we can restrict the linearized equation
(\ref{linearized-Euler-shear-frame}) on $E^{c}$. The linear inviscid damping
still holds true for initial data in $E^{c}$. By the same proof of Theorem
\ref{thm damping stable}, we have the following.

\begin{theorem}
Consider the linearized equation (\ref{linearized-Euler-shear-frame}) with $U$
to be the Sinus flow. \label{thm-damping-center}Let $\alpha$ be an unstable
wave number.

(i) If $\left\vert \alpha l\right\vert \neq{\sqrt{3}\pi}/{2}$ for any
$l\in\mathbf{Z}$, then
\[
\frac{1}{T}\int_{0}^{T}\left\Vert u\left(  t\right)  \right\Vert _{L^{2}}%
^{2}dt\rightarrow0,\ \text{when }T\rightarrow\infty,
\]
for any solution $\omega\left(  t\right)  \ $of
(\ref{linearized-Euler-shear-frame}) with $\omega\left(  0\right)  \in E^{c}$.
Here, $E^{c}$ is the center space defined in (\ref{defn-E-c}).

(ii) If $\left\vert \alpha l\right\vert ={\sqrt{3}\pi}/{2}$ for some
$l\in\mathbf{Z}$, then
\[
\frac{1}{T}\int_{0}^{T}\left\Vert u_{1}\left(  t\right)  \right\Vert _{L^{2}%
}^{2}dt\rightarrow0,\ \text{when }T\rightarrow\infty,
\]
where $u_{1}\left(  t\right)  $ is the velocity corresponding to the vorticity
$\left(  I-P_{1}\right)  \omega\left(  t\right)  $ with $\omega\left(
0\right)  \in E^{c}$.
\end{theorem}

\begin{remark}
For general flows $U$ in class $\mathcal{K}^{+}$, when there are no nonzero
imaginary eigenvalues for the linearized operator $JL$ (defined in
(\ref{linearized Euler-Hamiltonian})), the linear inviscid damping can be
shown as in Theorems \ref{thm damping stable}--\ref{thm-damping-center}, for
$\omega\left(  0\right)  \in L^{2}$.

When $\beta=0$, the nonexistence of nonzero imaginary eigenvalues and, as a
consequence, the linear damping is true for flows in class $\mathcal{K}^{+}$
(see \cite{lin-xu-2017}). Recently, when $\beta=0$, under the assumption that
the linearized operator has no embedding eigenvalues, more explicit linear
decay estimates of the velocity were obtained for symmetric and monotone shear
flows in \cite{Wei-Zhang-Zhao2017, Wei-Zhang-Zhao1704, zillinger} with more
regular initial data (e.g. $\omega\left(  0\right)  \in H^{1}$ or $H^{2}$).

When $\beta\neq0$, under the assumption that the linearized operator has no
embedding eigenvalues, linear inviscid damping was shown for a class of
general flows, and more explicit linear decay estimates of the velocity were
obtained for monotone shear flows in \cite{Wei-Zhang-Zhu2018} with more
regular initial data.
\end{remark}

\section{Conclusions for the Sinus flow}

\label{section discussions}

In this section, we summarize our results for the Sinus flow and compare them
with the previous work in \cite{Kuo1974,Pedlosky1987}. The stability picture
obtained in Theorem \ref{transition} is shown in Figure \ref{fig:criterior}
for the parameters $\left(  \alpha,\beta\right)  $ below.

\begin{figure}[th]
\centering
\includegraphics[width=140mm]{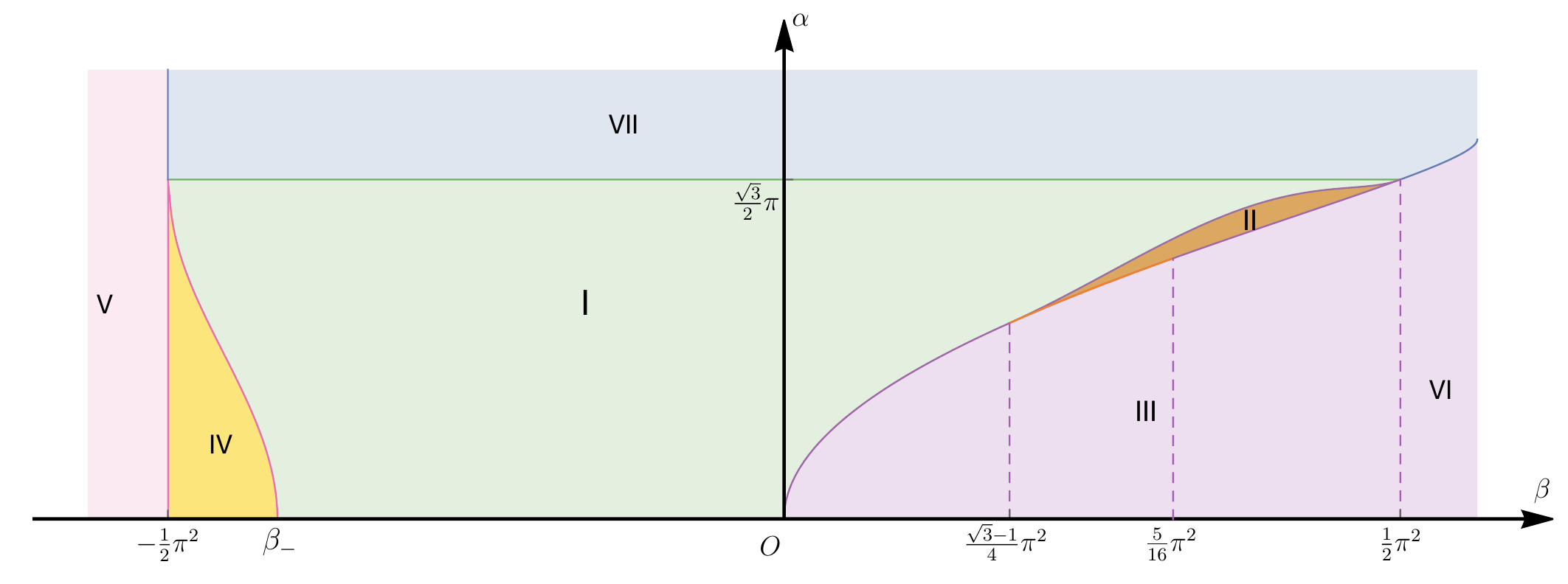} \caption{Stability picture for
Sinus flow}%
\label{fig:criterior}%
\end{figure}



In Figure \ref{fig:criterior}, the right boundary of region IV (yellow area)
is given by the curve
\[
\Gamma_{1}:\alpha=\sqrt{\Lambda_{\beta}},\ \ \ \beta\in\left(  -{{\pi^{2}}%
/{2}},\beta_{-}\right)  ,\
\]
The upper boundary of region III (purple area) is given by the curve
\[
\Gamma_{2}:\alpha\left(  \gamma\right)  =\pi\sqrt{1-\gamma^{2}},\ \ \ \beta
\left(  \gamma\right)  =\pi^{2}\left(  -\gamma^{2}+{{\gamma}/{2}}+{{1}/{2}%
}\right)  \in\left(  0,{{\pi^{2}}/{2}}\right)  ,\
\]
where $\gamma\in\left(  {\frac{1}{2}},1\right)  $ and $\left(  \alpha\left(
\gamma\right)  ,\beta\left(  \gamma\right)  \right)  \ $corresponds to the SNM
curve \ with $c=0$ (see (\ref{37})). Here, the boundary of regions III and VI
is $\alpha\in\left(  0,\frac{\sqrt{3}\pi}{2}\right)  $, $\beta=\frac{\pi^{2}%
}{2}$. The upper boundary of region II (orange area) is given by
\[
\Gamma_{3}:\alpha=\sqrt{\Lambda_{\beta}}>\pi\sqrt{1-\gamma^{2}},\ \ \beta
=\pi^{2}\left(  -\gamma^{2}+{{\gamma}/{2}}+{{1}/{2}}\right)  \in\left(
{(\sqrt{3}-1)\pi^{2}}/{4},{\pi^{2}}/{2}\right),
\]
which has been exaggerated in Figure \ref{fig:criterior} because it is too close to $\Gamma _2$.

Only region I (Green area) is the unstable domain, consisting of unstable
parameters $\left(  \alpha,\beta\right)  $ given in Theorem \ref{transition}.
In region I, there exist exactly one unstable mode and no neutral modes in
$H^{2}$. All other regions in Figure 3 are stable domains, but with different
properties on neutral modes. In region VII (blue area), there are no neutral
modes in $H^{2}$, see Lemma \ref{le-no-neutral}. In region V, there exists at
least one non-resonant neutral mode with $c>1$. In region VI, there exists at
least one non-resonant neutral mode with $c<0$. More discussion on the number
of non-resonant neutral modes in regions V and VI is under investigation. In
region II, there exist exactly two non-resonant neutral modes with $c<0$. In
region IV (yellow area), there exist exactly two non-resonant neutral modes
with $c>1$. In region III, there exists exactly one non-resonant neutral mode
with $c<0$. The dynamical behavior of the fluid equation (\ref{vorticity-eqn})
is quite different in these regions. In region VII, the linear inviscid
damping is shown for non-shear perturbations. In regions III, VI and V, the
non-resonant neutral mode generates nontrivial traveling waves with the wave
number $\alpha$. In regions II and IV, the two non-resonant neutral modes
generate two traveling waves with different speeds. Moreover, for region I,
the linear inviscid damping is true in the finite codimensional center space.
These different behavior indicates that with the addition of Coriolis effects,
the dynamics near the Sinus flow is very rich.

In the work of \cite{Kuo1974} (see Section A of Chapter VII), based on
numerical results, Kuo wrongly claimed that the stability boundary in the
rectangular domain $(\beta,\alpha^{2})\in\left(  -{{\pi^{2}}/{2},{\pi^{2}}%
/{2}}\right)  \times\left(  0,{3\pi^{2}}/{4}\right)  $ is given by the curve
$\Gamma_{2}\ $of SNMs, that is, the instability domain in \cite{Kuo1974}
consists of regions I, II, and IV. See (b) in Figure 6 of \cite{Kuo1974}. The
same stability picture can be also found in \cite{Pedlosky1987}. The reason of
incorrectness using the SNM curve $\Gamma_{2}\ $as the stability boundary
can be seen in Remarks \ref{Kuo1} and \ref{Kuo2}. Our results in Theorem
\ref{transition} correct the stability picture. More precisely, the stability
boundary in the rectangular domain is $\sqrt{\Lambda_{\beta}}$ with
$\beta\in( -{{\pi^{2}}/{2},{\pi^{2}}/{2}})$, and regions II and IV actually
lie in the stability domain. The stability boundary $\Gamma_{1}$ for
$\beta\in\left(  -{\pi^{2}}/{2},\beta_{-}\right)  $ was not detected in
\cite{Kuo1974}. Moreover, two of the curves of stability boundary in our
results, the right boundary $\Gamma_{1}\ $of region IV and upper boundary
$\Gamma_{3}\ $of region II, are not SNM curves. Instead, they consist of
non-resonant neutral modes with $c>1$ or $c<0$.

To confirm our theoretical results in Theorem \ref{transition}, we run the
numerical simulations with more accuracy for $\beta\in\left(  {\frac{{\sqrt
{3}-1}}{{4}}}\pi^{2},{\frac{1}{{2}}}{\pi^{2}}\right)  $. We find that the
difference between the $\alpha$ values in the stability boundary
$\sqrt{\Lambda_{\beta}}$ and those in the SNM curve $(\ref{37})$ is actually
very small. More precisely, as shown in the following Table \ref{tab:table},
for a fixed $\beta$, the difference between $\sqrt{\Lambda_{\beta}}$ and
$\pi\sqrt{1-\gamma^{2}}$ is as small as $10^{-5}$ to $10^{-3}$, and the phase
speed $c^{*}\in(-\infty,0)$ such that $\Lambda_{\beta}=\lambda_{\beta}%
^{-}(c^{*})$ is as small as $10^{-3}$. Such small difference partly explained
why the true stability boundary was not found by the numerical results in
\cite{Kuo1974}.

\begin{longtable}[c]{ccccc}
\caption{Difference between $\sqrt{\Lambda_{\beta}}$ and $\pi\sqrt{1-\gamma%
^{2}}$} \\
\hline \hline
$\beta$ & $\sqrt{\Lambda_{\beta}}$ & $\pi\sqrt{1-\gamma^{2}}\ $ & difference & $c^*$ \\\hline
1.80626 & 1.57080 & 1.57080 & 0 & 0\\
2.60650 & 1.90050 & 1.90050 & 0.000004894 & -0.00003\\
2.85444 & 1.99395 & 1.99394 & 0.000014579 & -0.00006\\
3.05645 & 2.06795 & 2.06792 & 0.000029048 & -0.00009\\
3.24603 & 2.13593 & 2.13588 & 0.000049360 & -0.00012\\
3.44449 & 2.20585 & 2.20577 & 0.000078511 & -0.00015\\
3.69853 & 2.29388 & 2.29376 & 0.000126720 & -0.00018\\
4.18261 & 2.45904 & 2.45882 & 0.000222321 & -0.00018\\
4.37126 & 2.52328 & 2.52304 & 0.000233368 & -0.00015\\
4.49531 & 2.56575 & 2.56554 & 0.000219151 & -0.00012\\
4.59739 & 2.60097 & 2.60078 & 0.000188895 & -0.00009\\
4.69034 & 2.63332 & 2.63318 & 0.000144032 & -0.00006\\
4.78396 & 2.66631 & 2.66623 & 0.000083277 & -0.00003\\
4.93480 & 2.72070 & 2.7207 & 0 & 0\\\hline\hline
\label{tab:table}
\end{longtable}


\bigskip

\noindent\textbf{Acknowledgements}\medskip

J. Yang and H. Zhu would like to thank School of Mathematics at Georgia
Institute of Technology for the hospitality during their visits. H. Zhu
sincerely thanks Profs. X. Hu and Y. Shi for their continuous encouragement.
Z. Lin is supported in part by NSF grants DMS-1411803 and DMS-1715201. J. Yang
is supported in part by NSF grant DMS-1411803. H. Zhu is supported in part by
NSFC (No. 11425105), PITSP (No. BX20180151), CPSF (No. 2018M630266) and
Shandong University Overseas Program.


\begin{thebibliography}{99}                                                                                               %


\bibitem {Bailey1993}P. B. Bailey, W. N. Everitt, J. Weidmann, A. Zettl,
\textit{Regular approximations of singular Sturm-Liouville problems}, Results
Math., 23 (1993), pp. 3--22.

\bibitem {Balmforth-Piccolo2001}N. J. Balmforth, C. Piccolo, \textit{The onset
of meandering in a barotropic jet}, J. Fluid Mech., 449 (2001), pp. 85--114.

\bibitem {Bateman1953}H. Bateman, \textit{Higher transcendental functions},
McGraw-Hill Book Company, Inc, 1953.

\bibitem {Burns-Maslowe-Brown2002}A. G. Burns, S. A. Maslowe, S. N. Brown,
\textit{Barotropic instability of the Bickley jet at high Reynolds numbers},
Stud. Appl. Math., 109 (2002), pp. 279--296.

\bibitem {CR71}M. Crandall, P. Rabinowitz, \textit{Bifurcation from simple
eigenvalues}, J. Funct. Anal., 8 (1971), pp. 321--340.

\bibitem {CFKS}H. L. Cycon, R. G. Froese, W. Kirsch, B. Simon,
\textit{Schr\"{o}dinger operators with application to quantum mechanics and
global geometry}. Texts and Monographs in Physics. Springer-Verlag, Berlin, 1987.

\bibitem {Dickinson-Clare1973}R. E. Dickinson, F. J. Clare, \textit{Numerical
study of the unstable modes of a hyperbolic-tangent barotropic shear flow}, J.
Atmos. Sci., 30 (1973), pp. 1035--1049.

\bibitem {Drazin-Beaumont-Coaker1982}P. G. Drazin, D. N. Beaumont, S. A.
Coaker, \textit{On Rossby waves modified by basic shear, and barotropic
instability}, J. Fluid Mech., 124 (1982), pp. 439--456.

\bibitem {Drazin-Howard1966}P. G. Drazin, L. N. Howard, \textit{Hydrodynamic
stability of parallel flow of inviscid fluid}, Adv. Appl. Mech., 9 (1966), pp. 1--89.

\bibitem {Drazin-Reid1981}P. G. Drazin, W. H. Reid, \textit{Hydrodynamic
stability}, Cambridge Monogr. Mech. Appl. Math., Cambridge University Press,
Cambridge, UK, 1981.

\bibitem {Engevik2004}L. Engevik, \textit{A note on the barotropic instability
of the Bickley jet}, J. Fluid Mech., 499 (2004), pp. 315--326.

\bibitem {Howard1961}L. N. Howard, \textit{Note on a paper of John W. Miles},
J. Fluid Mech., 10 (1961), pp. 509--512.

\bibitem {Howard1964}L. N. Howard, \textit{The number of unstable modes in
hydrodynamic stability problems}, J. M\'{e}canique, 3 (1964), pp. 433--443.

\bibitem {Howard-Drazin1964}L. N. Howard, P. G. Drazin, \textit{On instability
of a parallel flow of inviscid fluid in a rotating system with variable
Coriolis parameter}, J. Math. Phys., 43 (1964), pp. 83--99.

\bibitem {Kato1980}T. Kato, \textit{Perturbation theory for linear operators},
second edition, Springer-Verlag, Heidelberg, 1980.

\bibitem {Kong1999}Q. Kong, H. Wu, A. Zettl, \textit{Dependence of the $n$-th
Sturm-Liouville eigenvalue on the problem}, J. Differential Equations, 156
(1999), pp. 328--354.

\bibitem {Kuo1949}H. L. Kuo, \textit{Dynamic instability of two-dimensional
non-divergent flow in a barotropic atmosphere}, J. Meteor., 6 (1949), pp. 105--122.

\bibitem {Kuo1974}H. L. Kuo, \textit{Dynamics of quasi-geostrophic flows and
instability theory}, Adv. Appl. Mech., 13 (1974), pp. 247--330.

\bibitem {CLin1955}C. C. Lin, \textit{The theory of hydrodynamic stability},
Cambridge University Press, Cambridge, UK, 1955.

\bibitem {Lin2003}Z. Lin, \textit{Instability of some ideal plane flows}, SIAM
J. Math. Anal., 35 (2003), pp. 318--356.

\bibitem {Lin2005}Z. Lin, \textit{Some recent results on instability of ideal
plane flows}, Nonlinear partial differential equations and related analysis,
217--229, Contemp. Math., 371, Amer. Math. Soc., Providence, RI, 2005.

\bibitem {lin-xu-2017}Z. Lin, M. Xu, \textit{Metastability of Kolmogorov flows
and inviscid damping of shear flows}, Arch. Ration. Mech. Anal.,    231  (2019), pp. 1811--1852.

\bibitem {lin-zeng-couette}Z. Lin, C. Zeng, \textit{Inviscid dynamic
structures near Couette flow}, Arch. Ration. Mech. Anal., 200 (2011), pp. 1075--1097.

\bibitem {Lin2017}Z. Lin, C. Zeng, \textit{Instability, index theorem, and
exponential trichotomy for Linear Hamiltonian PDEs}, to appear in Mem. Amer.
Math. Soc..

\bibitem {lindzen1988}R. S. Lindzen, \textit{Instability of plane parallel
shear flow (toward a mechanistic picture of how it works)}, Pure Appl.
Geophys., 126 (1988), pp. 103--121.

\bibitem {lindzen1978}R. S. Lindzen, K. K. Tung, \textit{Wave overreflection
and shear instability}, J. Atmos. Sci., 35 (1978), pp. 1626--1632.

\bibitem {Lipps1962}F. B. Lipps, \textit{The barotropic stability of the mean
winds in the atmosphere}, J. Fluid Mech., 12 (1962), pp. 397--407.

\bibitem {Maslowe1991}S. A. Maslowe, \textit{Barotropic instability of the
Bickley jet}, J. Fluid Mech., 229 (1991), pp. 417--426.

\bibitem {Miles1964}J. W. Miles, \textit{Baroclinic instability of the zonal
wind}, Rev. Geophys., 2 (1964), pp. 155--176.

\bibitem {Pedlosky1963}J. Pedlosky, \textit{Baroclinic instability in two
layer systems}, Tellus, 15 (1963), pp. 20--25.

\bibitem {Pedlosky1964}J. Pedlosky, \textit{The stability of currents in the
atmosphere and the ocean}, Part I. J. Atmos. Sci., 21 (1964), pp. 201--219.

\bibitem {Pedlosky1987}J. Pedlosky, \textit{Geophysical fluid dynamics}, 2nd
edn. Springer, New York (1987).

\bibitem {Reed1972}M. Reed, B. Simon, \textit{Methods of modern nathematical
physics}, vol. I, Academic Press, New York, 1972.

\bibitem {SHS1993}T. H. Solomon, W. J. Holloway, H. L. Swinney, \textit{Shear
flow instabilities and Rossby waves in barotropic flow in a rotating annulus},
Phys. Fluids A, 5 (1993), pp. 1971--1982.

\bibitem {Tollmien1935}W. Tollmien, \textit{Ein Allgemeines Kriterium der
Instabitit\"{a}t laminarer Geschwindigkeitsverteilungen}, Nachr. Ges. Wiss.
G\"{o}ttingen Math. Phys., 50 (1935), pp. 79--114.

\bibitem {Tung1981}K. K. Tung, \textit{Barotropic instability of zonal flows},
J. Atmos. Sci., 38 (1981), pp. 308--321.

\bibitem {Rayleigh1880}J. W. S. Rayleigh, \textit{On the stability, or
instability, of certain fluid motions}, Proc. London Math. Soc., 9 (1880), pp. 57--70.

\bibitem {Rossby1939}C. G. Rossby, \textit{Relation between variations in the
intensity of the zonal circulation of the atmosphere and the displacements of
the semi-permanent centers of action}, J. Mar. Res., 2 (1939), pp. 38--55.

\bibitem {Wei-Zhang-Zhao2017}D. Wei, Z. Zhang, W. Zhao, \textit{Linear
Inviscid damping for a class of monotone shear flow in Sobolev spaces}, Comm.
Pure. Appl. Math., 71 (2018), 617--687.

\bibitem {Wei-Zhang-Zhao1704}D. Wei, Z. Zhang, W. Zhao, \textit{Linear
inviscid damping and vorticity depletion for shear flows}, preprint (2017).
Available at arXiv:1704.00428.

\bibitem {Wei-Zhang-Zhu2018}D. Wei, Z. Zhang, H. Zhu, \textit{Linear inviscid
damping for the $\beta$-plane equation}, preprint (2018). Available at arXiv:1809.03065.

\bibitem {Weidmann1980}J. Weidmann, \textit{Linear operators in Hilbert
spaces}, Springer-Verlag, New York, 1980.

\bibitem {Weidmann1987}J. Weidmann, \textit{Spectral theory of ordinary
differential operators}, Lecture Notes in Mathematics, Vol. 1258,
Springer-Verlag, Berlin, 1987.

\bibitem {Zettl2005}A. Zettl, \textit{Sturm-Liouville theory}, Mathematical
Surveys Monographs, vol. 121, Amer. Math. Soc., 2005.

\bibitem {zillinger}C. Zillinger, \textit{Linear inviscid damping for monotone
shear fows in a finite periodic channel, boundary effects, blow-up and
critical Sobolev regularity}, Arch. Ration. Mech. Anal., 221 (2016), pp. 1449--1509.
\end{thebibliography}
\end{document}